\documentclass[12pt,twoside]{article}
\usepackage{}
\usepackage{amsfonts}
  \usepackage{url}
  \usepackage{mathrsfs}
  \usepackage{amssymb,amsmath}
     \usepackage{graphicx}

   \usepackage[a4paper, margin=2cm]{geometry}

  \input xypic
  \xyoption{all}

 \usepackage{tikz}
\usetikzlibrary{arrows}

 \usepackage{times}

  \usepackage{amsmath, amsthm, amsfonts}

\newtheorem{theorem}{Theorem}[section]

\newtheorem{definition}[theorem]{Definition}

\newtheorem{lemma}[theorem]{Lemma}

\newtheorem{remark}[theorem]{Remark}

\def \fkz{\mathfrak{z}}

\def \R{\mathfrak{R}}

\def \fc {\mathfrak{c}}

\def \half{\frac{1}{2}}

\def \p{\partial}

              \def \M{{\mathcal M}}

              \def \R{{\bf R}}

              \def \E{{\mathcal E}}

\def \p{\partial}
\def \fs {\mathbf{s}}
\def \ft {\mathbf{t}}
\def \E{{\mathcal E}}

\def \v {\vskip 0.1in}
\def \n {\noindent}

\begin{document}

  \begin{center}
   {\LARGE \bf Virtual Neighborhood Technique for \\Holomorphic Curve Moduli Spaces  }
\\\;\;\;\;
   {\large An-Min Li and Li Sheng}\footnote{partially  supported by a NSFC grant}
   \footnote{anminliscu@126.com, lshengscu@gmail.com}

{Department of Mathematics, Sichuan University
        Chengdu, PRC}
 \end{center}

\v\v
\begin{abstract}
In this paper we use the approach of Ruan (\cite{R2}) and Li-Ruan (\cite{LR}) to construct virtual neighborhoods and show that the Gromov-Witten invariants can be defined as an integral over top strata
of virtual neighborhood. We prove that the invariants defined in this way satisfy all the Gromov-Witten axioms of Kontsevich and Manin.
\end{abstract}

\section{\bf Introduction }

Ruan and Tian established the theory of Gromov-Witten invariants for {\em semi-positive} symplectic manifold about 90's : Ruan \cite{R1} first introduced a new invariant to the symplectic manifold by counting $J$-holomorphic maps from $S^2$ to a fixed {\em semi-positive} symplectic manifold. Nowdays these invariants
are called Gromov-Witten invariants. Later Ruan and Tian studied the higher genus case and proved the associativity for the quantum cup product and the WDVV equations in 1997.
\v
In the effort to remove semi-positive condition, the technology went on a significant change.  There had been several different approaches to define Gromov-Witten Invariants for general symplectic manifolds, such as Fukaya-Ono \cite{FO}, Li-Tian \cite{LT}, Liu-Tian \cite{LiuT},
Ruan \cite{R2}, Siebert \cite{S} and etc.
\v
Recently, there is a great deal of interest among symplectic geometric community to re-visit the latter approach with the purpose to clean up some of issues ( see \cite{C15,C16},\cite{CLW,CLW1},\cite{FOOO12}-\cite{FOOO17},\cite{MW12}-\cite{MW15-3},\cite{TF17}). The main complication is that the moduli space has various lower strata. How to deal with these lower strata is one
of main issues discussed recently. Our idea is that if we can show that the relevant differential form decays in certain rate near lower strata, the Gromov-Witten invariants can be defined as an integral over top strata
of virtual neighborhood. Therefore, all the complication of lower strata of the of virtual neighborhood can be avoided entirely.
\v
In this paper we use the approach of Ruan (\cite{R2}) and Li-Ruan (\cite{LR}). Let us describe the main idea.

\v
\v

\v

\subsection{Local regularization }\label{Local regularization}

\v

We explain the construction of local regularization  for the top strata $\mathcal{M}_{g,n}(A)$. For details and for lower strata please see the section \S\ref{local_regularization} and \S\ref{local_regularization-lower strata}. Consider the universal curve over the Teichm\"uller space (for detail see section \S\ref{sub_sect_Teich})
$$\pi_{\mathbf{T}}:   \mathcal{Q}\to \mathbf{T}_{g,n}.$$
We assume that $n>2-2g$, and $(g,n)\ne (1,1), (2,0)$.
For any $[b_o]=[(p_o,u)]\in \mathcal{M}_{g,n}(A)$ let $\gamma_o\in \mathbf{T}_{g,n}$ such that $\pi_{\mathcal M}(\gamma_o)=[p_o]$, where $\pi_{\mathcal M}: \mathbf{T}_{g,n}\to \mathcal{M}_{g,n}$ is the projection. We
choose a local slice for $Q$, which gives a local coordinate chart on $U\subset \mathbf{T}_{g,n}$ and a local trivialization on $\pi_{\mathbf{T}}^{-1}(U)$:
\begin{equation}\label{local coordinates}
\psi: U\rightarrow \mathbf{A},\;\;\;\Psi:\pi_{\mathbf{T}}^{-1}(U)\rightarrow \mathbf{A}\times \Sigma,\end{equation}
with $\psi(\gamma_o)=a_o$, where $U\subset \mathbf{T}_{g,n}$ is a open set.
We have a continuous family of Fredholm system
$$\left(\widetilde{\mathcal B}(a),\; \widetilde{\mathcal{E}}(a),\;
\bar{\p}_{j,J}\right)$$
parameterized by $a\in \mathbf{A}$. Denote by $j_a$ the complex structure on $\Sigma$ associated with $a=(j,{\bf y})$ and put $j_{a_o}:=j_o$.  For any $v\in \widetilde{\mathcal B}(a)$ let $b=(a,v)$ and denote $\widetilde{\E}(a)|_{v}:=\widetilde{\E}|_{b}.$ Let $b_o=(a_o,u)$, denote by $G_{b_{o}}$ the isotropy group at $b_o$. We can choose a $G_{b_o}$-invariant finite dimensional subspace
$\widetilde{K}_{b_o} \subset \widetilde{\E}|_{b_o}$ such that every member of $\widetilde{K}_{b_o}$ is in $C^{\infty}(\Sigma,u^{*}TM\otimes\wedge^{0,1}_{j_o}T^{*}\Sigma)$ and
\begin{equation}\label{regularization_operator}
\widetilde{K}_{b_o} + image D_{b_o} = \widetilde{\E}|_{b_o},
\end{equation}
where $D_{b_o}=D\bar{\p}_{j_o,J}$ is the vertical differential of $\bar{\p}_{j_o,J}$ at $u$.
\v
The Weil-Petersson metric $g_{\mathsf{wp}}$ on $\mathbf{T}_{g,n}$ induces a $Diff^+(\Sigma)$-invariant distance $d_{\mathbf{A}}(a_o,a)$ on $\mathbf{A}$.
Set
$$
\widetilde{\mathbf{O}}_{b_{o}}(\delta,\rho):=\{(a,v)\in \mathbf{A}\times \widetilde{\mathcal{B}} \;|\; d_{\mathbf{A}}(a_o,a)<\delta, \|h\|_{j_a,k,2}<\rho \},
$$
$$\mathbf{O}_{[b_{o}]}(\delta,\rho)=
\widetilde{\mathbf{O}}_{b_{o}}(\delta,\rho)/G_{b_o},$$
where $h\in W_j^{k,2}(\Sigma;u^{\ast}TM)$, $v=\exp_{u}(h)$. Note that both $d_{\mathbf{A}}$ and $\|h\|_{j,k,2}$ are $Diff^+(\Sigma)$-invariant, we may identified $\mathbf{O}_{[b_{o}]}(\delta,\rho)$ with a neighborhood of $[b_o]\in \mathcal{M}_{g,n}(A)$. We can choose $\delta$, $\rho$ so small such that there is an isomorphism
$$P_{b_{o}, b}:\widetilde{\E}_{b_o}\rightarrow \widetilde{\E}_{b}\;\;\;\forall \;b\in \widetilde{\mathbf{O}}_{b_{o}}(\delta,\rho).$$
Now we define a thickned Fredholm system $(\widetilde{K}_{b_o}\times \widetilde{\mathbf{O}}_{b_{o}}(\delta,\rho), \widetilde{K}_{b_o}\times \widetilde{\E}|_{\widetilde{\mathbf{O}}_{b_{o}}(\delta,\rho)}, S)$.
Let $(\kappa, b)\in \widetilde{K}_{b_o}\times \widetilde{\mathbf{O}}_{b_{o}}(\delta,\rho)$, define
\begin{equation}\label{local regu}
S(\kappa,b) = \bar{\partial}_{j_a,J}u + P_{b_o,b}\kappa.
\end{equation}
We can choose $(\delta, \rho)$ small such that
the linearized operator $DS_{(\kappa,b)}$ is surjective for any $b\in \widetilde{\mathbf{O}}_{b_{o}}(\delta,\rho)$.
\v
\subsection{\bf Global regularization and virtual neighborhoods}\label{global_r}

There exist finite points $[b_i]\in \overline{\mathcal{M}}_{g,n}(A)$, $1\leq i \leq \mathfrak{m}$, such that
\begin{itemize}
\item[(1)] The collection $\{\mathbf{O}_{[b_i]}(\delta_i/3,\rho_i/3)
\mid 1\leq i \leq \mathfrak{m}\}$ is an open cover of $ \overline{\mathcal{M}}_{g,n}(A)$.
\item[(2)] Suppose that $\widetilde{\mathbf{O}}_{b_i}(\delta_i,\rho_i)
\cap \widetilde{\mathbf{O}}_{b_j}(\delta_j,\rho_j)
\neq\phi$. For any $b\in \widetilde{\mathbf{O}}_{b_i}(\delta_i,\rho_i)
\cap \widetilde{\mathbf{O}}_{b_j}(\delta_j,\rho_j)$, $G_b$ can be imbedded into both $G_{b_i}$ and $G_{b_j}$ as subgroups.
\end{itemize}
Set
$$\mathcal{U}=\bigcup_{i=1}^{\mathfrak{m}}
\mathbf{O}_{[b_{i}]}(\delta_i,\rho_i).$$
There is a forget map
$$forg: \mathcal{U}\to \overline{\mathcal{M}}_{g,n},\;\;\;[(j,{\bf y}, u)]\longmapsto [(j,{\bf y})].$$

\v\n
In Section \S\ref{global_r} we construct a finite rank orbi-bundle $\mathbf{F}$ over
$\mathcal{U}$ such that, for every $i\leq \mathfrak{m}$,
$\widetilde{\mathbf{F}}\mid_{b_i}$ contains a copy of group ring
$\mathbb{R}[G_{b_i}].$
The construction imitates SiebertÂ¡Â¯s construction.
\v
Then we construct a bundle map $\mathfrak{i}([\kappa,b]):\mathbf{F}\to \E$ and
define a global regularization to be the bundle map $\mathcal{S}:\mathbf{F}\to \E$
$$\mathcal{S}([\kappa,b])
=[\bar{\partial}_{j,J}v] + \mathfrak{i}([\kappa,b])
$$
such that $D\mathcal{S}$ is surjective.
Denote
$$\mathbf{U}=\mathcal{S}^{-1}(0)|_{\mathcal{U}}.$$
By restricting the bundle $\mathbf{F}$ to $\mathbf{U}$ we have a bundle $\mathbf{E}$ of finite rank with a canonical section $\sigma$. We call
$(\mathbf{U},\mathbf{E},\sigma)$  a virtual neighborhood for $\overline{\mathcal{M}}_{g,n}(A)$. Denote by $\mathbf{U}^T$ the top strata of $\mathbf{U}$.
In Section \S\ref{top strata}
 we prove
\begin{theorem} \label{Smooth}
$\mathbf{U}^T$ is a smooth oriented, effective orbifold of dimension $\mathcal{N}= rank(\mathbf{F}) + ind \;D\mathcal{S}.$
\end{theorem}

\subsection{\bf Gromov-Witten invariants}\label{GW inv.}

Recall that we have a natural evaluation map
$$ev_i: \mathbf{U}^T \longrightarrow
M\;\;\;\;\left(\Sigma,j,{\bf y},(\kappa, u)\right)\longmapsto
u(y_{i}) $$ for $i\leq n$ defined by evaluating at marked points. We have another map
$$\mathscr{P}: \mathbf{U}^T \longrightarrow \mathcal{M}_{g,n}\;\;\;\;\left(\Sigma,j,{\bf y},(\kappa, u)\right)\longmapsto (\Sigma,j,{\bf y}).$$
Choose a smooth metric $\mathbf{h}$ on the bundle $\mathbf{E}$. Using $\mathbf{h}$ we construct a Thom form $\Theta$ supported in a small $\varepsilon$-ball of the $0$-section of $\mathbf{E}$. The Gromov-Witten invariants are defined as
\begin{equation}\label{integral-0}
\Psi_{A,g,n}(K;\alpha_1,...,
\alpha_{n})=\int_{\mathbf{U}^T}\mathscr{P}^*(K)\wedge\prod^n_{j=1} ev^*_j\alpha_j\wedge \sigma^*\Theta
\end{equation}
 for $\alpha_i\in H^*(M, {\mathbb{R}})$ represented by differential form and $K$ represented by a good differential form defined on $\mathcal{M}_{g,n}$ in Mumford's sense.

\v
In \cite{LS-1} and \cite{LS-2} we proved the exponential decay of the derivatives of the gluing maps with respect to the gluing parameter near lower strata. Using these estimates we prove in \S\ref{gromov-witten} and \S\ref{properties of gromov-witten}

\begin{theorem}\label{Conver}
The integral \eqref{integral-0} is convergent.
\end{theorem}

\begin{theorem}\label{Common Properties}
 \begin{itemize}
\item[(1).] $\Psi_{A,g,n}(K;\alpha_1,...,
\alpha_{n})$ is well-defined, multi-linear and skew symmetry.
\item[(2).] $\Psi_{A,g,n}(K;\alpha_1,...,
\alpha_{n})$ is independent of the choices of forms $K, \alpha_i$ representing the
cohomology classes $[K], [\alpha_i]$ and is independent of the choice of $\Theta$.
\item[(3).] $\Psi_{A,g,n}(K;\alpha_1,...,
\alpha_{n})$ is independent of the choices of the regularization.
\item[(4).] $\Psi_{A,g,n}(K;\alpha_1,...,
\alpha_{n})$ is independent of J and is a symplectic deformation invariant.
\item[(5).] When M is semi-positive,
    $\Psi_{A,g,n}(K;\alpha_1,...,
\alpha_{n})$ agrees with the definition of \cite{RT2}.

\end{itemize}
\end{theorem}
\v\n
\begin{theorem}\label{Common Properties-1} Suppose that $(g,n)\neq (0,3),(1,1)$. Let
$\pi: \overline{\M}_{g,n}\rightarrow \overline{\M}_{g, n-1}$ be the map by forgetting the last marked point.
\vskip 0.1in
\noindent
(1) For any $\alpha _1, \cdots , \alpha _{n-1}$ in $H^*(M, \R)$,
we have
$$\Psi_{A,g,n}(K;\alpha_1,...,
\alpha_{n-1},1)=\Psi_{A,g,n-1}(\pi_*(K); \alpha _1, \cdots,\alpha _{n-1}),$$
\vskip 0.1in
\noindent
(2)  Let $\alpha_n$ be in $H^2(Y, \R)$, then
$$\Psi_{A,g,n}(\pi^*(K); \alpha _1,
\cdots,\alpha _{n-1}, \alpha_n)=\alpha_n (A)
\Psi_{A,g,n-1}(K; \alpha _1, \cdots,\alpha _{n-1}).$$
\end{theorem}
\v
\begin{theorem}\label{split-3}
 For any $K_1\times K_2 \in H^*(\overline{\mathcal{M}}_{g_1,g_2,n_1,n_2}, \mathbb{R})$,
$\alpha _1,\cdots,\alpha _n \in H^*(M,\mathbb{R})$, represented by smooth forms, we have
$$\Psi_{(A,g,n)}((\theta)_{!}(K_1\times K_2));\{\alpha _i\})$$$$
=\epsilon(K,\alpha) \sum \limits _{A=A_1+A_2} \sum \limits_{a,b}
\Psi_{(A_1,g_1,n_1+1)}(K_1;\{\alpha _{i}\}_{i\le n_1}, \beta _a)
\eta ^{ab}
\Psi_{(A_2,g_2,n_2+1)}(K_2;\beta _b,
\{\alpha _{j}\}_{j>n_1}),
$$
\end{theorem}
where $\epsilon(K,\alpha)=(-1)^{deg(K_2)\sum^{n_1}_{i=1} (deg (\alpha_i))}$.
\v\v
We conclude that the invariants defined in this way satisfy all the Gromov-Witten axioms of Kontsevich and Manin.
\v

\v

{\bf Acknowledgement}:   We would like to thank Yongbin Ruan, Huijun Fan, Jianxun Hu and Bohui Chen for many useful discussions.

\v

\section{Preliminary}
\v
First all, we recall some results on the Deligne-Mumford moduli space $\overline{\mathcal{M}}_{g,n}$ of stable curves, for detail see \cite{Tromba}, \cite{Wolp-1}, \cite{Wolp-2}.
\v
\subsection{\bf Metrics on $\Sigma$}\label{metric on surfaces} Let $(\Sigma,j,\mathbf{y})$ be a smooth Riemann surface of genus $g$ with $n$ marked points. In this paper we assume that $n>2-2g$, and $(g,n)\ne (1,1), (2,0)$. It is well-known that there is a unique complete hyperboloc metric $\mathbf{g}_0$ in $\Sigma\setminus \{{\bf y}\}$ of constant curvature $-1$ of finite volume, in the given conformal class $j$ ( see \cite{Wolp-1}).
Let $\mathbb H=\{\zeta=\lambda+\sqrt{-1}\mu|\mu>0\}$ be the half upper plane with the Poincare metric
$$
\mathbf{g}_0(\zeta)=\frac{1}{(Im(\zeta))^2}d\zeta d\bar\zeta.
$$
 Let
$$
\mathbb{D}=\frac{ \{\zeta \in \mathbb H|  Im (\zeta) \geq 1\}}{\zeta \sim \zeta + 1}
$$
be a cylinder, and $\mathbf{g}_0$ induces a metric on $\mathbb{D}$, which is still denoted  by $\mathbf{g}_0$. Let $z=e^{2\pi i\zeta}$, through which we identify $\mathbb{D}$ with $D(e^{-2\pi}):=\{z||z|< e^{-2\pi}\}$.
An important result is that for any punctured point $y_i$ there exists a
neighborhood $O_i$ of $y_i$ in $\Sigma$ such that
$$
(O_i\setminus \{y_i\},\mathbf{g}_0)\cong (D(e^{-2\pi})\setminus \{0\},\mathbf{g}_0),
$$
moreover, all $O_i$'s are disjoint with each other. Then we can view $D_{y_i}(e^{-2\pi})$ as a neighborhood of $y_i$ in $\Sigma$ and $z$ is a local complex coordinate on $D_{y_i}(e^{-2\pi})$ with $z(y_i)=0$. For any $c>0$ denote $$\mathbf{D}(c)=\bigcup D_{y_i}(c),\;\;\;\Sigma(c)=\Sigma\setminus \mathbf{D}(c).$$ Let $\mathbf{g}'=dzd\bar{z}$ be the standard Euclidean metric on each $D_{y_i}(e^{-2\pi})$. We fix a smooth cut-off function $\chi(|z|)$ to glue $\mathbf{g}_0$ and $\mathbf{g}'$, we get a smooth metric $\mathbf{g}$  in the given conformal class $j$ on $\Sigma$ such that
\[
\mathbf{g}=\left\{
\begin{array}{ll}
\mathbf{g}_0 \;\;\;\;\; on \;\;\Sigma\setminus \mathbf{D}(e^{-2\pi}),    \\  \\
\mathbf{g}'\;\;\;on\; \mathbf{D}(\frac{1}{2}e^{-2\pi})\;\;.
\end{array}
\right.
\]
\v
\v
 Let $\mathbf{g}^{c}=ds^2+d\theta^2$ be the  cylinder metric on each $D_{y_i}^{*}(e^{-2\pi})$, where $z=e^{s+2\pi\sqrt{-1}\theta}$.
We also define another metric $\mathbf{g}^\diamond$ on $\Sigma$ as above by glue $\mathbf{g}_0$ and $\mathbf{g}^c$, such that
\[
\mathbf{g}^{\diamond}=\left\{
\begin{array}{ll}
\mathbf{g}_0 \;\;\;\;\; on \;\;\Sigma\setminus \mathbf{D}(e^{-2\pi}),    \\  \\
\mathbf{g}^{c}\;\;\;on\; \mathbf{D}(\frac{1}{2}e^{-2\pi})\;\;.
\end{array}
\right.
\]

\v
The metric $\mathbf{g}$ (resp. $\mathbf{g}^{\diamond}$) can be generalized to marked nodal surfaces in a natural way. Let $(\Sigma, j,\mathbf{y})$ be a
marked nodal surfaces with $\mathfrak{e}$ nodal points $\mathbf{p}=(p_{1},\cdots,p_{\mathfrak{e}})$. Let $\sigma:\tilde{\Sigma}=\sum_{\nu=1}^{\mathfrak{r}}\Sigma_{\nu}\to \Sigma$ be the normalization. For every node $p_i$ we have a pair $\{\mathbf{a}_i, \mathbf{b}_i\}$. We view $\mathbf{a}_i$, $\mathbf{b}_i$ as marked points on $\tilde{\Sigma}$ and define the metric $\mathbf{g}_{\nu}$ (resp. $\mathbf{g}^{\diamond}_{\nu}$) for each $\Sigma_{\nu}$. Then we define
$$\mathbf{g}:=\bigoplus_{1}^{\nu} \mathbf{g}_{\nu},\;\;\;\;\;\mathbf{g}^{\diamond}:=\bigoplus_{1}^{\nu} \mathbf{g}^{\diamond}_{\nu}.$$

\subsection{Teichm\"uller space}\label{sub_sect_Teich}

Denote by $\mathcal{J}(\Sigma)\subset End(T\Sigma)$ the manifold of all $C^{\infty}$ complex structures on $\Sigma$, let $\mathcal{G}$ denote the manifold of $C^{\infty}$ Riemannian metrics with constant scalar curvature $-1$ on $\Sigma$. Denote by $Diff^+(\Sigma)$ the group of orientation preserving $C^{\infty}$ diffeomorphisms of $\Sigma$, by $Diff^+_{0}(\Sigma)$ the identity component of $Diff^+(\Sigma).$ $Diff^+(\Sigma)$ acts on $\mathcal{J}(\Sigma)$ and $\mathcal{G}$ by
$$(\phi^*J)_x:=(d\phi_x)^{-1}J_{\phi(x)}d\phi_x,\;\;\;
(\phi^*g)(x)(w,v):=g(\phi(x))(d\phi(x)w, d\phi(x)v)$$
for all $\phi\in Diff^+(\Sigma)$, $x\in \Sigma$, $w,v\in T_x\Sigma.$ There is a bijective,  $Diff^+(\Sigma)$-equivariant correspondence between $\mathcal{J}(\Sigma)$ and $\mathcal{G}$:
$$\mathcal{J}(\Sigma)\cong \mathcal{G}.$$
Put
$$\mathbf{P}:=\mathcal J(\Sigma)\times (\Sigma^{n} \setminus \Delta),$$
where $\Delta\subset \Sigma^{n}$ denotes the fat diagonal.
The orbit spaces are
\begin{align*}
\mathcal{M}_{g,n}=\left(\mathcal J(\Sigma)\times (\Sigma^{n} \setminus \Delta)\right)/ Diff^+(\Sigma),\;\;\;\mathbf{T}_{g,n}=\left(\mathcal J(\Sigma)\times (\Sigma^{n} \setminus \Delta)\right)/ Diff^+_{0}(\Sigma).
\end{align*}
$\mathcal{M}_{g,n}$ is called the Deligne-Mumford space, $\mathbf{T}_{g,n}$ is called the Teichm\"uller space. The mapping class group of $\Sigma$ is
$$Mod_{g,n}=Diff(\Sigma)/Diff_0(\Sigma).$$
It is well-known that $Mod_{g,n}$ acts properly discontinuously on $\mathbf{T}_{g,n}$ and $$\mathcal{M}_{g,n}=\mathbf{T}_{g,n}/Mod_{g,n}$$ is a complex orbifold of dimension $N:=3g-3 +n$. Let $\pi_{\mathcal{M}}:\mathbf{T}_{g,n}\to \mathcal{M}_{g,n}$ be the projection.

\v

\v
Consider the principal fiber bundle
$$
Diff^+_0(\Sigma) \to \mathbf{P}   \to \mathbf{T}_{g,n}
$$
and the associated fiber bundle
$$
\pi_{\mathbf{T}}:   \mathcal{Q}:= \mathbf{P}\times_{Diff^+_0(\Sigma)}\Sigma\to \mathbf{T}_{g,n},
$$
which has fibers isomorphic to $\Sigma$ and is equipped with $n$ disjoint sections
$$\mathcal Y_i := \left\{[j,y_1,\dots,y_n,z]\in \mathcal{Q}\,:\,z=y_i\right\},\qquad
i=1,\dots,n.
$$
It is commonly called the universal curve over $\mathbf{T}_{g,n}$. The following result is well-known ( cf \cite{RS} ):
\begin{lemma}\label{slice} Suppose that $n+ 2g \geq 3$. Then for any $\gamma_o=[(j_o,{\bf y}_o)]\in \mathbf{T}_{g,n}$, and any $(j_o,{\bf y}_o)\in \mathbf{P}$ with $\pi_{\mathbf{T}}(j_o,{\bf y}_o)=\gamma_o$ there is an open neighborhood $\mathbf{A}$ of zero in $\mathbb{C}^{3g-3+n}$ and a local holomorphic slice
$\iota=(\iota_{0},\cdots,\iota_{n}):\mathbf{A}\to \mathbf{P} $ such that
\begin{equation}\label{slice-1}
\iota_0(o)=j_o,\qquad \iota_i(o)=y_{io},\qquad i=1,\dots,n,
\end{equation}
and the map
$$
\mathbf{A}\times Diff_{0}(\Sigma) \to \mathbf{P}:(a ,\phi)\mapsto (\phi^*\iota_{0}(a),\phi^{-1}(\iota_1(a)),
\cdots,\phi^{-1}(\iota_{n}(a))
$$
is a diffeomorphism onto a neighborhood of the orbit of $(j_o,{\bf y}_o)$.
\end{lemma}
\v\n

From the local slice we have a local coordinate chart on $U$ and a local trivialization on $\pi_{\mathbf{T}}^{-1}(U)$:
\begin{equation}\label{local coordinates}
\psi: U\rightarrow \mathbf{A},\;\;\;\Psi:\pi_{\mathbf{T}}^{-1}(U)\rightarrow \mathbf{A}\times \Sigma,\end{equation}
 where $U\subset \mathbf{T}_{g,n}$ is a open set. We call $(\psi, \Psi)$ in \eqref{local coordinates} a local coordinate system for $\mathcal{Q}$. Suppose that we have two local coordinate systems
\begin{equation}\label{local coordinates-1}
(\psi, \Psi): (O, \pi_{\mathbf{T}}^{-1}(O))\rightarrow (\mathbf{A}, \mathbf{A}\times \Sigma),\end{equation}
\begin{equation}\label{local coordinates-2}
(\psi', \Psi'): (O', \pi_{\mathbf{T}}^{-1}(O'))\rightarrow (\mathbf{A}', \mathbf{A}'\times \Sigma).\end{equation}
Suppose that $O\bigcap O'\neq \emptyset.$ Let $W$ be a open set with $W\subset O\bigcap O'$.
Denote $V=\psi(W)$ and $V'=\psi'(W)$. Then ( see \cite{RS})
\begin{lemma}\label{local coordinates-3}
$\psi'\circ \psi^{-1}|_{V}:V\to V'$ and $\Psi'\circ \Psi^{-1}|_{V}: {V}\times \Sigma\to {V'}\times \Sigma$ are holomorphic.
\end{lemma}
The diffeomorphism group $ Diff^+(\Sigma)$ acts on $\Sigma^{n}\setminus \Delta$
 by
\begin{equation}\label{group action-1}
\varphi^*(j,y_1,\dots,y_n)
:=(\varphi^*j,\varphi^{-1}(y_1),\dots,\varphi^{-1}(y_n)).
\end{equation}
It is easy to see that $\mathbf{g}$ is $Diff^+(\Sigma)$-invariant.
\v
 Let $\overline{\mathcal{M}}_{g,n}$ be the Deligne-Mumford compactification space, $g_{\mathsf{wp}}$ be the Weil-Petersson metric on $\overline{\mathcal{M}}_{g,n}$. Denote by $\overline{\mathbb{B}}_{g,n}$ the groupoid whose objects are stable marked nodal Riemann surfaces of type $(g,n)$ and whose morphisms are isomorphisms of marked nodal Riemann surfaces. J. Robbin, D. Salamon \cite{RS} used the universal marked nodal family to give an orbifold groupoid structure on $\overline{\mathbb{B}}_{g,n}$. Then $\overline{\mathcal{M}}_{g,n}$ has the structure of a complex orbifold, and $\mathcal{M}_{g,n}$ is an effective orbifold. It is possible that $(g_i,n_i)= (1,1)$ for some smooth component $\Sigma_i$, in this case we consider the reduced effective orbifold structure.

\v

\v
\subsection{The moduli space of stable holomorphic maps}\label{stable holo}
\v
Let $(M,\omega,J)$ be a closed $C^{\infty}$ symplectic manifold of dimension $2m$ with $\omega$-tame almost complex structure $J$, where $\omega$ is a symplectic form. Then there is a Riemannian metric
\begin{equation}\label{definition_of_metrics}
G_J(v,w):=<v,w>_J:=\frac{1}{2}\left(\omega(v,Jw)+\omega(w,Jv)\right)
\end{equation}
for any $v, w\in TM$. Following \cite{MS} we choose the complex linear connection
 $$
 \widetilde {\nabla}_{X}Y=\nabla_{X}Y-\tfrac{1}{2}J\left(\nabla_{X}J\right)Y $$
 induced by the Levi-Civita connection $\nabla$ of the metric $G_{J}.$
\v
Let $(\Sigma, j, {\bf y})$ be a smooth Riemann surface of genus $g$ with $n$ marked points.
Let $u:\Sigma\longrightarrow M$ be a smooth map. For any
section $h\in C^{\infty}(\Sigma;u^{\ast}TM)$ and section $\eta \in
C^\infty(\Sigma, u^{*}TM\otimes \wedge^{0,1}_jT^{*}\Sigma)$ and given integer $k>4$ we define the norms
\begin{eqnarray}\label{norm_k_2}
&&\|h\|_{j,k,2}=
 \left(\int_{\Sigma}\sum_{i=0}^{k} |\nabla^i h|^2
dvol_{\Sigma}\right)^{1/2},\\
\label{norm_2_}
&&\|\eta\|_{j,k-1,2}= \left(\int_{\Sigma}\sum_{i=0}^{k-1} |\nabla^i \eta|^2
dvol_{\Sigma}\right)^{1/2}.
\end{eqnarray}
Here all norms and
covariant derivatives are taken with respect to the  metric $G_J$ on $u^{\ast}TM$ and
the metric $\mathbf{g}$ on $(\Sigma, j, {\bf y})$, $dvol_{\Sigma}$ denotes the volume form with respect to $\mathbf{g}$.
Denote by $W^{k,2}(\Sigma;u^{\ast}TM)$ and
$W^{k-1,2}(\Sigma, u^{*}TM\otimes\wedge^{0,1}_jT^{*}\Sigma)$ the complete spaces with respect to the norms \eqref{norm_k_2} and \eqref{norm_2_} respectively. Denote
 $$
 \widetilde{\mathcal B}=\{u\in W^{k,2}(\Sigma,M)|\;u_{*}([\Sigma])=A\}.
 $$
$\widetilde{\mathcal B}$ is an infinite dimensional Banach manifold. For any $h_1,h_2\in W^{k,2}(\Sigma,u^{*}TM)$ we define
\begin{equation}\label{iner prod.}
\ll h_1, h_2\gg=\int_{\Sigma}G_J(h_1,h_2)dvol_{\Sigma}.
\end{equation}
Then $W^{k,2}(\Sigma;u^{\ast}TM)$ is a Hilbert space,
$\widetilde{\mathcal B}$ is a Hilbert manifold. The map $u$ is called a $(j,J)$-holomorphic map if $du\circ j=J\circ du$. Alternatively
\begin{equation}\label{holo}
\bar\p_{j,J}(u):=\half\left(du + J(u) du\circ j\right)=0.
\end{equation}
Let $\widetilde{\mathcal E}$ be the infinite dimensional Bananch bundle over $\widetilde{\mathcal B}$ whose fiber at $b=(j,{\bf y},u)$ is
$$W^{k-1,2}(\Sigma, u^{*}TM\otimes\wedge^{0,1}_jT^{*}\Sigma).$$
The Cauchy-Riemann operator defines a Fredholm section $\overline{\partial}_{j,J}:\widetilde{\mathcal B}\longrightarrow \widetilde{\mathcal E}.$
\v
The diffeomorphism group $ Diff^+(\Sigma)$ acts on $(\Sigma^{n}\setminus \Delta)\times \widetilde{\mathcal B}$ and $(\Sigma^{n}\setminus \Delta)\times \widetilde{\mathcal E}$ by
\begin{equation}\label{group action-1}
\varphi^*(j,y_1,\dots,y_n,u)
:=(\varphi^*j,\varphi^{-1}(y_1),\dots,\varphi^{-1}(y_n), u\circ \varphi)
\end{equation}
\begin{equation}\label{group action-2}
\varphi^*\kappa= \kappa\cdot d\varphi\;\;\;\forall\; \kappa\in W^{k-1,2}(\Sigma, u^{*}TM\otimes\wedge^{0,1}_jT^{*}\Sigma)
\end{equation}
for $\varphi\in Diff^+(\Sigma)$. Put
$$Aut(j,{\bf y},u)=
\{\phi\in Diff^+(\Sigma)|\varphi^*(j,y_1,\dots,y_ n,u)=(j,y_1,\dots,y_n,u)\}.$$
We call it the automorphism group at $(j,{\bf y},u)$.
\v\v
Our moduli space $\mathcal{M}_{g,n}(A)$ is the quotient space
$$\mathcal{M}_{g,n}(A)=((\Sigma^{n} \setminus \Delta)\times \overline{\partial }_{J}^{-1}(0) )/Diff^+(\Sigma).$$
\v
In order to compactify the moduli space we need to consider the holomorphic maps from marked nodal Riemann surfaces. A {configuration} in $M$ is a tuple $(\Sigma,j,{\bf y},\nu, u)$ where $(\Sigma,j,{\bf y},\nu)$ is a marked nodal Riemann surface of genus $g$ with $n$ distinct marked points,
and $u:\Sigma\to M$
is a smooth map satisfying the nodal conditions
$$
\{p,q\}\in\nu\implies u(p)=u(q).
$$
The configuration $(\Sigma,j,{\bf y},\nu,u)$ is called {holomorphic} if the restriction of $u$ to each smooth component of $\Sigma$
satisfies \eqref{holo}. The configuration $(\Sigma,j,{\bf y},\nu,u)$ is called stable, if the restriction on every smooth component is stable. Let $\overline{\mathcal{M}}_{g,n}(A)$ be the Gromov compactification of $\mathcal{M}_{g,n}(A)$. It is well-known that one can define a topology on $\overline{\mathcal{M}}_{g,n}(A)$ such that  $\overline{\mathcal{M}}_{g,n}(A)$ is a compact Hausdorff space.
\v
Denote by $\overline{\mathcal M}_{g,n}'$   the set of domains of each element in $\overline{\mathcal M}_{g,n}(A)$.  Denote
$$
\overline{\mathcal B}_{g,n}(A)=\left\{ [(j,\mathbf{y},u)]\;|\;u\in W^{k,2}(\Sigma,M),u_{*}([\Sigma])=A,\;\;[(j,\mathbf{y})]\in \overline{\mathcal M}_{g,n}'\right\},
$$
$$
\mathcal B_{g,n}(A)=\left\{ [(j,\mathbf{y},u)]\in \overline{\mathcal B}_{g,n}(A)\;|\; [(j,\mathbf{y})]\in {\mathcal M}_{g,n}\right\}.
$$
\v
For any $[b_o]=[(p_o,u)]\in \mathcal{M}_{g,n}(A)$ with $[p_{o}]\in \mathcal M_{g,n}$
let $\gamma_o=[(j_o,{\bf y}_o)]\in \mathbf{T}_{g,n}$, $(j_o,{\bf y}_o)\in \mathbf{P}$ with $\pi_{\mathcal M}(\gamma_o)=[p_{o}]$ and $\pi_{\mathbf{T}}(j_o,{\bf y}_o)=\gamma_o$. Choose a local coordinate system $(\psi,\Psi)$ on $U$ with $\psi(\gamma_o)=a_o$ for $\mathcal{Q}$ as in \eqref{local coordinates}, we have a local coordinate chart on $U$ and a local trivialization on $\pi_{\mathbf{T}}^{-1}(U)$:
\begin{equation}\label{local coordinates-1}
\psi: U\rightarrow \mathbf{A},\;\;\;\Psi:\pi_{\mathbf{T}}^{-1}(U)\rightarrow \mathbf{A}\times \Sigma,\end{equation}
 where $U\subset \mathbf{T}_{g,n}$ is an open set. We can view $a=(j,{\bf y})$ as parameters, and the domain $\Sigma$ is a fixed smooth surface. Denote by $j_a$ the complex structure on $\Sigma$ associated with $a=(j,{\bf y})$ and put $j_{a_o}:=j_o$. The Weil-Pertersson  metric induces a $Diff^+(\Sigma)$-invariant distance $d_{\mathbf{A}}(a_o,a)$ on $\mathbf{A}$ such that $d^{2}_\mathbf{A}(a):=d^{2}_\mathbf{A}(a_o,a)$ is a smooth function on $\mathbf{A}$. Denote by $\mathbf G_{a}$ the isotropy group at $a$, that is
  $$
  \mathbf G_{a}=\{\phi\in Diff^{+}(\Sigma)\;|\; \phi^{*}(j,\mathbf y)=(j,\mathbf y)\}.
  $$
  Since $\mathcal{M}_{g,n}$ is an effective orbifold, we can choose $\delta$ small such that $\mathbf{G}_{a}$ can be imbedded into $\mathbf{G}_{a_o}$ as a subgroup for any $a$ with $d_{\mathbf{A}}(a_o,a)<\delta$. Denote by $im(\mathbf{G}_{a})$ the imbedding.
\v
Let $b_{o}=(a_{o},u)=(j_o,{\bf y}_o,u)$ be the expression of $[(\gamma_o, u)]$ in this local coordinates.
Set
$$
\widetilde{\mathbf{O}}_{b_{o}}(\delta,\rho):=\{(a,v)\in \mathbf{A}\times \widetilde{\mathcal{B}} \;|\; d_{\mathbf{A}}(a_o,a)<\delta, \|h\|_{j_a,k,2}<\rho \},
$$
$$\mathbf{O}_{[b_{o}]}(\delta,\rho)=
\widetilde{\mathbf{O}}_{b_{o}}(\delta,\rho)/G_{b_o},$$
where $v=\exp_{u}(h),$ $G_{b_o}$ is the isotropy group at $b_o$, that is
$$G_{b_{o}}=\{\phi\in Diff^{+}(\Sigma)\;|\; \phi^{*}(j,\mathbf y,u)=(j,\mathbf y,u)\}.$$
 Obviously, $G_{b_o}$ is a subgroup of $\mathbf{G}_{a_o}$. Note that both $d_{\mathbf{A}}$ and $\|h\|_{j_a,k,2}$ are $Diff^+(\Sigma)$-invariant, we may identified $\mathbf{O}_{[b_{o}]}(\delta,\rho)$ with a neighborhood of $[b_o]\in \mathcal{M}_{g,n}(A)$ in $\mathcal{B}_{g,n}(A)$.

\v\v
\section{\bf Gluing}\label{gluing}

\v
Let $(\Sigma, j,{\bf y}, q)$ be a marked nodal Riemann surface of genus $g$ with $n$ marked points ${\bf y}=(y_1,...,y_n)$ and one nodal point $q$.  We write the marked nodal Riemann surface as
$$
\left(\Sigma=\Sigma_{1}\wedge\Sigma_{2},j=(j_{1},j_{2}),{\bf y}=({\bf y}_{1},{\bf y}_{2}), q=(p_1,p_2)\right),$$
where $(\Sigma_{i},j_{i},{\bf y}_{i}, q_i)$, $i=1,2$, are smooth Riemann surfaces. We say that $q_1,q_2$ are paired to form $q$. Assume that $(\Sigma_{i},j_{i},{\bf y}_{i},q_{i})$ is stable, i.e., $n_{i}+2g_{i}+1\geq 3$, $i=1,2$. We choose metric $\mathbf{g}_i$ on each $\Sigma_i$ as in \S\ref{metric on surfaces}.
Let $z_i$ be the cusp coordinates around $q_i$, $z_i(q_i)=0$, $i=1,2$.
 Let
$$
z_1=e^{-s_1 - 2\pi \sqrt{-1} t_1},\;\;\;z_2=e^{s_2 +2\pi\sqrt{-1}  t_2}.
$$
$(s_i, t_i)$ are called the holomorphic cusp cylindrical coordinates near $q_{i}$.  In terms of the holomorphic cusp cylindrical coordinates we write
$$\stackrel{\circ}{\Sigma}_{1}:=\Sigma_1\setminus\{q_1\}\cong\Sigma_{10}\cup\{[0,\infty)\times S^1\},\;\;\;\stackrel{\circ}{\Sigma}_{2}:=\Sigma_2\setminus\{q_2\}\cong\Sigma_{20}\cup\{(-\infty,0]\times S^1\}.$$
Here $\Sigma_{i0}\subset \Sigma_i$, $i=1,2$, are compact surfaces with boundary. Put $\stackrel{\circ}{\Sigma}= \Sigma\setminus\{q_1,q_2\} =\stackrel{\circ}{\Sigma}_{1}\cup\stackrel{\circ}{\Sigma}_{2}$. We introduce the notations
$$
\Sigma_i(R_0)=\Sigma_{i0}\cup \{(s_i,t_i)|\;|s_i| \leq R_0\},\;\; \;\;\;\;\;\Sigma(R_0)=\Sigma_1(R_0)\cup \Sigma_2(R_0).$$
For any gluing parameter $(r,\tau)$ with $r\geq R_{0}$ and $\tau\in S^1$ we construct a surface $\Sigma_{(r)}$
with the gluing formulas:
\begin{equation}\label{gluing surface}
s_1=s_2 + 2r,\;\;\;t_1=t_2 + \tau.
\end{equation}
where we use $(r)$ to denote gluing parameters.
\v
We will use the cusp cylinder coordinates to describe the construction of $u_{(r)}:\Sigma_{(r)}\to M$. Write
$$u=(u_1,u_2),\;\;u_{i}:\Sigma_i\rightarrow M\; with \;\;u_{1}(q)=u_{2}(q).$$
We choose local normal coordinates $(x^{1},\cdots,x^{2m})$ in a neighborhood  $O_{u(q)}$ of $u(q)$ and choose $R_0$ so large that $u(\{|s_i|\geq \frac{r}{2}\})$ lie in $O_{u(q)}$ for any $r>R_0$. We glue the map $(u_1,u_2)$ to get a pregluing maps $u_{(r)}$ as follows. Set\\
\[
u_{(r)}=\left\{
\begin{array}{ll}
u_1 \;\;\;\;\; on \;\;\Sigma_{10}\bigcup\{(s_1,\theta_1)|0\leq s_1 \leq
\frac{r}{2}, \theta_1 \in S^1 \}    \\  \\
u_1(q)=u_2(q) \;\;
on \;\{(s_1,\theta_1)| \frac{3r}{4}\leq s_1 \leq
\frac{5r}{4}, \theta_1 \in S^1 \}  \\   \\
u_2 \;\;\;\;\; on \;\;\Sigma_{20}\bigcup\{(s_2,\theta_2)|0\geq s_2
\geq - \frac{r}{2}, \theta_2 \in S^1 \}     \\
\end{array}.
\right.
\]
To define the map $u_{(r)}$ in the remaining part we fix a smooth cutoff
function $\beta : {\mathbb{R}}\rightarrow [0,1]$ such that
\begin{equation}\label{def_beta}
\beta (s)=\left\{
\begin{array}{ll}
1 & if\;\; s \geq 1 \\
0 & if\;\; s \leq 0
\end{array}
\right.
\end{equation}
and $\sqrt{1-\beta^2}$ is a smooth function,  $0\leq \beta^{\prime}(s)\leq 4$ and $\beta^2(\frac{1}{2})=\frac{1}{2}.$
 We define\\
$$u_{(r)}= u_1(q)+ \left(\beta\left(3-\frac{4s_1}{r}\right)(u_1(s_1,\theta_1)-u_1(q)) +\beta\left(\frac{4s_1}{r}-5\right)(u_2(s_1-2r,\theta_1-\tau)- u_2(q))\right).$$
\vskip 0.1in
\noindent
\v
We define weighted norms. Fix a positive function $W$ on $\Sigma$ which has order equal
to $e^{\alpha |s|} $ on each end of $\Sigma_i$, where $\alpha$ is a small
constant such that $0<\alpha  <1$. We will write the weight function simply as $e^{\alpha |s|} $. For any $h\in C^{\infty}_{c}(\Sigma_{i};u_{i}^{\ast}TM)$ and
any section $\eta \in
C^{\infty}_{c}(\Sigma_{i}, u_{i}^{\ast}TM\otimes \wedge^{0,1}_jT^{*}\Sigma_{i})$
we define the norms
\begin{equation}
\|h\|_{j_i,k,2,\alpha}=
 \left(\int_{\Sigma_i}e^{2\alpha|s|} \sum_{i=0}^{k} |\nabla^i h|^2
dvol_{\Sigma}\right)^{1/2},
\end{equation}
\begin{equation}
\|\eta\|_{j_i,k-1,2,\alpha}= \left(\int_{\Sigma}e^{2\alpha|s|}\sum_{i=0}^{k-1} |\nabla^i \eta|^2
dvol_{\Sigma}\right)^{1/2}.
\end{equation}
Here all norms and
covariant derivatives are taken with respect to the  metric $G_J$ on $u^{\ast}TM$ and
the metric $\mathbf{g}^\diamond$ on $(\Sigma, j, {\bf y})$, $dvol_{\Sigma}$ denotes the volume form with respect to $\mathbf{g}^\diamond$. Denote by $W^{k,2,\alpha}({\Sigma}_{i};u_{i}^{\ast}TM)$ and
$W^{k-1,2,\alpha}({\Sigma}_{i}, u_{i}^{*}TM\otimes \wedge^{0,1}_{j_i}T^{*}{\Sigma}_{i})$ the complete spaces with respect to the norms respectively.
\v
We choose $R_0$ so large that  $u_{i}(\{|s_i|\geq \frac{r}{2}\})$ lie in $O_{u_{i}(q)}$ for any $r>R_0$. In this coordinate system we identify $T_xM$ with $T_{u_{i}(q)}M$ for all $x\in O_{u_{i}(q)}$. Any $h_0\in T_{u_{i}(q)}M$ may be considered as a vector field in the coordinate neighborhood.
We fix a smooth cutoff function $\varrho$:
\[
\varrho(s)=\left\{
\begin{array}{ll}
1, &\mbox{ if }\ |s|\geq \bar{d}, \\
0, &\mbox{ if }\ |s|\leq \frac{\bar{d}}{2}
\end{array}
\right.
\]
where $\bar{d}$ is a large positive number. Put
$$\hat{h}_0=\varrho h_0.$$
Then for $\bar{d}$ large enough $\hat{h}_0$ is a section in $C^{\infty}( {\Sigma_i}; u_{i}^{\ast}TM)$
supported in the tube $\{(s,t)||s|\geq \frac{\bar{d}}{2}, t \in {S} ^1\}$.
Denote
$${\mathcal W}^{k,2,\alpha}({\Sigma}_{i};u_{i}^{\ast}TM)=\left\{h+\hat{h}_0 | h \in
W^{k,2,\alpha}({\Sigma}_{i};u_{i}^{\ast}TM),h_0 \in T_{u_{i}(q)}M\right\}.$$
We define weighted Sobolev  norm  on ${\mathcal W}^{k,2,\alpha}$ by
$$\| h+\hat{h}_{0}\|_{\mathcal{W},j,k,2,\alpha}=
 \sum_{i=1}^2\|h\|_{j_i,k,2,\alpha} + |h_{0}| ,$$
 where $|h_{0}|=[G_{J}(h_{0},h_{0})_{u(q)}]^{\frac{1}{2}}$.
\v
Denote
$$\beta_{1;R}(s_1)=\beta\left(\frac{1}{2}+\frac{r-s_1}{R}\right),\;\;\;
\beta_{2;R}(s_{2})=\sqrt{1-\beta^2\left(\frac{1}{2}-\frac{s_{2}+r}{R}\right)}, $$
where $\beta$ is the cut-off function defined in \eqref{def_beta}.
For any $\eta \in
C^{\infty}(\Sigma_{(r)};u_{(r)}^{\ast}TM\otimes \wedge_{j}^{0,1}T\Sigma_{(r)})$,
let
$$
\eta_{i}(p) =\left\{
\begin{array}{ll}
\eta  & if\;\; p\in \Sigma_{i0}\cup\{|s_{i}|\leq r-1\}\\
\beta_{i;2}(s_{i})\eta(s_{i},t_{i}) & if\;\; p\in \{r-1\leq |s_{i}|\leq r+1\} \\
0 & otherwise.
\end{array}
\right..
$$
If no danger of confusion  we will simply write $\eta_{i}=\beta_{i;2}\eta.$ Then $\eta_{i}$  can be considered as
a section over $\Sigma_i$. Define
\begin{equation}
\|\eta\|_{r,k-1,2,\alpha}=\|\eta_{1} \|_{\Sigma_1,j_1,k-1,2,\alpha} +
\|\eta_{2} \|_{\Sigma_2,j_2,k-1,2,\alpha}.
\end{equation}
\v
We now define a norm $\|\cdot\|_{r,k,2,\alpha}$ on
$C^{\infty}(\Sigma_{(r)};u_{(r)}^{\ast}TM).$ For any section
$h\in C^{\infty}(\Sigma_{(r)};u_{(r)}^{\ast}TM)$ denote
$$h_0=\int_{ {S}^1}h(r,t)dt,$$
$$h_1(s_1,t_1) = (h-\hat h_{0})(s_1,t_1)\beta_{1;2}(s_1),\;\;\;h_2(s_2,t_2)= (h-\hat h_{0})(s_2,t_2)\beta_{2;2}(s_{2}).$$
We define
\begin{equation}
\|h\|_{r,k,2,\alpha}=\|h_1\|_{\Sigma_1,j_1,k,2,\alpha} +
\|h_2\|_{\Sigma_2,j_2,k,2,\alpha}+|h_{0}|.
\end{equation}
Denote the resulting completed spaces by $W^{k-1,2,\alpha}(\Sigma_{(r)};u_{(r)}^{\ast}TM\otimes \wedge_{j_r}^{0,1}T\Sigma_{(r)})$ and $W^{k,2,\alpha}(\Sigma_{(r)};u_{(r)}^{\ast}TM)$  respectively.
\v
The above construction can be generalized to the case with several nodes. Let $\Sigma\in \overline{\mathcal M}_{g,n}$.
Suppose that $\Sigma$ has $\mathfrak{e}$ nodal points $\mathbf{p}=(p_{1},\cdots,p_{\mathfrak{e}})$, n marked points $y_{1},\cdots,y_n$ and $\iota$ smooth components. We can choose the plumbing coordinates $(\mathbf{s}, \mathbf{t})$
following \cite{Wolp-1}.
   Set $\stackrel{\circ}\Sigma= \Sigma \setminus \{p_{1},\cdots,p_{\mathfrak{e}},y_{1},\cdots,y_n\}$. Then $\stackrel{\circ}\Sigma$ is a Riemann surface with additional punctures $a_{j}, b_{j}$ in the place of the $j$th node of $\Sigma,$ $j=1,\cdots, \mathfrak{e}$.
For each node $p_{j},j=1,\cdots,\mathfrak{e}$, there is a neighborhood   isomorphic to
$$\{(z_{j},w_{j})\in \mathbb C^{2}| |z_{j}|<1,|w_{j}|<1,z_{j}w_{j}=0\}.$$
Denote by $\Sigma_{i}$ the connected components  of $\stackrel{\circ}\Sigma$, $i=1,\cdots, \iota$. Suppose that $\Sigma_i$ has $n_{i}$ marked points, $q_{i}$ punctures and has genus $g_{i}$.
\v
We can parameterize a neighborhood of $\stackrel{\circ}\Sigma$ in the deformation space  by Beltrami differentials.  Let $z_i$ (resp. $w_{i}$) be a local   coordinate around $a_{i}$ (resp. $b_{i}$), $z_i(a_{i})=0,w_i(b_{i})=0$, $i=1,\cdots, \mathfrak{e}$.
Let $\mathbb{U}_{j}=\{p\in\Sigma||z_{j}|(p)<1\}$ and  $\mathbb{V}_{j}=\{p\in\Sigma||w_{j}|(p)<1\}$ be
disjoint neighborhoods of the punctures $a_{j}$ and $b_{j},j=1,\cdots,\mathfrak{e}$.
We pick an open set $\mathbb{U}_{o}\subset \stackrel{\circ}\Sigma$ such that each component of $\stackrel{\circ}\Sigma$ intersects $\mathbb{U}_o$ in a nonempty relatively compact set and the intersection $\mathbb{U}_o\bigcap(\mathbb{U}_j\cup \mathbb{V}_j)$ is empty for all $j$. Denote $N=\sum\limits_{i=1}^{\fc} (3g_{i}-3+n_{i}+q_{i})$.
Choose Beltrami differentials $\nu_{j}, j=1,\cdots,N$ which are supported in $\mathbb{U}_o$ and
form a basis of the deformation space at $\Sigma $.
Let $\mathbf{s}=(\fs_{1},\cdots,\fs_{N})\in \mathbb C^{N}$,
$\nu=\sum\limits_{i=1}^{N} \fs_{i}\nu_{i}.$ Assume $|\mathbf{s}|$ small enough such that $|\nu|<1$. The nodal surface $\Sigma_{\mathbf{s},0}$ is obtained by solving the Beltrami equation $\bar{\p}w=\nu(\mathbf{s})w$.
\v
	We recall the plumbing construction for $\Sigma$ with a pair of punctures $a_{j},b_{j}$.  	
Let $z_{j,\mathbf{s}},$ $w_{j,\mathbf{s}}$ be cusp coordinates in
	$\mathbb{U}_{j}, \mathbb{V}_{j}$ near $a_{j},b_{j}$ respectively, thus
$$
ds^2_{\mathbf{s},0} (z_{j,\mathbf{s}})=\frac{|dz_{j,\fs}|^2}{|z_{j,\fs}|^2 \log^2 |dz_{j,\fs}|},\;\;\;ds^2_{\fs,0} (w_{j,\fs})=\frac{|dw_{j,\fs}|^2}{|w_{j,\fs}|^2 \log^2 |dw_{j,\fs}|}.
$$
where   $ds^2_{\fs,0} $ be the normalized hyperbolic metric on $\Sigma_{\fs,0}$  of curvature $-1$.
 As \cite{DasMes} denote
$$
	F_{j,\mathbf{s}}=z_{j}\circ z_{j,\mathbf{s}}^{-1},\;\;\;\;	G_{j,\mathbf{s}}=w_{j}\circ w_{j,\mathbf{s}}^{-1}.
$$
  By the removalbe singularity theorem and setting $\widetilde F_{j,\mathbf{s}}=F_{j,\mathbf{s}}/F_{j,\mathbf{s}}'(0)  $ and  $\widetilde G_{j,\mathbf{s}}=G_{j,\mathbf{s}}/G_{j,\mathbf{s}}'(0)$, if necessary , we can assume that
  $$
  F_{j,\mathbf{s}}(0)=0,\;\;\;  F'_{j,\mathbf{s}}(0)=1,\;\;\;G_{j,\mathbf{s}}(0)=0,\;\;\;  G'_{j,\mathbf{s}}(0)=1.\;\;\;  $$
Since $\mathbb{U}_{o}$ is disjoint from the $\mathbb{U}_{j}, \mathbb{V}_{j}$, the $F_{j,\mathbf{s}},G_{j,\mathbf{s}}$ are also holomorphic onto their image. For any $  \mathbf{t}=(\ft_{1},\cdot,\cdot,\ft_{\mathfrak{e}})$
with $0<|\ft_{j}|<1$, remove the discs $|z_{j}|< |\ft_{j}|$ and $|w_{j}|< |\ft_{j}|$ when $|\ft_{j}|$ small, and identify  $z_{j}$ via the plumbing equation  $$w_{j}=\frac{\ft_{j}}{z_{j}}.$$ We can rewrite the equation as
	$$
	(F_{j,\mathbf{s}}\circ z_{j,\mathbf{s}})\cdot (G_{j,\mathbf{s}}\circ w_{j,\mathbf{s}})=\ft_{j}.$$
Then we form a new Riemann surface $\Sigma_{\mathbf{s},\ft}.$ We call $(\ft_{1},\cdot\cdot\cdot,\ft_{\mathfrak{e}})$ plumbing coordinate. We obtain a family of Riemann surfaces over $\Delta_{\mathbf{s}}\times \Delta_{\ft}$,
whose fiber over $(\fs,\ft)$ is the Riemann surface $\Sigma_{\fs,\ft}$, where $\Delta_{\mathbf{s}}=(\Delta)^N\subset \mathbb{C}^N$, $\Delta_{\ft}=(\Delta)^{\mathfrak{e}}\subset \mathbb{C}^{\mathfrak{e}}$ are  polydiscs.
\v
In the coordinate system $(\fs,\ft)$ the $g_{\mathsf{wp}}$ metric induces a $Diff^+(\Sigma)$-invariant distance $d_{\mathbf{s},\ft}(\cdot,\cdot)$ on $\Delta_{\mathbf{s}}\times \Delta_{\ft}$.
Put
$$O(\delta)=\{(\fs,\ft)\mid d_{\mathbf{s}, \ft }((0,0),(\fs,\ft))<\delta \}.$$
We can choose $\delta$ small such that $\mathbf{G}_{(\fs,\ft)}$ can be imbedded into $\mathbf{G}_{(0,0)}$ as a subgroup for any $(\fs,\ft)\in O(\delta)$. Denote by $im(\mathbf{G}_{(\fs,\ft)})$ the imbedding.
\v
Let $u_{\mathbf{s},0}:\Sigma_{\mathbf{s},0}\to M$ be a $W^{k,2,\alpha}$-map. We can construct $u_{\mathbf{s},\ft}:\Sigma_{\mathbf{s},\ft}\to M$.
For any $h\in C^{\infty}_{c}(\Sigma_{\fs,0};u_{\fs,0}^{\ast}TM)$ and
any section $\eta \in
C^{\infty}_{c}(\Sigma_{\fs,0}, u_{\fs,0}^{\ast}TM\otimes \wedge^{0,1}_jT^{*}\Sigma_{\fs,0})$
we define the norms $\|h\|_{\fs,k,2,\alpha}$ and $\|\eta\|_{\fs,k-1,2,\alpha}$, and  for any section
$h\in C^{\infty}(\Sigma_{\fs,\ft};u_{\fs,\ft}^{\ast}TM)$
and any $\eta \in
C^{\infty}(\Sigma_{\fs,\ft};u_{\fs,\ft}^{\ast}TM\otimes \wedge_{j}^{0,1}T\Sigma_{\fs,\ft})$, we define the norms $\| h\|_{\fs,\ft,k,2,\alpha}$ and $\|\eta\|_{\fs,\ft,k-1,2,\alpha}$ as for one node case.
\v
Let $b_o=(\Sigma,0,0,u)$.
Set
$$\widetilde{\mathbf{O}}_{b_{o}}(\delta_{b_o},\rho_{b_o}):=\left\{((\fs,\ft),v_{\fs,\ft})\;|\; d_{\fs,\ft}((0,0),(\fs,\ft))<\delta_{b_o}, \|h\|_{\fs,\ft,k,2,\alpha}<\rho_{b_o} \right\},$$
$$\mathbf{O}_{[b_{o}]}(\delta_{b_{o}},\rho_{b_{o}})=
\widetilde{\mathbf{O}}_{b_{o}}(\delta_{b_{o}},\rho_{b_{o}})/G_{b_o},$$
where $v_{\fs,\ft}=\exp_{u_{\fs,\ft}}(h).$

\v
\v

\section{\bf Local regularization-Top strata}\label{local_regularization}

 When the transversality fails we need to take the regularization.
\subsection{\bf Local regularization for the top strata}\label{Local regularization for the top strata}

Let   $[b_o]=[(p_o,u)]\in \mathcal{M}_{g,n}(A)$ and let $\gamma_o\in \mathbf{T}_{g,n}$ such that $\pi(\gamma_o)=[p_o]$, where $\pi: \mathbf{T}_{g,n}\to \mathcal{M}_{g,n}$ is the projection.
We choose a local coordinate system $(\psi,\Psi)$ on $U$ with $\psi(\gamma_o)=a_o$ for $\mathcal{Q}$.
We view $a=(j,{\bf y})$ as a family of parameters defined on a fixed $\Sigma$. Denote
$$\widetilde{\mathcal B}(a)=\left\{u\in W^{k,2}(\Sigma,M)|\;u_{*}([\Sigma])=A\right\}.$$
Let $\widetilde{\mathcal E}(a)$ be the infinite dimensional Banach bundle over $\widetilde{\mathcal B}(a)$ whose fiber at $v$ is $$W^{k-1,2}(\Sigma,v^{*}TM\otimes\wedge_{j_a}^{0,1}T^{*} \Sigma),$$
where we denote by $j_a$ the complex structure on $\Sigma$ associated with $a=(j,{\bf y})$. We will denote $j_{a_o}:=j_o$. We have a continuous family of Fredholm system
$$\left(\widetilde{\mathcal B}(a),\; \widetilde{\mathcal{E}}(a),\;
\bar{\p}_{j_a,J}\right)$$
parameterized by $a\in \mathbf{A}$ with $d_{\mathbf{A}}(a_o,a)<\delta$. For any $v\in \widetilde{\mathcal B}(a)$ let $b=(a,v)$ and denote $\widetilde{\E}(a)|_{v}:=\widetilde{\E}|_{b}.$ Let $b_o=(a_o,u)$. Choose $\widetilde{K}_{b_o} \subset \widetilde{\E}|_{b_o}$ to be a finite dimensional subspace such that every member of $\widetilde{K}_{b_o}$ is in $C^{\infty}(\Sigma, u^{*}TM\otimes\wedge^{0,1}_{j_o}T^{*}\Sigma)$ and
\begin{equation}\label{regularization_operator}
\widetilde{K}_{b_o} + image D_{b_o} = \widetilde{\E}|_{b_o},
\end{equation}
where $D_{b_o}=D\bar{\p}_{j_o,J}$ is the vertical differential of $\bar{\p}_{j_o,J}$ at $u$.
\v
Let $G_{b_o}$ be the isotropy group at $b_o$.
In case the isotropy group $G_{b_{o}}$ is non-trivial, we must construct a $G_{b_{o}}$-equivariant regularization.
Note that $G_{b_{o}}$ acts on $W^{k-1,2}(\Sigma, u^{*}TM\otimes\wedge^{0,1}_{j_o}T^{*}\Sigma)$ in a natural way: for any $\kappa\in W^{k-1,2}(\Sigma, u^{*}TM\otimes\wedge^{0,1}_{j_o}T^{*}\Sigma)$ and any $g\in G_{b_{o}}$
$$g\cdot \kappa=\kappa\circ dg\in W^{k-1,2}\left(\Sigma, u^{*}TM\otimes\wedge^{0,1}_{j_o}T^{*}\Sigma\right).$$
Set
$$\bar{K}_{b_{o}}=\bigoplus_{g\in G_{b_{o}}}
g \widetilde{K}_{b_{o}}.$$
Then $\bar{K}_{b_{o}}$ is $G_{b_{o}}$-invariant. To simplify notations we assume that $\widetilde{K}_{b_o}$ is already $G_{b_{o}}$-invariant.
\v\v
\begin{lemma}\label{isomorphism} There are constants $\delta>0$, $\rho>0$ depending on $b_o$ such that there is an isomorphism
$$P_{b_{o}, b}:\widetilde{\E}_{b_o}\rightarrow \widetilde{\E}_{b}\;\;\;\forall \;b\in \widetilde{\mathbf{O}}_{b_{o}}(\delta,\rho).$$
\end{lemma}
\v\n
{\bf Proof.} Denote by $inj_{M}$ the injective radius of $(M,G_J)$. Given $h\in W^{k,2}(\Sigma,u^{*}TM)$ with $\|h\|_{L^{\infty}}<inj_M$, let
$$\Phi:u^*TM\rightarrow (\exp_{u}h)^*TM$$
denote the complex bundle isomorphism, given by parallel transport with  respect to the connection $\widetilde{\nabla}$, along the
geodesics $s\rightarrow \exp_{u}(sh)$. We choose $\rho< inj_{M}$. For any $j\in \mathcal{J}(\Sigma)$ near $j_o$
we can write $j=(I + H)j_o(I + H)^{-1}$ where $H\in T_{j_o}\mathcal{J}(\Sigma)$.
We define two maps
$$\Psi_{j_{o},j_a}:u^{*}TM\otimes \wedge^{0,1}_{j_{o}}T^{*}\Sigma  \to  u^{*}TM\otimes \wedge^{0,1}_{j_a}T^{*}\Sigma$$ and $$\Psi_{j_a,j_{o}}:u^{*}TM\otimes \wedge^{0,1}_{j_a}T^{*}\Sigma  \to  u^{*}TM\otimes \wedge^{0,1}_{j_{o}}T^{*}\Sigma$$ by
\begin{align*}
 \Psi_{j_{o},j_a}(\eta)= \frac{1}{2}(\eta-\eta \cdot j_{o} j_a),\;\;\;\; \Psi_{j_a,j_{o}}(\varpi)= \frac{1}{2}(\varpi-\varpi \cdot j_aj_{o}).
\end{align*}
Since $J\eta=-\eta j_{o}$ and $J\varpi=-\varpi j_a,$ one can check that $J \Psi_{j_{o},j_a}(\eta) =-\Psi_{j_{o},j_a}(\eta) j$ and $J \Psi_{j_a,j_{o}}(\varpi) =-\Psi_{j_a,j_{o}}(\varpi) j_{o}.$ Then $\Psi_{j_{o},j_a}$ and $\Psi_{j_a,j_{o}}$ are well defined. The proof of the following claim can be find in \cite{LS-1}, we omit the proof here.
\v\n
{\bf Claim.}
$\Psi_{j_{o},j_a}$ is an isomorphism when $|H|$ small enough.
\v
$\Psi_{j_{o},j_a}$ induces an isomorphism  $$\Psi_{j_{o},j_a}: W^{k-1,2}(\Sigma, u^{*}TM\otimes\wedge^{0,1}_{j_o}T^{*}\Sigma)\to W^{k-1,2}(\Sigma, u^{*}TM\otimes\wedge^{0,1}_{j_a}T^{*}\Sigma)$$
in a natural way. Let $P_{b_{o}, b}=\Phi\circ\Psi_{j_o,j_a}$. We can choose $\delta$, $\rho$ such that there is an isomorphism
$$P_{b_{o}, b}:\widetilde{\E}_{b_o}\rightarrow \widetilde{\E}_{b}\;\;\;\forall \;b\in \widetilde{\mathbf{O}}_{b_{o}}(\delta,\rho).$$

\v
Now we define a thickned Fredholm system $(\widetilde{K}_{b_o}\times \widetilde{\mathbf{O}}_{b_{o}}(\delta,\rho), \widetilde{K}_{b_o}\times \widetilde{\E}|_{\widetilde{\mathbf{O}}_{b_{o}}(\delta,\rho)}, S)$.
Let $(\kappa, b)\in \widetilde{K}_{b_o}\times \widetilde{\mathbf{O}}_{b_{o}}(\delta,\rho)$, $b=(a,v)\in \widetilde{\mathbf{O}}_{b_{o}}(\delta,\rho)$. Define
\begin{equation}\label{local regu}
S(\kappa,b) = \bar{\partial}_{j_a,J}u + P_{b_o,b}\kappa.
\end{equation}
We can choose $(\delta, \rho)$ small such that
the linearized operator $DS_{(\kappa,b)}$ is surjective for any $b\in \widetilde{\mathbf{O}}_{b_{o}}(\delta,\rho)$.
\v
\subsection{\bf The norm $\|h\|_{j_a,k,2}^{2}$ and the isotropy group on
the top strata}\label{norm on the top strata}

If we fix the complex structure $j_o$, then $W^{k,2}(\Sigma;u^{\ast}TM)$ is a Hilbert space. It is well-known that $\|h\|_{j_o,k,2}^{2}$ is a smooth function ( see \cite{MVV}). Now the $\|h\|_{j_a,k,2}^{2}$ is a family of norms, so the following lemma is important.
\begin{lemma}\label{smoothness of norms}
For any $[b_o]=[(p_o,u)]\in \mathcal{M}_{g,n}(A)$ and any local coordinates $(\psi,\Psi)$ on $U$ with $\psi: U\rightarrow \mathbf{A}\ni a_o$
the norm
$\|h\|_{j_a,k,2}^{2}$ is a smooth function in $\widetilde{\mathbf{O}}_{b_{o}}(\delta,\rho)$.
\end{lemma}
\v\n
{\bf Proof.} We first prove that, in the coordinate system $(\psi,\Psi)$ on $U$, $\|h\|_{j_a,k,2}^{2}$ is smooth. Since the slice we choose is smooth, the complex structure $j_a$ smoothly depends on $a$, so the associated hyperbolic metric $\mathbf{g}$ smoothly depends on $a$. It follows that
$\|h\|_{j_a,k,2}^{2}$ is smooth in $(\psi,\Psi)$. Now we choose another coordinate system $(\psi',\Psi')$, then
$$a=\psi\circ (\psi')^{-1}(a'),\;\;\;u=u'\circ \vartheta_{a},$$
where $\vartheta_{a}=\Psi'\circ \Psi^{-1}|_{a}\in Diff^{+}(\Sigma)$ is a family of diffeomorphisms.
On the other hand, in terms of the two coordinate system we have   $$\|h\|_{j_a,k,2}^{2}=\|h\|_{j_{a'},k,2}^{2}.$$
Then
$$\frac{\p}{\p a'}\|h\|_{j_{a'},k,2}^{2}=\frac{\p}{\p a'}\|h\|_{j_{a},k,2}^{2}=\frac{\p}{\p a}\|h\|_{j_{a},k,2}^{2}\frac{\p a}{\p a'}.$$
By Lemma \ref{local coordinates-3}, $\psi'\circ \psi^{-1}$ is smooth. So $\|h\|_{j_{a'},k,2}^{2}\in C^1$. By the same argument we conclude that $\|h\|_{j_{a'},k,2}^{2}$ is smooth. $\Box$

\begin{lemma}\label{lem_orbi_T} There exist two constants $\delta_{o},\rho_{o}>0$ depend only on $b_{o}$ such that for any $\delta<\delta_{o},\rho<\rho_{o}$ the following hold.
 \begin{itemize}
\item[(1)] For any $p\in \widetilde{\mathbf{O}}_{b_{o}}(\delta,\rho)$, let $G_{p}$ be the isotropy group at $p$, then $im(G_{p})$ is a subgroup of $G_{b_o}$.
\item[(2)] Let $p\in \widetilde{\mathbf{O}}_{b_{o}}(\delta,\rho)$
be an arbitrary point with isotropy group $G_p$, then there is a $G_p$-invariant neighborhood $O(p)\subset \widetilde{\mathbf{O}}_{b_{o}}(\delta,\rho)$ such that for any $q\in O(p)$, $im(G_{q})$ is a subgroup of $G_{p}$, where $G_p$, $G_q$ denotes the isotropy groups at $p$ and $q$ respectively.
\end{itemize}
\end{lemma}
\begin{proof} We only prove (2), the proof of (1) is similar. If (2) is not true, there exists a point $p=(a,\hat u)\in \widetilde{\mathbf{O}}_{b_{o}}(\delta,\rho)$ and a sequence $b_i =( a_i,u_i)\in \widetilde{\mathbf{O}}_{p}(\delta_i, \rho_i)$ such that
\v
{(a)} $\delta_i\to 0$, $\rho_i\to 0$,
\v
{(b)} $im(G_{b_i})$ is not a subgroup of $G_{p}$.
\v\n
By (a) we have
\begin{equation}\label{eqn_a_16}
 \lim_{i\to \infty} a_{i}=a,\;\;\; \lim_{i\to \infty} \|h_{i}\|_{k,2}=0,
\end{equation}
where $u_{i}=\exp_{\hat u}(h_{i})$. Since $\mathcal M_{g,n}$ is an orbifold, by \eqref{eqn_a_16} we have $im(\mathbf G_{a_{i}})$   is a subgroup of $\mathbf{G}_{a}$. Obviously, $G_{b_{i}}$  is a subgroup of   $\mathbf G_{a_{i}}$. Then $im(G_{b_i})$ is a subgroup of $\mathbf{G}_{a}$.
   Since $\mathbf{G}_{a}$ is a finite group, by choosing subsequence we may assume that   $im(G_{b_i})$ is a fixed subgroup of $\mathbf{G}_{a}$ independent of $i$, denoted by $im(G_{b})$. By (b), $im(G_{b})$ is not a subgroup of $G_{p}$.
On the other hand, by \eqref{eqn_a_16}
 $ G_{b_i} \cdot u_i$ converges to $ G_{b} \cdot \hat u$ and $u_i$ converges to $\hat u$ in $W^{k,2}$.
  Hence for any $g\in  G_{b} $, we have
  $$
  \|g\cdot \hat u-\hat u \|_{k,2}=0.
  $$
It follows that $im(G_{b})$ is a subgroup of $G_{p}$. We get a contradiction.
\end{proof}
\v
\begin{remark} It is easy to see from the proof that the Lemma \ref{lem_orbi_T} also hold for lower stratum.
\end{remark}

\section{\bf Local regularization-Lower strata}\label{local_regularization-lower strata}

\subsection{\bf Local regularization for lower stratum : without bubble tree}\label{without bubble tree}

Let $\Sigma$, $\stackrel{\circ}\Sigma$, $\Sigma_i$ be as in \S\ref{gluing}. We choose local plumbing coordinates $(\fs,\ft)$ and construct $\Sigma_{\fs,\ft}\to \Delta_{\fs}\times \Delta_{\ft}$.
Consider the family of Bananch manifold
$$
 \widetilde{\mathcal B}(\fs,\ft)=\{u\in   W^{k,2,\alpha}( \Sigma_{\fs,\ft},M)|\;u_{*}([\Sigma])=A\}.
 $$
Let $\widetilde{\mathcal E}(\fs,\ft)$ be the infinite dimensional Banach bundle over $\widetilde{\mathcal B}(\fs,\ft)$ whose fiber at $b=(\fs,\ft,u)$ is
$W^{k-1,2,\alpha}(\Sigma_{\fs,\ft}, u^{*}TM\otimes\wedge^{0,1}_{j_{\fs,\ft}}T^{*}\Sigma_{\fs,\ft}).$
We have a continuous family of Fredholm system
$$\left(\widetilde{\mathcal{B}}(\fs,\ft),\widetilde{\mathcal E}(\fs,\ft),\bar{\p}_{j_{\fs,\ft},J}\right)$$
parameterized by $(\fs,\ft)\in \Delta_{\fs}\times \Delta_{\ft}$. Let $b_o=(0,0,u)$, $b=(\fs,\ft,v)$. We use the same method as in subsection \S\ref{Local regularization for the top strata} to choose $\widetilde{K}_{b_{o}}=\bigoplus_{i=1}^{\iota}\widetilde{K}_{b_{oi}} \subset \widetilde{\E}|_{b_{o}}=\bigoplus_{i=1}^{\iota}\widetilde{\E}_{b_{oi}}$ to be a finite dimensional subspace such that
\begin{itemize}
\item[(1)]  Every member of $\widetilde{K}_{b_{oi}}$ is in
$C^{\infty}\left(\Sigma_{i,0}, u^{*}_{i}TM\otimes\wedge^{0,1}_{j_{oi}}T^{*} \Sigma_{i, 0}\right)$ and supports in the compact subset $\Sigma_{0,0}(R_0) $ of $\Sigma_{0,0}$.
\item[(2)] $\widetilde{K}_{b_{oi}} + image D_{b_{oi}} = \widetilde{\E}|_{b_{oi}},\;\;\;\forall i=1,2,...,\iota.$
\item[(3)] $\widetilde{K}_{b_{oi}}$ is $G_{b_{oi}}$-invariant.
\end{itemize}
where we denote by $j_{oi}$ the complex structure on $\Sigma_i$ associated with $(0, 0)$, and
\begin{equation}\label{eqn_sig_R}
W(R_{0}):=\cup_{l=1}^{\mathfrak{e}} (\{  |z_{l}|< e^{-R_{0}}\}\cup \{ |w_{l}|< e^{-R_{0}}\})\cup \mathbf D(e^{-R_{0}}),\;\;\;\;\;\;\Sigma_{\fs,\ft}(R_{0})=\Sigma_{\fs,\ft}\setminus W(R_{0}) .
\end{equation}
for a constant $R_{0}>1$. We identify each $\Sigma_{\fs,\ft}(R_0)$ with $\Sigma_{0,0}(R_0):=\Sigma(R_0)$ for $|\fs|,|\ft|$ small. Denote by $j_{\fs,\ft}$ the family of complex structure on $\Sigma(R_0)$. Denote $j_o:=j_{0,0}$.
Then when $|H|$ small
$$\Psi_{j_{o},j_{\fs,\ft}}: W^{k-1,2,\alpha}(\Sigma(R_0), u^{*}TM\otimes \wedge^{0,1}_{j_{o}}T^{*}\Sigma(R_0))\to W^{k-1,2,\alpha}(\Sigma(R_0), u^{*}TM\otimes \wedge^{0,1}_{j_{\fs,\ft}}T^{*}\Sigma(R_0))$$
is an isomorphism. Let $P_{b_{o}, b}=\Phi\circ\Psi_{j_o,j_{\fs,\ft}}$.
We fix a smooth cutoff
function $\beta_{R_{0}}: {\mathbb{R}}\rightarrow [0,1]$ such that
\begin{equation}\label{def_beta1}
\beta_{R_{0}} (s)=\left\{
\begin{array}{ll}
0 & if\;\; |s| \geq R_0 \\
1 & if\;\; |s| \leq R_0-1.
\end{array}
\right.
\end{equation}
\begin{lemma}\label{isomorphism-1} Let
$\bar{\mathcal E}(\fs,\ft)$ be the infinite dimensional Banach bundle over $\widetilde{\mathcal B}(\fs,\ft)$ whose fiber at $b=(\fs,\ft,u)$ is
$$\bar{\mathcal E}_{(\fs,\ft,u)}:=\{\beta_{R_{0}}(s)\eta\;| \;\eta\in \widetilde{\mathcal E}_{(\fs,\ft,u)}\}.$$
Then there are constants $\delta>0$, $\rho>0$ depending on $b_o$ such that there is an isomorphism
$$P_{b_{o}, b}:\bar{\E}_{b_o}\rightarrow \bar{\E}_{b}\;\;\;\forall \;b\in \widetilde{\mathbf{O}}_ {b_{o}}(\delta,\rho).$$
\end{lemma}
\v\n
The proof is the same as in Lemma \ref{isomorphism}.
\v
Now we define a thickned Fredholm system $(\widetilde{K}_{b_o}\times \widetilde{\mathbf{O}}_{b_{o}}(\delta,\rho), \widetilde{K}_{b_o}\times \widetilde{\E}|_{\widetilde{\mathbf{O}}_{b_{o}}(\delta,\rho)}, S)$.
Let $(\kappa, b)\in \widetilde{K}_{b_o}\times \widetilde{\mathbf{O}}_{b_{o}}(\delta,\rho)$, $b=(a,v)\in \widetilde{\mathbf{O}}_{b_{o}}(\delta,\rho)$. Define
\begin{equation}\label{local regu}
S(\kappa,b) = \bar{\partial}_{j,J}u + P_{b_o,b}\kappa.
\end{equation}
We can choose $(\delta, \rho)$ small such that
the linearized operator $DS_{(\kappa,b)}$ is surjective for any $b\in \widetilde{\mathbf{O}}_{b_{o}}(\delta,\rho)$.

\subsection{\bf Local regularization for lower stratum : with bubble tree}\label{with bubble tree}
\v

{\bf A-G-F procedure.} We introduce the A-G-F procedure.
\v
Consider a strata $\mathcal{M}^{\Gamma}$ of $\overline{\mathcal{M}}_{g,n}(A)$. Let
$b_{o}=[(\Sigma,j,\mathbf y,u)]\in \mathcal{M}^{\Gamma}$.
Then $(\Sigma, j,\mathbf y)$ is a marked nodal Riemann surface. Suppose that $\Sigma$
has a principal part $\Sigma^P$ and some bubble tree $\Sigma^B$ attaching to $\Sigma^{P}$ at $q$. Let $u=(u_{1},u_{2})$ where $u_{1}:\Sigma^{P}\to M$ and $u_{2}:S^{2}\to M$ are $J$-holomorphic maps. We consider the simple case $\Sigma^{B}=(S^2,q)$ with $[u_2(S^2)]\ne 0,$ the general cases are similar. Denote $b_{oo}:=(S^{2},q,u_2)$,
$$
\widetilde O_{b_{oo}}(\rho_{o})=\left\{ v\in \mathcal W^{k,2,\alpha}((S^2,q),u_2^*TM)| \|h\|_{k,2,\alpha} \leq \rho_{o},\mbox{ where } v=\exp_{u_2}(h)\right\}.
$$
$$
O_{b_{oo}}(\rho_{o})=\widetilde O_{b_{oo}}(\rho_{o})/G_{b_{oo}}
$$
where $G_{b_{oo}}=\{\phi\in Diff(S^2)\mid \phi^{-1}(q)=q, \; u_2\circ\phi=u_2\}$ is the isotropy group at $b_{oo}$.
\v

We can choose a local smooth codimension-two submanifold $Y$ such that $u_2(S^2)$ and $Y$ transversally intersects, and $u_2^{-1}(Y)=\mathbf{x}=(x_1,...,x_{\ell})$ ( see \cite{TF17} and \cite{Par}). We add  these intersection points as marked points to $S^{2}$ such that $S^2$ is stable. Denote the Riemann surface by $(S^{2},q,\mathbf{x})$. We may choose $\rho_{o}$ such that for any $(S^{2},q,v)\in O_{b_{oo}}(\rho_{o})$,  $v(S^2)$ and $Y$ transversally intersects, and  $v^{-1}(Y)$ has $\ell$ points.
Denote
$$
\widetilde O_{\hat b_{oo}}(1+\ell,\rho_{o})=\left\{(S^{2},q,\mathbf{x},v)|v(\mathbf{x})\in Y,\;v\in \widetilde O_{b_{oo}}(\rho_{o})\right\}.
$$
 Note that the additional marked points are unordered, so we consider the space
$$\widetilde O_{\hat b_{oo}}(1\mid\ell,\rho_{o}) Â£Âº=\widetilde O_{\hat b_{oo}}(1+\ell,\rho_{o})/Sy(\ell)$$
where $Sy(\ell)$ denotes the symmetric group of order $\ell$.  Denote $\hat b_{oo}:=(S^{2},q\mid \mathbf{x},u_2),$
where the points after $``\mid"$ are unordered.
Denote
$$G_{\hat{b}_{oo}}=\left\{\phi\in Diff(S^2)\mid \phi^{-1}(q)=q,\; u_2\circ\phi=u_2,\; \phi^{-1}\{x_1,...,x_{\ell}\}=\{x_1,...,x_{\ell}\}\right\}.$$
For any $\phi\in G_{b_{oo}}$, since $u_2\circ\phi=u_2$, we have $\phi^{-1}\{x_1,...,x_{\ell}\}=\{x_1,...,x_{\ell}\}$.
Then the following lemma holds.
\begin{lemma}\label{isotropy group}
$G_{b_{oo}}=G_{\hat b_{oo}}.$
\end{lemma}
\v
Let $\tilde{b}_{oo}:=(S^{2},q,\mathbf{x},u_2)$ be a representive of
$\hat b_{oo}:=(S^{2},q\mid\mathbf{x},u_2),$ where $\mathbf{x} =(x_1,...,x_{\ell})$ is an ordered set. We choose cusp coordinates $z$ on $\Sigma^{P}$ and  $w$ on $S^{2}$ near $q$. We can construct a metric $\mathbf {g}$ on $(S^{2},q,\mathbf{x})$ as in section \S\ref{metric on surfaces} such that  $\mathbf g^{\diamond}$  is the standard cylinder  metric near marked points and nodal points.  Put $\Sigma_1=\Sigma^P$, $\Sigma_2=S^2$, $b_o=(b_{o1}, b_{o2})$. Let $G_{b_{oi}}$ be the isotropy group at $b_{oi}.$
Denote $\tilde{b}_o=(\tilde{b}_{o1},\tilde{b}_{o2})$, where $\tilde{b}_{o1}$ is a lift of $b_{o1}$ to the uniformization system, and $\tilde{b}_{o2}:=\tilde{b}_{oo}.$  Note that the cusp coordinates $z$ and $w$ are unique modulo rotations near nodal point $q$ and the metric $\mathbf{g}$ on $\Sigma^{P}$ is  $G_{b_{1}}$-invariant
   and  $\mathbf{g}$ on  $(S^{2},q,\mathbf{x})$ is $G_{b_{2}}$-invariant.  In the coordinates $z,w$   for any  $\phi_{i}\in G_{b_{i}}$,  $$\phi_{1}(z)=e^{-\sqrt{-1} \gamma_{1}}z,\;\;\;\phi_{2}(w)=e^{-\sqrt{-1} \gamma_{2}}w.$$ By the finitness of  $G_{b_{i}}$, we have $\gamma_{i}=\frac{2j_{i}\pi}{l_{i}}$ where $j_{i}<l_{i},j_{i},l_{i}\in \mathbb Z,i=1,2.$
\v
We choose $$\widetilde{K}_{b_{o}}=\bigoplus_{i=1}^{2}\widetilde{K}_{b_{oi}} \subset \widetilde{\E}|_{b_{o}}=\bigoplus_{i=1}^{2}\widetilde{\E}_{b_{oi}}$$ to be a finite dimensional subspace satisfying (1), (2) and (3) in \S\ref{without bubble tree}.
\v
Then we glue $\tilde{b}_{o1}$ and $\tilde{b}_{o2}$ at $q$ with gluing parameters $(r^*, \tau^*)$ in the coordinates $z$, $w$ to get representives of $\hat{p}^*:=(\Sigma_{(r^*)}, \mathbf{y}\mid \mathbf{x})$ and pregluing map $\hat{u}_{(r^*)}$. Let $\hat{b}^*_{o}=(\hat{p}^*, \hat{u}_{(r^*)})$, denote by $G_{\hat{b}^*_{o}}$ the  isotropy group at $\hat{b}^*_{o}$. Now we forget Y and the additional marked points $\mathbf{x}$. We get a element $\Sigma^*:=\Sigma_{(r^*)}$, which is a point $p^*=(\Sigma_{(r^*)},\mathbf{y})\in \overline{\mathcal{M}}_{g,n}$. Let $b^*_o=(p^*, u_{(r^*)})$, denote by $\mathbf{G}_{p^*}$ and $G_{b^*_o}$ the isotropy groups at $p^*$ and $b^*_o$ respectively.
The following lemma is obvious.
\begin{lemma}\label{isotropy group-1}
$G_{\hat{b}^*_{o}}=G_{b^*_{o}}.$
\end{lemma}

\v
We call this procedure a {\bf A-G-F} procedure ( Adding marked points-Gluing-Forgetting Y and marked points). This procedure can be extended to bubble tree and bubble chain in an obvious way.
\v
We use the same method as in \S\ref{without bubble tree} to construct the local regularization.

\v
\section{\bf Global regularization}\label{global_r}
\v

\subsection{\bf A finite rank orbi-bundle over $\overline{\mathcal{M}}_{g,n}(A)$ }\label{finite rank orbi-bundle}
\v
By the compactness of $\overline{\mathcal{M}}_{g,n}(A)$ there exist finite points $[b_i]\in \overline{\mathcal{M}}_{g,n}(A)$, $1\leq i \leq \mathfrak{m}$, such that
\begin{itemize}
\item[(1)] The collection $\{\mathbf{O}_{[b_i]}(\delta_i/3,\rho_i/3)
\mid 1\leq i \leq \mathfrak{m}\}$ is an open cover of $ \overline{\mathcal{M}}_{g,n}(A)$.
\item[(2)] Suppose that $\widetilde{\mathbf{O}}_{b_i}(\delta_i,\rho_i)
\cap \widetilde{\mathbf{O}}_{b_j}(\delta_j,\rho_j)
\neq\phi$. For any $b\in \widetilde{\mathbf{O}}_{b_i}(\delta_i,\rho_i)
\cap \widetilde{\mathbf{O}}_{b_j}(\delta_j,\rho_j)$, $G_b$ can be imbedded into both $G_{b_i}$ and $G_{b_j}$ as subgroups.
\end{itemize}
\begin{remark}\label{finite covering}
We may choose $[b_i]$, $1\leq i \leq \mathfrak{m}$, such that if $[b_i]$ lies in the top strata for some $i$, then
$\mathbf{O}_{[b_i]}(\delta_i,\rho_i)$ lies in the top strata.
\end{remark}
\v\n
Set
$$\mathcal{U}=\bigcup_{i=1}^{\mathfrak{m}}
\mathbf{O}_{[b_{i}]}(\delta_i/2,\rho_i/2).$$
There is a forget map
$$forg: \mathcal{U}\to \overline{\mathcal{M}}_{g,n},\;\;\;[(j,{\bf y}, u)]\longmapsto [(j,{\bf y})].$$
\v
We construct a finite rank orbi-bundle $\mathbf{F}$ over
$\mathcal{U}$. The construction imitates Siebert Â¡Â¯s construction. First of all, we can slightly
deform $\omega$ to get a rational class $[\omega^*]$. By taking multiple, we can
assume that $[\omega^*]$ is an integral class on $M$. Therefore, it is the Chern class of a complex line bundle $L$ over $M$. Let $i$ be the complex structure on $L$. We choose a Hermition metric $G^L$ and the associate unitary connection $\nabla^{L}$ on $L$.
\v
Let $(\Sigma, j, {\bf y})$ be a marked nodal Riemann surface of genus $g$ with $n$ marked points.
Let $u:\Sigma\longrightarrow M$ be a $W^{k,2}$ map.
We have complex line bundle $u^*L$ over $\Sigma$ with complex structure $u^*i$.
The unitary connection $u^{*}\nabla^{L}$ splits into $ u^{*}\nabla^{L}:=u^{*}\nabla^{L,(1,0)}\oplus u^{*}\nabla^{L,(0,1)}$.
Denote
$$D^L:=u^{*}\nabla^{L,(0,1)}:W^{k,2}(\Sigma,u^{*}L)\to W^{k-1,2}(\Sigma,u^{*}L\otimes\wedge_{j}^{(0,1)}T^{\star}\Sigma).$$
$D^{L}$ takes $s\in W^{k,2}(\Sigma,u^{*}L)$ to the $\mathbb C$-antilinear part of $\nabla^{L}$, where $s$ is a section of $L$. One can check that
$$
D^{L}(f\xi)=\bar{\p}_{\Sigma}f \otimes \xi +  f\cdot D^{L}\xi.
$$
 $D^{L}$ determines a holomorphic structure on $u^*L$, for
which $D^{L}$ is an associated Cauchy-Riemann operator (see \cite{HLS,IS}).
Then $u^*L$ is a holomorphic line bundle.
\v
Let $\lambda_{(\Sigma, j)}$ be
the dualizing sheaf of meromorphic 1-form with at worst simple pole at the nodal points and for each
nodal point p, say $\Sigma_1$ and $\Sigma_2$ intersects at p,
$$Res_p(\lambda_{(\Sigma_1, j_1)}) + Res_p(\lambda_{(\Sigma_2, j_2)})=0.$$
Let $\Pi:\overline{\mathscr{C}}_{g}\to \overline{\mathcal{M}}_{g}$ be the universal curve. Let $\lambda$ be the relative dualizing sheaf over $\overline{\mathscr{C}}_{g}$, the restriction of $\lambda$ to $(\Sigma, j)$  is $\lambda_{(\Sigma, j)}$.
\v
Set $\Lambda_{(\Sigma, j)}:=\lambda_{(\Sigma, j)}\left(\sum_{i=1}^{n} y_i\right)$.
 Let $(\psi, \Psi): (O, \pi_{\mathbf{T}}^{-1}(O))\rightarrow (\mathbf{A}, \mathbf{A}\times \Sigma)$ be a local coordinate systems, where $O\subset \mathbf{T}_{g,n}$ is an open set.
 $\Lambda$ induces a line bundle over  $ \mathbf{A}\times \Sigma,$ denoted by $\widetilde{\Lambda}$.
 Then
$\widetilde{\mathbf{L}}\mid_{b}:=\mathscr{P}^{*}\widetilde{\Lambda}\otimes u^*L $ is a holomorphic line bundle over $\Sigma$, where $\mathscr{P}$ denote the forgetful map. We have a Cauchy-Riemann operator $  \bar{\p}_b.$
Then $H^0(\Sigma, \widetilde{\mathbf{L}}\mid_{b})$ is the $ker \bar{\p}_b$. Here the $\bar{\p}$-operator depends on the complex structure $j$ on $\Sigma$ and the bundle $u^*L$, so we denote it by $\bar{\p}_b$.
\v
If $\Sigma_\nu$
is not a ghost component, there exist a constant $\hbar_{o}>0$ such that
$$\int_{u(\Sigma_\nu)}\omega^*> \hbar_{o} .$$ Therefore, $c_1(u^*L)(\Sigma_\nu)>0$. For ghost component $\Sigma_\nu$, $\lambda_{\Sigma_\nu}\left(\sum_{i=1}^{n} y_i\right)$ is positive. So
for any $b=(a,v)\in \widetilde{\mathbf{O}}_{b_{o}}(\delta,\rho)$
by taking the higher power of $\widetilde{\mathbf{L}}\mid_{b}$, if necessary, we can assume that
$\widetilde{\mathbf{L}}\mid_{b}$ is very ample. Hence, $H^1(\Sigma, \widetilde{\mathbf{L}}\mid_{b})= 0$. Therefore,
$H^0(\Sigma, \widetilde{\mathbf{L}}\mid_{b})$ is of constant rank ( independent of $b\in \widetilde{\mathbf{O}}_{b_{o}}(\delta,\rho)$ ). We have a finite rank bundle $\widetilde{F}$ over $\widetilde{\mathbf{O}}_{b_{o}}(\delta,\rho)$,
whose fiber at $b=(j,{\bf y},v)\in\widetilde{\mathbf{O}}_{b_{o}}(\delta,\rho)$
is $H^0(\Sigma, \widetilde{\mathbf{L}}\mid_{b})$.
The finite group $G_{b}$ acts on the bundle on $\widetilde{F}\mid_{b}$ in a natural way.
\v
\begin{lemma}\label{Transformation for F}
For any $\varphi\in Diff^+(\Sigma)$ denote
	$$ b'=(j',{\bf y}',u') =\varphi\cdot (j,{\bf y},u)=(\varphi^*j, \varphi^{-1}{\bf y}, \varphi^{*}u).$$
	Then the following hold
\v {\bf (a).} $\mathbf{L}|_{b'}=\varphi^{*}\mathbf{L}|_{b},$\;\;\; $(u')^{*}i=\varphi^{*}(u^{*}i)$
\v {\bf (b).}
$D^{\mathbf L}|_{b'}(\varphi^{*}\xi)=\varphi^{*}(D^{\mathbf L}|_{b}(\xi)).$
\end{lemma}
\v\n
It follows from {\bf (b)} above that
if we choose another coordinate system $\mathbf{A}'$ and another local model $\widetilde{\mathbf{O}}_{b'_{o}}(\delta',\rho')/G_{b'_o}$, we have
$$H^0(\Sigma, \widetilde{\mathbf{L}}\mid_{b})=H^0(\Sigma, \widetilde{\mathbf{L}}'\mid_{b'}).$$
But the coordinate transformation is continuous. So we get a continuous bundle
$F\rightarrow \mathcal{U}$. Moreover,  by (1) and (2) we conclude that $F$ has a ``orbi-vector bundle" structure over $\mathcal{U}$.
\v

\v
Both $\widetilde{K}_{b_i}$ and $\tilde{F}\mid_{b_i}$ are representation spaces of $G_{b_i}$. Hence they can be decomposed as sum of irreducible representations.
There is a result in algebra saying that the irreducible factors of group
ring contain all the irreducible representations of finite group. Hence, it is
enough to find a copy of group ring in $\widetilde{F}(b_i)\mid_{b_i}$. This is done by algebraic geometry.
We can assume that $\mathbf L$
induces an embedding of $\Sigma$
into $\mathbb{C}P^{N_i}$ for some $N_i$. Furthermore, since $\mathbf L$ is invariant under
$G_{b_i}$ , $G_{b_i}$ also acts effectively naturally on $\mathbb{C}P^{N_i}$. Pick any point $x_0 \in
im(\Sigma)\subset \mathbb{C}P^{N_i}$ such that $\sigma_k(x_0)$ are mutually different for any $\sigma_k\in G_{b_i} $.
Then, we can find a homogeneous polynomial f of some degree, say $\mathsf{k}_i,$ such
that $f(x_0)\ne 0$, $f(\sigma_k(x_0)) = 0$ for $\sigma_k\ne I_d$. Note that $f\in H^0(\mathcal O(\mathsf{k}_i))$. By pull back over $\Sigma$, $f$ induces
a section $v \in H^0(\Sigma, \mathbf L
^{\mathsf{k}_i})$. We replace $\mathbf L$
by $\mathbf L^{\mathsf{k}_i}$ and redefine $F_i\mid_{b_i} =   H^0(\Sigma, \mathbf L^{\mathsf{k}_i}\mid_{b_i})$. Then $G_{b_i}\cdot v$ generates a group ring, denoted by $\ll G_{b_i}\cdot v\gg$. It is obvious that $\ll G_{b_i}\cdot v\gg$ is isomorphic to $\mathbb{R}[G_{b_i}]$, so $F_i\mid_{b_i}$ contains a copy of group ring. We denote the obtained bundle by $\mathbf{F}(\mathsf{k}_i)$.
\begin{lemma}\label{finite cov}
We have a continuous ``orbi-vector bundle" $\mathbf{F}(\mathsf{k}_i)\rightarrow \mathcal{U}$ such that
$\mathbf{F}(\mathsf{k}_i)\mid_{b_i}$ contains a copy of group ring
$\mathbb{R}[G_{b_i}].$
\end{lemma}
\v\n
In \cite{LS-2} we proved
\begin{lemma}\label{smoothness of bundle-1} For the top strata, in the local coordinate system $\mathbf{A}$ the bundle $\widetilde{\mathbf{F}}$ is smooth. Furthermore, for any base $\{e_{\alpha}\}$ of the fiber at $b_o$ we can get a smooth frame fields $\{e_{\alpha}(a,h)\}$ for the bundle $\widetilde{\mathbf{F}}$ over $\widetilde{\mathbf{O}}_{b_o}(\delta_{o},\rho_{o})$ .
\end{lemma}

\v
\begin{remark}\label{isotropy group}
Let $G_{b_o}$ be the isotropy group at $b_o$. $D^{\widetilde{\mathbf L}}$ is $G_{b_o}$-equivariant and $G_{b_o}$ acts on $\mbox{ker} D^{\widetilde{\mathbf L}}|_{b_{o}}$. We may choose a $G_{b_o}$-equivariant right inverse $Q^{\widetilde{\mathbf{L}}}_{b_{o}}$.
So we have a $G_{b_o}$-equivariant version of Lemma \ref{smoothness of bundle-1}. In particular,
for any base $\{e_{\alpha}\}$ of the fiber at $b_o$ we can get a smooth $G_{b_o}$-equivariant frame fields $\{e_{\alpha}(a,h)\}$ for the bundle $\widetilde{\mathbf{F}}$ over $\widetilde{\mathbf{O}}_{b_o}(\delta_{o},\rho_{o})$.
\end{remark}

\v

\v\n
Put $\mathbf{F}=\bigoplus_{i=1}^{\mathfrak{m}}\mathbf{F}(\mathsf{k}_i).$
\v

\subsection{Gluing the finite rank bundle $\widetilde{\mathbf{F}}$}

We recall some results in \cite{LS-2}.
Let $(U,z)$ be a local coordinates on $\Sigma$ around a nodal point ( or a marked point) $q$ with  $z(q)=0$ . Let $b=(\mathbf{s},u)\in \widetilde{\mathbf{O}}_{b_o}(\delta_{o},\rho_{o})$ and $e$ be a local holomorphic section of $u^{*}L|_{U}$ with $\|e\|_{G^L}(q)\neq 0$ for $q\in U$.
Then for any $\phi\in   \widetilde{\mathbf F}|_{b}$ we can write
\begin{equation}\label{eqn_phi_loc_hol}
\phi|_{U} =f \left(\frac{dz}{z}\otimes e \right)^{\mathsf{k}},\;\;\mbox{ where } f\in \mathcal{O}(U).
\end{equation}
In terms of the holomorphic cylindrical coordinates $(s,t)$ defined by $z=e^{-s+2\pi\sqrt{-1}t}$ we can re-written \eqref{eqn_phi_loc_hol} as
$$
\phi(s,t)|_{U}=f(s,t) \left((ds+2\pi\sqrt{-1}dt)\otimes e \right)^{\mathsf{k}},
$$
where $f(z)\in \mathcal O(U)$. It is easy to see that $|f(s,t)-f(-\infty, t)|$ uniformly exponentially converges to 0 with respect to $t\in S^1$ as
$|s|\to \infty$.
\v
For any $\zeta\in C^{\infty}_{c}(\Sigma,\widetilde{\mathbf L}|_{b})$ and
any section $\eta \in
C^{\infty}_{c}(\Sigma, \widetilde{\mathbf L}|_{b}\otimes\wedge^{0,1}_{j}T^{*}\Sigma)$ we
define weighted norms $\|\zeta\|_{j,k,2,\alpha}$ and $\|\eta\|_{j,k-1,2,\alpha}$. Denote by $W^{k,2,\alpha}(\Sigma;\widetilde{\mathbf L}|_{b})$ and
$W^{k-1,2,\alpha}(\Sigma, \widetilde{\mathbf L}|_{b}\otimes \wedge^{0,1}_{j}T^{*}\Sigma)$ the complete spaces with respect to the norms respectively. We also define the space ${\mathcal W}^{k,2,\alpha}(\Sigma;\widetilde{\mathbf L}|_{b})$.
\v
Let $(\Sigma, j, {\bf y})$ be a marked nodal Riemann surface of genus $g$ with $n$ marked points. Suppose that $\Sigma$ has $\mathfrak{e}$ nodal points $\mathbf{p}=(p_{1},\cdots,p_{\mathfrak{e}})$ and $\iota$ smooth components. We fix a local coordinate system $\mathbf{s}\in\mathbf{A}$ for the strata of $\overline{\mathcal{M}}_{g,n}$, where
$\mathbf{A}=\mathbf{A}_1\times \mathbf{A}_2\times...\times \mathbf{A}_\iota$. Let $b_o=(\mathbf{s},u)$ where $u:\Sigma\to M$ be $(j,J)$-holomorphic map. For each node $p_i$ we can glue $\Sigma$ and $u$ at $p_i$ with gluing parameters $(\mathbf{r})=((r_1,\tau_1),...,(r_\mathfrak{e}, \tau_\mathfrak{e}))$ to get $\Sigma_{(\mathbf{r})}$ and $u_{(\mathbf{r})}$, then we glue $\widetilde{\mathbf{F}}\mid_b$ to get $\widetilde{\mathbf{F}}\mid_{b_{(\mathbf{r})}}$.
Denote $|\mathbf r|=\min_{i=1}^{\mathfrak{e}}|r_{i}|.$
\begin{lemma}\label{aright_inverse_after_gluing}
$D^{\tilde{\mathbf L}}|_{b_{(\mathbf{r})}}$ is surjective for $|\mathbf{r}|$ large enough. Moreover, there is a $G_{b_{(\mathbf{r})}}$-equivariant right inverse $ Q^{\tilde{\mathbf{L}}}_{b_{(\mathbf{r})}}$
such that
\begin{equation}
\label{right_estimate}
\| Q^{\tilde{\mathbf{L}}}_{b_{(\mathbf{r})}}\|\leq  \mathsf{C}
\end{equation}
 for some constant $ \mathsf{C}>0$ independent of $(\mathbf{r})$.
\end{lemma}

\begin{lemma}\label{lem_est_I_r}
{\bf (1)} $I^{\widetilde{\mathbf{L}}}_{(r)}: \ker D^{\widetilde{\mathbf L}}|_{b_{o}}\longrightarrow \ker D^{\widetilde{\mathbf L}}|_{b_{(r)}}$ is a $\frac{|G_{b_{o}}|}{|G_{b_{(r)}}|}$-multiple covering map for $r_i$, $1\leq i \leq \mathfrak{e}$, large enough, and
$$\|I^{\widetilde{\mathbf L}}_{(\mathbf{r})}\|\leq \mathsf C ,$$
for some constant $\mathsf C>0$ independent of $(\mathbf{r})$.
\v
{\bf (2)} $I^{\widetilde{\mathbf{L}}}_{(r)}$ induces a isomorphism $I^{\mathbf{L}}_{(r)}: \ker D^{\mathbf L}|_{b_{o}}\longrightarrow \ker D^{\mathbf L}|_{b_{(r)}}$.
\end{lemma}

For fixed $(\mathbf{r})$ we consider the family of maps:
$$
\mathcal {F}_{(\mathbf{r})}: \mathbf{A}\times   W^{k,2,\alpha}(\Sigma_{(\mathbf{r})},u^{\star}_{(\mathbf{r})}TM)
\times {\mathcal W}^{k,2,\alpha}(\Sigma_{(\mathbf{r})},\widetilde{\mathbf L}|_{b_{(\mathbf{r})}})\to W^{k-1,2,\alpha}(\Sigma_{(\mathbf{r})},\wedge^{0,1}T\Sigma_{(\mathbf{r})}\otimes \widetilde{\mathbf L}|_{b_{(\mathbf{r})}})$$
defined by
\begin{equation}\label{def_F_r}
\mathcal {F}_{(\mathbf{r})}(\mathbf{s},h,\xi)= P^{\widetilde{\mathbf{L}}}_{b,b_{(\mathbf{r})}}
\circ D^{\widetilde{\mathbf L}}_{b}\circ (P^{\widetilde{\mathbf{L}}}_{b,b_{(\mathbf{r})}})^{-1}\xi,
\end{equation}
where $b=((\mathbf r),\mathbf{s},v_{\mathbf r})$ and $v_{\mathbf r}=\exp_{u_{(\mathbf r)}}h$.
By implicit function theorem we have

\v\v
\begin{lemma}\label{gluing}
There exist $\delta>0$, $\rho>0$ and a small neighborhood $\widetilde O_{(\mathbf{r})}$ of $0 \in \ker\;D^{\widetilde{\mathbf L}}|_{b_{(\mathbf{r})}}$ and a unique smooth map
$$f^{\widetilde{\mathbf L}}_{(\mathbf{r})}: \widetilde{\mathbf{O}}_{b_{(\mathbf{r})}}(\delta,\rho)
\times \widetilde{O}_{(\mathbf{r})}\rightarrow W^{k-1,2,\alpha}(\Sigma_{(\mathbf{r})},\wedge^{0,1}T\Sigma_{(\mathbf{r})}\otimes \widetilde{\mathbf L}|_{b_{(\mathbf{r})}})$$ such that for any $(b,\zeta)\in \widetilde{\mathbf{O}}_{b_{(\mathbf{r})}}(\delta,\rho)
\times \widetilde{O}_{(\mathbf{r})}$
\begin{equation*}
D^{\widetilde{\mathbf L}}_{b}\circ(P^{\widetilde{\mathbf{L}}}_{b,b_{(\mathbf{r})}})^{-1}\left(\zeta + Q^{\widetilde{\mathbf L}}_{b_{(\mathbf{r})}}\circ f^{\widetilde{\mathbf L}}_{\mathbf{s},h,(\mathbf{r})}(\zeta)\right)=0.
\end{equation*}
\end{lemma}
\v Together with  $I^{\mathbf L}_{(\mathbf{r})}$ we have gluing map
$$Glu^{\mathbf L}_{(\mathbf r)}:\mathbf{F}\mid_{[b_o]}\to \mathbf{F}\mid_{[b]}\;\;\;for\; any \;[b]\in \;\mathbf{O}_{[b_{(\mathbf{r})}]}(\delta,\rho)$$ defined\;by
$$Glu^{\mathbf L}_{(\mathbf r)}([\zeta]):=\left[(P^{\widetilde{\mathbf{L}}}_{b,b_{(\mathbf{r})}})^{-1}\left(I^{\widetilde{\mathbf L}}_{(\mathbf{r})}\zeta + Q^{\widetilde{\mathbf L}}_{b_{(\mathbf{r})}}\circ f^{\widetilde{\mathbf L}}_{\mathbf{s},h,(\mathbf{r})}I^{\widetilde{\mathbf L}}_{(\mathbf{r})}\zeta \right)\right],\;\;\;\;\forall [\zeta]\in \mathbf{F}\mid_{[b_o]}.$$

Given a frame $e_{\alpha}(z)$ on $\widetilde{\mathbf{F}}\mid_{b_o}$, $1\leq \alpha\leq rank\; \widetilde{\mathbf{F}},$ as Remark \ref{isotropy group} we have a $G_{b_o}$-equivariant frame field
\begin{equation}\label{equivariant frame field}
e_{\alpha}((\mathbf{r}),\mathbf{s}, h)(z)=(P^{\widetilde{\mathbf{L}}}_{b,b_{(\mathbf{r})}})^{-1}
\left(I^{\widetilde{\mathbf L}}_{(\mathbf{r})}e_{\alpha}  + Q^{\widetilde{\mathbf L}}_{b_{(\mathbf{r})}}\circ f^{\widetilde{\mathbf L}}_{\mathbf{s},h,(\mathbf{r})}I^{\widetilde{\mathbf L}}_{(\mathbf{r})}e_{\alpha} \right)(z)\end{equation}
over $D_{R_{0}}^{*}(0)\times \widetilde{\mathbf{O}}_{b_o}(\delta_{o},\rho_{o})$,
where $z$ is the coordinate on $\Sigma$, and
$$D_{R_{0}}^{*}(0):=\bigoplus_{i=1}^{\mathfrak{e}}\left\{(r, \tau)\mid R_0< r<\infty, \;\tau\in S^1\right\}.$$
For any fixed $(\mathbf{r})$, $e_{\alpha}$ is smooth with respect to $\mathbf{s},h$ over $\widetilde{\mathbf{O}}_{b_o}(\delta_{o},\rho_{o})$.

\v
To discuss the smoothness with respect to $(\mathbf{r}),\mathbf{s},h$ we need to fix a Riemann surface $\Sigma_{(\mathbf{R_o})}$.
Let $\alpha_{(r_{i})}:[0,2r_{i}]\to [0,2R_{0}]$ be a smooth function satisfying
$$
  \alpha_{(r_{i})}(s)=\left\{
\begin{array}{ll}
s\;\;\;\; \;   & if\; s\in [0,\frac{R_{0}}{2}-1] \\
 \frac{R_{0}}{2}+\frac{R_{0}}{2r_{i}-R_{0}}(s-R_{0}/2) \;\;\;\;\;\;if  &  s\in [R_{0}/2,2r_{i}-R_{0}/2]   \\
 s-2r_{i}+2R_{0} \; & if\;s\in [2r_{i}-\frac{R_{0}}{2}+1,2r_{i}]
\end{array}
\right.
$$
Set $\alpha_{(r_{i})}:[-2r_{i},0]\to [-2R_{0},0]$ by $\alpha_{(r_{i})}(s)=-\alpha_{(r_{i})}(-s).$
Let $(s_{1}^{i},t_{1}^{i})$ and $(s_{2}^{i},t_{2}^{i})$ be cusp cylinder coordinates around $p_{i}$, thus $z_{i}=e^{-s_{1}^{i}-2\pi \sqrt{-1}t_{1}^{i}}$ and $w_{i} =e^{ s_{2}^{i}+2\pi \sqrt{-1}t_{2}^{i}}$.  Denote $$W_{i}(R)=\{ |s_{1}^{i}|>R \}\cup \{ |s_{2}^{i}|>R \}.$$ Obviously, $W(R)=\cup_{i=1}^{\mathfrak{e}}W_{i}(R).$
We can define a map $\varphi_{(\mathbf{r})}:\Sigma_{(\mathbf{r})}\to \Sigma_{(\mathbf{R}_{0})}$ as follows:
$$\varphi_{(\mathbf{r})}=\left\{
\begin{array}{ll}
p, & p\in \Sigma(R_{0}/4).\\
(\alpha_{(r_{i})}(s_{i}),t_{i})   \;\;\;\;\;\;\;& (s_{1}^{i},t_{1}^{i})\in W_{i}(R_{0}/4),\;i=1,\cdots,\mathfrak{e}.
\end{array}
\right.
$$
Then we obtain a family of Riemann surfaces $\left(\Sigma_{(\mathbf{R}_{0})},(\varphi_{(\mathbf{r})}^{-1})^{*}j_{\mathbf{r}} ,\varphi_{(\mathbf{r})}^{-1}(\mathbf{y})\right)$.
Denote $u^{\circ}_{(\mathbf{r})}:=u_{(\mathbf{r})}\circ \varphi_{\mathbf{r}}^{-1}.$

In \cite{LS-2} we have proved the following lemma.
\begin{lemma}\label{smooth line}
There exists positive  constants  $ \mathsf{d},R$ such that
for any  $h\in W^{k,2,\alpha}\left(\Sigma_{(R_{0})},(u_{(R_{0})})^*TM \right),$  $\zeta\in \ker D^{\widetilde{\mathbf L}}|_{b_{o}}$    with    $$\|\zeta\|_{\mathcal W,k,2,\alpha}\leq \mathsf{d},\;\;\;\;\;\|h-\hat h_{(\mathbf r)}\|< \mathsf{d},\;\;\;\;\;|\mathbf r|\geq R, $$
$( \varphi_{\mathbf r}^{-1})^{*}(Glu^{\widetilde{\mathbf L}}_{\fs,(\mathbf r),h'}(e_{\alpha}) )$ is smooth with respect to $(\mathbf{s}, (\mathbf{r}),h)$ \; for any $e_{\alpha}\in \ker D^{\widetilde{\mathbf L}}|_{b_o}$,  where  $h'=(\exp_{u_{(\mathbf r)}}^{-1}\circ (\exp_{u_{(\mathbf R)_{0}}}(h)\circ \varphi_{(\mathbf{r})})$. In particular $Glu^{\widetilde{\mathbf L}}_{\fs,(\mathbf r),h'}(e_{\alpha})\mid_{\Sigma(R_0)}$
 is smooth.
\end{lemma}

\subsection{\bf Global regularization and virtual neighborhoods}\label{global_regu}
\v

 We are going to construct a bundle map $\mathfrak{i}:\mathbf{F}\rightarrow  \E$. We first define a bundle map
$\mathfrak{i}:  {\mathbf{F}}(\mathsf{k}_{i})\rightarrow \mathcal{E}$.
Consider two different cases:
\v
{\bf Case 1.} $[b_i]$ lies in the top strata $\mathcal{M}_{g,n}(A)$. Denote $b_o=b_i$. Choose a local coordinate system $(\psi, \Psi)$ for $\mathcal{Q}$ and a local model $
\widetilde{\mathbf{O}}_{b_o}(\delta_{b_{o}},\rho_{b_{o}})/G_{b_o}$ around $[b_o]$. We have an isomorphism
\begin{equation}\label{trivialization-1}
P_{b_o,b}=\Phi\circ\Psi_{j_o,j_a}:\widetilde{\E}_{b_o}\rightarrow \widetilde{\E}_{b},\;\;\forall \;\;b\in \widetilde{\mathbf{O}}_{b_{o}}(\delta_{b_{o}},\rho_{b_{o}}).
\end{equation}
To simplify notations we denote  $\widetilde{\mathbf{F}}(\mathsf{k}_i)=\widetilde{H}$, $P_{b_o,b}=P$ in this section.
\v
Choosing a base $\{e_{\alpha}\}$ of the fiber $\widetilde{H}\mid_{b_o}$, by Lemma \ref{smoothness of bundle-1} we can get a smooth frame fields $\{e_{\alpha}\}$ for the bundle $\widetilde{H}$ over $\widetilde{\mathbf{O}}_{b_o}(\delta_{o},\rho_{o}),$ which
induces another isomorphism
\begin{equation}\label{trivialization-2}
Q:\widetilde{H}\mid_{b_o}\to \widetilde{H}\mid_{b},\;\;\;\forall \;\;b\in \widetilde{\mathbf{O}}_{b_{o}}(\delta_{b_{o}},\rho_{b_{o}})
\end{equation}
\begin{equation}\label{trivialization-2}
\sum c_{\alpha}e_{\alpha}\mid_{b_o}\longmapsto \sum c_{\alpha}e_{\alpha}\mid_{b}.
\end{equation}

\v\n
Let $\rho_{\widetilde K_{b_o}}: G_{b_o}\rightarrow GL(\widetilde{K}_{b_o})$ be the natural linear representation, and let $\rho_{\mathbb{R}}: G_{b_o}\rightarrow GL(\mathbb{R}[G_{b_o}])$ be the standard representation. Both $\widetilde{K}_{b_o}$ and $\widetilde{H}\mid_{b_o}$ can be decomposed as sum of irreducible representations.
Without loss of generality we assume that $\rho_{\widetilde K_{b_o}}$ is an irreducible representation.
Let $\eta_1,...,\eta_l$ be a base of $\widetilde{K}_{b_o}$, let $\widetilde{H}\mid_{b_o}= \bigoplus_{i=1}^m E_i$ be the decomposition of irreducible representations such that $E_1$ has base $e_1,...,e_l.$ Define map $em(\eta_i)=e_i$, $i=1,...,l$. Thus we have map $p:\widetilde{H}\mid_{b_o}\rightarrow \widetilde{K}_{b_o}$ with $p\cdot em=id$.
\v
Let $\mathbb R^+=\{x\in \mathbb R|x\geq 0\}$ and $f_{\delta_{o},\rho_{o}}:\mathbb R^+\times \mathbb R^+\rightarrow \mathbb R^+$ be a smooth cut-off function such that
\[
f_{\delta_{o},\rho_{o}}(x,y)=\left\{
\begin{array}{ll}
1 \;\;\;\;\; on \;\;\{(x,y)|\;0\leq x\leq  \delta_o/3 ,\;0\leq y\leq \rho_o/3 \},    \\  \\
0 \;\;\;\;\;
on \;\;    \{(x,y)|\;x\geq  2\delta_o/3 \}\bigcup\{(x,y)|\; y\geq  2\rho_o/3 \}.
\end{array}
\right.
\]
We define a  cut-off function $\alpha_{b_o}: \widetilde{\mathbf{O}}_{b_o}(\delta_{b_{o}},\rho_{b_{o}}) \to [0, 1]$
by
\begin{equation}\label{cut-off-1}
\alpha_{b_o}(b)=f_{\delta_{b_{o}},\rho_{b_{o}}}(d_{\mathbf{A}}^{2}(a_o,a),\|h\|_{j_{a},k,2}^{2}).
\end{equation}
For any $\kappa\in \widetilde H\mid_{b}$ with $b\in \widetilde{\mathbf{O}}_{b_o}(\delta_{b_{o}},\rho_{b_{o}})$, in terms of
the local coordinate system $(\psi, \Psi)$, we define
\[
\mathfrak{i}(\kappa,b)_{b_o}=\left\{
\begin{array}{ll}
\alpha_{b_o}(b)P\circ p\circ Q^{-1} (\kappa) \;\;\;\;\; \mbox{ if }\;\|h\|_{j_{a},k,2} < \rho_{b_{o}}, \mbox{ and } d_{\mathbf{A}}^{2}(a_{o},a)<  \delta_{b_{o}}\\  \\
0 \;\;\;\;\;otherwise.
\end{array}
\right.
\]
\begin{lemma}\label{smooth of bundle map}
In the local coordinates $(\psi,\Psi)$ on $U$ and in  $\widetilde{\mathbf{O}}_{b_o}(\delta_o,\rho_o)$ the bundle map $\mathfrak{i}(\kappa,b)_{b_o}: \widetilde{\mathbf{F}}(\mathsf{k}_i)\rightarrow \widetilde{\mathcal{E}}$ is smooth with respect to $(\kappa, a, h)$.
\end{lemma}
\v\n
{\bf Proof.} By Lemma \ref{smoothness of norms} we immediately obtain that the cut-off function $\alpha_{b_o}(b)$ is a smooth function. Note that, in the local coordinates $(\psi,\Psi)$,
$P$, $p$ and $Q^{-1}$ are smooth. We conclude that $\mathfrak{i}(\kappa,b)_{b_o}$ is a smooth function of $(\kappa, a,h)$.
 $\Box$
\v\v

We can transfer the definition to other local coordinate system $(\psi', \Psi')$ and local model $\widetilde{\mathbf{O}}_{b_o}(\delta'_{b_{o}},\rho'_{b_{o}})$.
Suppose that in the coordinate system $(\psi, \Psi)$
$$b_o=(a_o, u_o),\;\;b=(a, v),\;\;v=\exp_{u_o}h,$$
and in the coordinate system $(\psi', \Psi')$
$$b'_o=(a_o', u_o'),\;\;b'=(a', v'),\;\;v'=\exp_{u'_o}h',\;\;where \;[b]=[b'].$$
We have
$$(\psi'\circ \psi^{-1}, \Psi'\circ \Psi^{-1})\cdot(a,v)=(a', v'),\;\;\;a'=\psi'\circ \psi^{-1}(a),\;\;v'=v\circ (\Psi'\circ \Psi^{-1})\mid_a.$$
$$(\psi'\circ \psi^{-1}, \Psi'\circ \Psi^{-1})\cdot(a_o,u_o)=(a_o', u_o'),\;\;\;a_o'=\psi'\circ \psi^{-1}(a),\;\;u_o'=u_o\circ (\Psi'\circ \Psi^{-1})\mid_{a_o}.$$
$(\psi'\circ \psi^{-1}, \Psi'\circ \Psi^{-1})$ send $e_{\alpha}$ to $e'_{\alpha}$. Then $(\Psi'\circ \Psi^{-1})\mid_{a}$
induces an isomorphism $\varphi_a:\widetilde{H}\mid_{(a,v)}\to \widetilde{H}'\mid_{(a',v')}$.
In $(\psi', \Psi')$ we have isomorphism $$Q':\widetilde{H}'\mid_{(a'_o,u'_o)}\to \widetilde{H}'\mid_{(a',v')},\;\;\forall b\in \widetilde{O}_{b_{o}}(\delta'_{b_{o}},\rho'_{b_{o}}),$$
$$Q'=\varphi_a \circ Q\circ \varphi_{a_o}^{-1}.$$
\v\n
We have chosen a finite dimensional subspace $\widetilde{K}_{(a,v)} \subset \widetilde{\E}|_{(a,v)}$ in $(\psi,\Psi)$. Denote $\vartheta_{a}=(\Psi'\circ \Psi^{-1})\mid_{a}.$ Define $\widetilde{K}'_{(a',v')}=\{\kappa\circ  d\vartheta_{a}^{-1}|\;\;\forall \;\kappa\in \widetilde{K}_{(a,v)}\}$.  Then $(\Psi'\circ \Psi^{-1})|_{a}$ induces a map
\begin{equation}\label{K isomorphic}
\phi_{a}:\widetilde{K}_{(a,v)}\to \widetilde{K}'_{(a',v')},\;\;\;
\phi_{a}(\kappa)= \kappa\circ  d\vartheta_{a}^{-1} ,\;\;\; \forall \kappa \in \widetilde{K}_{(a,v)}.
\end{equation}
Denote $\kappa'= \phi_{a}(\kappa)$.
Define
$$P':\widetilde{\E}'_{(a_o', u_o')}\rightarrow \widetilde{\E}'_{(a',v')},\;\;\mbox{ by }\; P'=\phi_a \circ P\circ \phi_{a_o}^{-1},$$
and
$$p':\widetilde{H}'\mid_{(a_o',u_o')}\rightarrow \widetilde{K}'_{(a'_o,u'_o)},\;\;\mbox{ by }\;p'=\phi_{a_o}\circ p\circ \varphi^{-1}_{a_o}.$$
$(\Psi'\circ \Psi^{-1})|_{a}$ also induces a map  $$\lambda_{a}: G_{(a_o,u_o)}\to G_{(a_o',u_o')}\;\;\;g\longmapsto g'=d\vartheta_{a}\circ g\circ (d\vartheta_{a})^{-1}.$$
It is easy to check that $\rho_{\widetilde K_{(a_o,u_o)}}: G_{(a_o,u_o)}\rightarrow GL(\widetilde{K}_{(a_o,u_o)})$ and
$\rho_{\widetilde K_{(a'_o,u'_o)}}: G_{(a'_o,u'_o)}\rightarrow GL(\widetilde{K}'_{(a'_o,u'_o)})$ are equivariant. Let
$$\eta'_i=\phi_{a}(\eta_i),\;\;e'_i=\varphi_{a}(e_i),\;\;em'(\eta'_i)=e'_i,\;\;i=1,2,...,l.$$
Then $em'(\widetilde{K}'_{(a'_o,u'_o)})=span\{e'_1,...e'_l\}\subset \widetilde{H}'\mid_{(a'_o,u'_o)}$. In the coordinate system $(\psi', \Psi')$
we define
\[
\mathfrak{i}(\kappa',b')_{b'_o}=\left\{
\begin{array}{ll}
\alpha_{b'_o}(b')P'\circ p'\circ (Q')^{-1} (\kappa') \;\;\;\;\; \mbox{ if }\;\|h\|_{j_{a'},k,2 }< \rho_{b'_{o}}, \mbox{ and } d_{\mathbf{A'}}^{2}(a'_{o},a')<  \delta_{b'_{o}}\\  \\
0 \;\;\;\;\;otherwise.
\end{array}
\right.
\]
We have
\begin{equation}\label{coord trans forbundle map}
\mathfrak{i}(\kappa',b')_{b'_o}=
\phi_a\circ\mathfrak{i}(\kappa,b)_{b_o}\circ\varphi^{-1}_a.
\end{equation}
If we choose three local coordinate systems $(\psi, \Psi)$,
$(\psi', \Psi')$ and $(\psi'', \Psi'')$,
since
$$
(\Psi\circ (\Psi'')^{-1})\circ(\Psi''\circ (\Psi')^{-1})\circ(\Psi'\circ \Psi^{-1})=Id,
$$
one can easily check that
\begin{equation}\label{coord trans equality}
\phi''_{a''}\phi'_{a'}\phi_a=Id,\;\;\; \varphi''_{a''}\varphi'_{a'}\varphi_a=Id.
\end{equation}
It follows from \eqref{coord trans forbundle map} and \eqref{coord trans equality} that the bundle map
$\mathfrak{i}:  {\mathbf{F}}(\mathsf{k}_i)\rightarrow \mathcal{E}$ is well defined. Obviously, $\mathfrak{i}([\kappa_{i},b])=[\mathfrak{i}(\kappa_{i},b)].$
\v\v
\begin{remark}\label{restri}
Let $(\psi', \Psi')$ be a local coordinate system in $\mathbf{O}_{[b'_o]}(\delta'_{[b'_{o}]},\rho'_{[b'_{o}]})\subset \mathbf{O}_{[b_o]}(\delta_{[b_{o}]},\rho_{[b_{o}]})$ such that $[b_o] \notin  \mathbf{O}_{[b'_o]}(\delta'_{[b'_{o}]},\rho'_{[b'_{o}]})$. The restriction of $[\mathfrak{i}(\kappa,b)_{b_o}]$ to $\mathbf{O}_{[b'_o]}(\delta'_{[b'_{o}]},\rho'_{[b'_{o}]})$ is a element in
$\E|_{\mathbf{O}_{[b'_o]}(\delta'_{[b'_{o}]},\rho'_{[b'_{o}]})}$. We can transfer it to $(\psi', \Psi')$ by \eqref{K isomorphic}.
\end{remark}

\v\v
{\bf Case 2.} $[b_i]$ lies in lies in a lower strata.
We choose $(\fs,\ft)$ coordinates. Put $\ft_{i}=e^{-2r_{i}-2\pi\tau_{i}}$, sometimes we use $(\fs, (\mathbf{r}))$ coordinates, where $(\mathbf{r})=((r_1,\tau_1),...,(r_{\mathfrak{e}},\tau_{\mathfrak{e}}))$.
Denote $b_{o}=b_i=(0,0,u)$, $\mathbf{F}(\mathsf{k}_i)=H(\fs,\ft)$, $\mathbf{F}(\mathsf{k}_i )\mid_{b_i}=H(0,0)$. We choose $|\fs|$, $|
\ft|$ small enough. In terms of $(\fs,\ft)$ we have an isomorphism

$$P:\bar{\E}_{b_o}\rightarrow \bar{\E}_{b},\;\;\forall \;\;b\in \widetilde{\mathbf{O}}_{b_{o}}(\delta_o,\rho_o).$$
Denote $\bar{H}=\{\zeta\mid_{\Sigma(R_0)} \mid \zeta \in \widetilde{H}\}$. Choosing a base $\{e_{\alpha}\}$ of the fiber $\bar{H}\mid_{b_o}$, by \eqref{equivariant frame field}
we can get a frame fields $\{e_{\alpha}((\mathbf{r}), a, h)\mid_{\Sigma(R_0)}\}$ for the bundle $\bar{H}$ over $\widetilde{\mathbf{O}}_{b_o}(\delta_{o},\rho_{o}).$
We have another isomorphism in the $(\fs,\ft)$ coordinates
$$ Q:\bar{H}(0,0)\to \bar{H}(\fs,\ft),\;\;\;\forall \;\;b\in \widetilde{\mathbf{O}}_{b_{o}}(\delta_o,\rho_o).$$
Denote $\mathbf{O}(\delta_o)=\{p\in \overline{\mathcal{M}}_{g,n}\mid d^2_{\mathsf{wp}}(0,p)< \delta_o\}$. Since $\overline{\mathcal{M}}_{g,n}$ has a natural effective orbifold structure, we can choose a smooth cut-off function in orbifold sense $\beta^\bullet_{\delta_{o}}: \mathbf{O}(\delta_o) \to [0, 1]$ such that
$$\beta^\bullet_{\delta_{o}}|_{\mathbf{O}(\delta_o/3)}=1,\;\;\;\;\beta^\bullet_{\delta_{o}}|_{\mathbf{O}(\delta_o)\setminus\mathbf{O}(2\delta_o/3)}=0. $$

\v\n
We define a  cut-off function $\alpha_{b_o}: \widetilde{\mathbf{O}}_{b_o}(\delta_o,\rho_o) \to [0, 1]$
by
\begin{equation}\label{cut-off-2}
\alpha_{b_o}(b)=f_{\delta_o,\rho_o}(\beta^\bullet_{\delta_{o}}(\fs,\ft) ,\|\beta_{R_{0}}h
\|_{j_{\fs,\ft},k,2}^{2}),\end{equation}
where $\beta_{R_{0}}$ is the function in \eqref{def_beta1}.
Using $\alpha_{b_o}(b)$ defined in \eqref{cut-off-2},
we can define the bundle map
$\mathfrak{i}:  {\mathbf{F}}(\mathsf{k}_i)\rightarrow \bar{\mathcal{E}}$ by
\[
\mathfrak{i}(\kappa,b)_{b_o}=\left\{
\begin{array}{ll}
\alpha_{b_o}(b)P\circ p\circ Q^{-1} (\kappa) \;\;\;\;\; \mbox{ if }\;\|h\|_{j_{a},k,2} < \rho_{b_{o}}, \mbox{ and }
\beta_{\delta_{o}}(\fs,\ft)<  \delta_{b_{o}}\\  \\
0 \;\;\;\;\;otherwise.
\end{array}
\right.
\]
For any fixed $(\mathbf{r})$, $\mathfrak{i}(\kappa,b)_{b_o}$
and $Q$ are smooth with respect to $(\mathbf{s},h)$ in the coordinates $(\mathbf{s},(\mathbf{r}))$. In order to study the smoothness with respect to $(\mathbf{r})$ we
note that $\mathfrak{i}(\kappa,b)_{b_o}$ is supported in $\Sigma(R_0)$. For any $v=\exp_{u_{(\mathbf{r})}}h$, we let   $$h^\circ=\left((h-\hat h_{0})(s_1,t_1)\beta_{1;2}(s_1),
(h-\hat h_{0})(s_2,t_2)\beta_{2;2}(s_{2})\right),$$
where
$$h_0=\int_{ {S}^1}h(r,t)dt.$$
Denote $v^\circ=\exp_u h^\circ$.
We can view $\bar{\E}\mid_v$ to be $\bar{\E}\mid_{v^\circ}$.
Then we view $P$ to be a family of operators in $\E$ over $W^{k,2}(\Sigma;u^{\ast}TM)$, where $\E\to W^{k,2}(\Sigma;u^{\ast}TM)$ is independent of $(\mathbf{r})$.
Consider the map
$$\mathfrak{i}(\kappa,b)_{b_o}\circ Q:  \bar{H}(0,0)\times \mathbf{A}\times D_{R_{0}}^{*}(0)\times W^{k,2}(\Sigma;u^{\ast}TM)
\rightarrow \E$$
$$\mathfrak{i}(\kappa,b)_{b_o}\circ Q(\kappa, \mathbf{s},(\mathbf{r}),h)  =\alpha_{b_o}(\mathbf{s},(\mathbf{r}),v)P\circ p(\kappa).
$$

 \v
\begin{lemma}\label{smooth of bundle map-1}
In the local coordinates $(\mathbf{s},(\mathbf{r}))$, the bundle map $\mathfrak{i}(\kappa,b)_{b_o}\circ Q$ is smooth with respect to $(\kappa, \mathbf{s}, (\mathbf{r}), h)$ in  $\widetilde{\mathbf{O}}_{b_o}(\delta_o,\rho_o)$.
\end{lemma}
\v\n
{\bf Proof.} $\alpha_{b_o}(\mathbf{s},(\mathbf{r}),v)$ is smooth
with respect to $(\mathbf{s}, (\mathbf{r}), h)$. For any $l\in Z^{+},$ denote $b_{\ft}=(\mathbf{s},\exp_{u}(h+\sum_{i=1}^{l}t_{l}h_{l} ))$
 and
$$
T^{l}(h;h_{1},\cdots,h_{l})=\nabla_{t_{1}}\cdots\nabla_{t_{l}}
\left.\left(P_{b_o,b_{\mathbf{t}}}\right)\right|_{\ft=\mathbf 0}
$$
By the same method as in the proof of Lemma 3.1 of \cite{LS-2} we can show that
$T^{l}(h;\cdot\cdot\cdot)$ is a bounded linear operator. The proof is complete.  \;\;$\Box$

\v
By {\bf Case 1}, {\bf Case 2} we have defined $\mathfrak{i}([\kappa_i,b])_i$ for all $i=1,...,\mathfrak{m}$. Set
$$\mathfrak{i}([\kappa,b])=\sum_{l=1}^{\mathfrak{m}} \mathfrak{i}([\kappa_l,b])_{l}\;\;
for\; any\; \kappa=(\kappa_1,...,\kappa_{\mathfrak{m}})\in \mathbf{F}\mid_{b}.$$ Then $\mathfrak{i}:\mathbf {F}\to \mathcal{E}$ is a bundle map.
We define a global regularization to be the bundle map $\mathcal{S}:\mathbf{F}\to  \E$
$$\mathcal{S}([\kappa,b])
=[\bar{\partial}_{j,J}v] + \mathfrak{i}([\kappa,b]).
$$
It is obvious that $D\mathcal{S}$ is surjective.
 Denote $\mathsf{p}:\mathbf F\to \mathcal U $ by the projection of the bundle. Set
$$\mathbf{U}=\mathcal{S}^{-1}(0)|_{\mathsf{p}^{-1}(\mathcal{U})}.$$
By restricting the bundle $\mathsf{p}^*\mathbf{F}$ to $\mathbf{U}$ we have a bundle $\mathsf{p}:\mathbf{E}\to \mathbf{U}$ of finite rank with a canonical section $\sigma$ defined by
$$\sigma([(\kappa,b)])=( [((\kappa,b)  , \kappa)]),\;\;\;\;\forall \;[(\kappa,b)]\in \mathbf {U}.$$ We call
$$(\mathbf{U},\mathbf{E},\sigma),$$  a virtual neighborhood for $\overline{\mathcal{M}}_{g,n}(A)$.
\v

\v
\section{\bf Smoothness of the top strata}\label{top strata}

{\bf Proof of Theorem \ref{Smooth}}
\v\n
The proof is divided into two steps, the subsections
\S\ref{smoothness-1} and \S\ref{smoothness-2}.
\v
\subsection{\bf Smoothness}\label{smoothness-1}
\v
\v
Let $[(\kappa_{o},b_{o})]\in \mathbf{U}^T$. To simplify notations we consider the following case, for the general case the argument are the same. We assume that
$$[b_o]\in \mathbf{O}_{[b_{1}]}(2\delta_1/3,2\rho_1/3)\bigcap
\mathbf{O}_{[b_{2}]}(2\delta_2/3,2\rho_2/3)$$
and
$$[b_o]\notin \overline{\mathbf{O}}_{[b_{i}]}(2\delta_i/3,2\rho_i/3)\;\;\;\forall i=3,...,\mathfrak{m}.$$
We choose a local coordinate system  $(\psi, \Psi)$ for $\mathcal{Q}$ and local model $
\widetilde{\mathbf{O}}_{b_o}(\delta_o,\rho_o)/G_{b_o}$ around $b_o$. Let $b_o=(a_o,u)$, and let $\widetilde{\mathbf{U}}^T$ be the local expression of $\mathbf{U}^T$ in terms of $(\psi, \Psi)$. We choose $(\delta_o,\rho_o)$ so small that
$$\mathbf{O}_{[b_o]}(\delta_o,\rho_o)\notin \mathbf{O}_{[b_{i}]}(2\delta_i/3,2\rho_i/3)\;\;\;\forall i=3,...,\mathfrak{m}.$$
Then we only need to consider the bundles ${\mathbf{F}}(\mathsf{k}_{1})$ and ${\mathbf{F}}(\mathsf{k}_{2})$.
We consider two different cases.
\v
{\bf Case 1.} Both $[b_1]$ and $[b_2]$ lie in the top strata. By Remark \ref{finite covering} we may assume that
both $\mathbf{O}_{[b_{1}]}(2\delta_1/3,2\rho_1/3)$ and $\mathbf{O}_{[b_{2}]}(2\delta_2/3,2\rho_2/3)$ lie in the top strata. Let
$$b_1=(a_1, u_1)\;\;in\;(\psi_1, \Psi_1),\;\;\;\;b_2=(a_2, u_2)\;in \;(\psi_2, \Psi_2).$$
In terms of the coordinate system $(\psi, \Psi)$, let $b=(a,v)\in \widetilde{\mathbf{O}}_{b_o}(\delta_o,\rho_o)$. Suppose that, in the coordinate system $(\psi_1, \Psi_1)$,
$$[b']=[b],\; b'=(a', v'),\;\; v'=\exp_{u_1}h_1,$$
and in the coordinate system $(\psi_2, \Psi_2)$,
$$[b'']=[b],\;\;b''=(a'', v''),\;\; v''=\exp_{u_2}h_2.$$
The bundle maps are given respectively by
$$\mathfrak{i}(\kappa_1,b')_{b_1}=\alpha_{b_1}(b)P_1\circ p_1\circ Q_1^{-1}(\kappa_1): (\widetilde{H}_1)\mid_{b'}\to \widetilde{K}_1\mid_{b'}\;\;\;in\; (\psi_1, \Psi_1),$$
$$\mathfrak{i}(\kappa_2,b'')_{b_2}=\alpha_{b_2}(b)P_2\circ p_2\circ Q_2^{-1}(\kappa_2):(\widetilde{H}_2)\mid_{b''}\to \widetilde{K}_2\mid_{b''}\;in\;(\psi_2, \Psi_2),$$
where $P_1=P_{b_1,b'}$ in $(\psi_1, \Psi_1)$, $P_2=P_{b_2,b''}$ in $(\psi_2, \Psi_2)$. By Lemma \ref{smooth of bundle map}, $\mathfrak{i}(\kappa_1,b)_{b_1}$ in $(\psi_1, \Psi_1)$ ( resp. $\mathfrak{i}(\kappa_2,b)_{b_2}$ in $(\psi_2, \Psi_2)$ ) is smooth with respect to $(\kappa_1,b)$ ( resp. $(\kappa_2,b)$).
\v
We transfer from both the local coordinate systems $(\psi_1, \Psi_1)$ and $(\psi_2, \Psi_2)$ to the coordinates $(\psi, \Psi)$. We have
$$(\psi\circ \psi_1^{-1}, \Psi\circ \Psi_1^{-1})\cdot(a',v')=(a, v),\;\;\;a=\psi\circ \psi_1^{-1}(a'),\;\;v=v'\circ (\Psi\circ \Psi_1^{-1})\mid_{a'},$$
$$(\psi\circ \psi_2^{-1}, \Psi\circ \Psi_2^{-1})\cdot(a'',v'')=(a, v),\;\;\;a=\psi\circ \psi_2^{-1}(a''),\;\;v=v''\circ (\Psi\circ \Psi_2^{-1})\mid_{a''}.$$
The $(\psi\circ \psi_i^{-1}, \Psi\circ \Psi_i^{-1})$, $i=1,2$, induces maps
$$\phi^1_{a'}:\widetilde{K}_1\to \widetilde{K}^\diamond_1,\;\;
\phi^2_{a''}:\widetilde{K}_2\to \widetilde{K}^\diamond_2$$
$$\varphi^1_{a'}:\widetilde{H}_1\to \widetilde{H}^\diamond_1,\;\;\varphi^2_{a''}:\widetilde{H}_2\to \widetilde{H}^\diamond_2.$$
Put
$$\widetilde{H}^\diamond=(\widetilde{H}^\diamond_1)\mid_{b_o}\oplus (\widetilde{H}^\diamond_2)\mid_{b_o},\;\;\kappa=(\kappa_1,\kappa_2)\in \widetilde{H}^\diamond,\;\;(Q_1^\diamond \kappa_1, Q_2^\diamond \kappa_2):=Q^\diamond \kappa.$$
Here $\tilde{H}^\diamond $, $\widetilde{K}^\diamond$ and $Q^\diamond$ denote the spaces and operator in $(\psi, \Psi)$.
By Remark \ref{restri} the bundle map in $(\psi, \Psi)$ becomes
$$\mathfrak{i}(\kappa,b)=\mathfrak{i}(\kappa_1,b')_{b_1}\circ  d\vartheta_{1}^{-1}
 +\mathfrak{i}(\kappa_2,b'')_{b_2}\circ  d\vartheta_{2}^{-1},
$$
where $\vartheta_{1}=(\Psi\circ \Psi_1^{-1})\mid_{a'}$, $\vartheta_{2}=(\Psi\circ \Psi_2^{-1})\mid_{a''}$. The key point is that $\Psi\circ \Psi_i^{-1}$, $i=1,2$, is a family of diffeomorphisms of $\Sigma$ depending on $a$.
For $v\in W^{k,2}$, $\frac{\p}{\p a}(v\circ\Psi\circ \Psi_i^{-1})$ is not in $W^{k,2}$. But for any fixed $a$ , $\mathfrak{i}(\kappa,b)$ and $Q_1^\diamond$, $Q_2^\diamond$ are smooth.
\v\n
Consider the map
$$ F_{(\kappa_o,b_o)}: \mathbf{A}\times \widetilde{H}^\diamond \times W^{k,2}(\Sigma;u^{\ast}TM)
\rightarrow W^{k-1,2}(u^{\ast}TM\otimes \wedge_{j_o}^{0,1})$$
$$F_{(\kappa_o,b_o)}(a,\kappa,h)=P_{b,b_o}\left(\bar{\partial}_{j_a,J}
v + \mathfrak{i}(Q^\diamond\kappa, b)
\right),$$ where  $b=(a,v)$, $v=\exp_{u}(h) $ for some
$h\in W^{k,2}(\Sigma, u^*TM)$. For any  $(a,\kappa ,h)\in F_{(\kappa_o,b_o)}^{-1}(0)$ we have
\begin{equation}\label{PDE}
 \bar{\partial}_{j_a,J}v + \mathfrak{i}(Q^\diamond\kappa,b)=0,
\end{equation}
 where $b=(a,v)$.
For any fixed $a$, it follows from the standard elliptic estimates and the smoothness of $\mathfrak{i}$ that $v\in C^{\infty}(\Sigma,M).$ Then by Lemma \ref{smooth of bundle map} and the smoothness of the frame field $e_{\alpha}$ we
conclude that $\mathfrak{i}\mid_v$ and $Q^\diamond\mid_v$ are smooth with respect to $(a,\kappa,h)$.
It is easy to see that
$F_{(\kappa_o,b_o)}(a, \kappa,h)$ is smooth with respect to $(a,\kappa,h)$. Then we use the implicity theorem with parameter $a$
to conclude that $v$ is smooth with respect to $(a,\kappa,h)$. It follows that $\widetilde{\mathbf{U}}^T\bigcap \mathsf{p}^*\widetilde{\mathbf{O}}_{b_o}(\delta_o,\rho_o)$ is smooth, where $\mathsf{p}:\widetilde{\mathbf{U}}^T\to \widetilde{ \mathcal B}$ is the projection.
\v
\v
{\bf Case 2.} $[b_2]$ lies in the top strata, $[b_1]$ lies in a lower strata. Without loss of generality we assume that $b_1=(\Sigma, j, {\mathbf{y}}, u)$, where $\Sigma$ has one node $q$, $\mathbf{s}_o\in \mathbf{A}=\mathbf{A_1\times \mathbf{A}}_2$. We glue $\Sigma$ at $q$ with gluing parameter $(r)$. We have bundle maps $\mathfrak{i}(\kappa_1,b)_{b_1}=\alpha_{b_1}(b)P_1\circ p_1\circ Q_1^{-1}(\kappa_1)$ and $\mathfrak{i}(\kappa_2,b')_{b_2}=\alpha_{b_2}(b')P_2\circ p_2\circ Q_2^{-1}(\kappa_2)$.
Then we transfer to the coordinates $(\psi, \Psi)$, and choose $(\mathbf{s}, \mathbf{t})$-coordinates. We use Lemma \ref{smooth of bundle map-1} and the same method as in {\bf Case 1} to prove that $v$ is smooth with respect to $(\mathbf{s}, (\mathbf{r}),\kappa,h)$. Then we use Lemma \ref{smooth line} to prove that $Q_1$ is smooth with respect to $(\mathbf{s}, (\mathbf{r}),\kappa,h)$. Then we can prove
the smoothness of $\widetilde{\mathbf{U}}^T\bigcap \mathsf{p}^{*}\widetilde{\mathbf{O}}_{b_o}(\delta_o,\rho_o)$.
\v
The proof of the orientation of $\mathbf{U}^T$ is standard, we omit here.

\subsection{The oribifold structure}\label{smoothness-2}

We introduce a notation. For any $(\kappa_o, b_o)\in \mathbf{U}$ we choose a local coordinate system  $(\psi, \Psi)$ on $U\ni a_o$ and local model $\widetilde{\mathbf{O}}_{b_o}(\delta_o,\rho_o)/G_{b_o}$.
Set
$$\widetilde{\mathbf{U}}_{\kappa_o,b_o}(\varepsilon,\delta_o,\rho_o) =\left\{(\kappa, b)\in \widetilde{\mathbf{U}}\mid |\kappa-\kappa_{o}|_{\mathbf{h}}<\varepsilon,
b\in \widetilde{\mathbf{O}}_{b_o}(\delta_o,\rho_o)\right\},$$
$$\mathbf{U}_{\kappa_{o},b_{o}}(\varepsilon,\delta_o,\rho_o)
=\widetilde{\mathbf{U}}_{\kappa_{o},b_{o}}(\varepsilon,\delta_o,\rho_o)/G_{\kappa_o,b_o},$$
where $G_{\kappa_o,b_o}$ is the isotropy group at $(\kappa_o,b_o)$. For any $(\kappa,b)\in
\widetilde{\mathbf{U}}_{\kappa_{o},b_{o}}(\varepsilon,\delta_o,\rho_o)$
denote by $G_{\kappa,b}$ the isotropy group at $(\kappa,b)$.  Any element $\varphi\in G_{\kappa,b}$ satisfies
$\varphi^*(\kappa, b)=(\kappa, b).$
It follows that $G_{\kappa,b}$ is a subgroup of $\mathbf{G}_a$.

\v

\begin{lemma}\label{orbi structure} Let $[(\kappa_{o},b_{o})]\in \mathbf{U}^T$.
Suppose that $\widetilde{\mathbf{U}}_{\kappa_{o},b_{o}}(\varepsilon,\delta_o,\rho_o)\subset \mathbf{U}^T.$ The following hold
\begin{itemize}
\item[(1)] For any $p\in\widetilde{\mathbf{U}}_{\kappa_{o},b_{o}}(\varepsilon,\delta_o,\rho_o)$
let $G_{p}$ be the isotropy group at $p$, then $im(G_{p})$ is a subgroup of $G_{\kappa_o,b_o}$.
\item[(2)] Let $p\in\widetilde{\mathbf{U}}_{\kappa_{o},b_{o}}(\varepsilon,\delta_o,\rho_o)$
be an arbitrary point with isotropy group $G_p$, then there is a $G_p$-invariant neighborhood $O(p)\subset \widetilde{\mathbf{U}}_{\kappa_{o},b_{o}}(\varepsilon,\delta_o,\rho_o)$
such that for any $q\in O(p)$, $im(G_{q})$ is a subgroup of $G_{p}$, where $G_p$, $G_q$ denotes the isotropy groups at $p$ and $q$ respectively.
\end{itemize}
\end{lemma}
\v\n
{\bf Proof:} We only prove (1), the proof of (2) is similar. Denote $b_o=(a_o,u)$.
If the lemma not true, we can find a sequence $(\kappa_i,b_i)=(\kappa_i,a_i,u_i)\in \widetilde{\mathbf{U}}_{\kappa_{o},b_{o}}(\varepsilon,\delta_o,\rho_o)$
such that
\v
{(1)} $\delta_i\to 0$, $\rho_i\to 0$, $\kappa_i\to \kappa_o$,
\v
{(2)} $im(G_{\kappa_i,b_i})$ is not a subgroup of $G_{\kappa_o,b_o}$.
\v\n
It is obvious that $G_{\kappa_o,b_o}$ is a subgroup of $\mathbf{G}_{a_o}$, $G_{\kappa_i,b_i}$ is a subgroup of $\mathbf{G}_{a_i}$ and $\mathbf{G}_{a_i}$ can be imbedded into $\mathbf{G}_{a_o}$ as a subgroup for $i$ large enough. So we can view $im(G_{\kappa_i,b_i})$ as a subgroup of $\mathbf{G}_{a_o}$,  By choosing subsequence we may assume that $im(G_{\kappa_i,b_i})$ convergies to a subgroup $G_{\kappa,b}$ of $\mathbf{G}_{a_o}$ and $im(G_{\kappa_i,b_i})\cdot u_i$ converges to $im(G_{\kappa,b})\cdot u$ and $u_i$ converges to $u$ in $W^{k,2}$. By Sobolev imbedding theorem and elliptic estimates we have $im(G_{\kappa_i,b_i})\cdot (\kappa_i,u_i)$ converges to $im(G_{\kappa,b})\cdot (\kappa_o,u)$, $(\kappa_i,u_i)$ converges to $(\kappa_o,u)$ in $C^{\ell}$ for any $\ell>1$. It follows that $im(G_{\kappa,b})\subset G_{\kappa_o,b_o}$. Since there are only finite many subgroups of $G_{a_o}$, for $i$ large enough we have $im(G_{\kappa_i,b_i})=G_{\kappa,b}$. So $G_{\kappa_i,b_i}$ can be imbedded into $G_{\kappa_o,b_o}$ as a subgroup for $i$ large enough. We get a contradiction. $\Box$
\v
As corollary of Lemma \ref{orbi structure} we conclude that $\mathbf{U}^T$ is an orbifold. Since $(g,n)\ne (1,1), (2,0)$,
$\mathbf{U}^T$ has the structure of an effective orbifold.

\v
Combination of the subsections \S\ref{smoothness-1},
\S\ref{smoothness-2} give us the proof of Theorem \ref{Smooth}.

\v
\subsection{A metric on $\mathbf{E}$}\label{a metric}
In this section we construct a metric on $\mathbf{E}|_{\mathbf U_{\epsilon}}.$
 By the compactness of $\mathbf{U}_{2\varepsilon}$ we may find finite many points $(\kappa_1,b_1),...,(\kappa_{\mathbf{n}}, b_{\mathbf{n}})\in \mathbf{U}_{\varepsilon}$ such that
\begin{itemize}
\item $\{\mathbf{U}_{[(\kappa_{\mathbf{a}},b_{\mathbf{a}})]}(\varepsilon_{\mathbf{a}},\delta_{\mathbf{a}},\rho_{\mathbf{a}})
, \;1\leq \mathbf{a}\leq \mathbf{n}\}$ is a covering of $\mathbf{U}_{2\varepsilon}$.
\item For any $\mathbf{a}\in \{1,...,\mathbf{n}\}$ there is $i_{\mathbf{a}}\in \{1,...,\mathbf{m}\}$ such that
$$\mathsf{p}({\mathbf{U}}_{[(\kappa_{\mathbf{a}},b_{\mathbf{a}})]}(\varepsilon_{\mathbf{a}},\delta_{\mathbf{a}},\rho_{\mathbf{a}}))
\subset  {\mathbf{O}}_{b_{i_{\mathbf a}}}(\delta_{i_{\mathbf a}},\rho_{i_{\mathbf a}}),$$  where ${\mathbf{O}}_{b_{i_{\mathbf a}}}(\delta_{i_{\mathbf a}},\rho_{i_{\mathbf a}})$ is as in subsection \S\ref{finite rank orbi-bundle},
	\item $ \widetilde{{\mathbf{U}}}_{(\kappa_{\mathbf{a}},b_{\mathbf{a}})}(\varepsilon_{\mathbf{a}},\delta_{\mathbf{a}},\rho_{\mathbf{a}})\subset \widetilde{\mathbf U}^{T}$ for all $1\leq \mathbf{a}\leq \mathbf{n}_{t}$.
\end{itemize}

 Let $\{e^{i_{\mathbf a}}_{\alpha}\}_{1\leq\alpha\leq \mathsf r}$ be a local smooth frame field of $\mathbf F$ over $\mathbf{O}_{b_{i_{\mathbf{a}}}}(\delta_{i_{\mathbf{a}}},\rho_{i_{\mathbf{a}}})$ as in section \S\ref{global_regu}. Let $\mathsf{p}: \mathbf{U}\to \mathcal{U}$ denote the projection. Denote $e^{\mathbf a}_{\alpha}=\mathsf{p}^{*}e^{i_{\mathbf a}}_{\alpha}|_{{\mathbf{U}}_{[(\kappa_{\mathbf{a}},b_{\mathbf{a}})]}(\varepsilon_{\mathbf{a}},\delta_{\mathbf{a}},\rho_{\mathbf{a}}) }.$
Then we have a smooth frame field $\{e_{\alpha}^{\mathbf a}\}_{1\leq\alpha\leq \mathsf r}$ of $\mathbf{E}$ over ${\mathbf{U}}_{[(\kappa_{\mathbf{a}},b_{\mathbf{a}})]}(\varepsilon_{\mathbf{a}},\delta_{\mathbf{a}},\rho_{\mathbf{a}})$, where $\mathsf r$ denotes the rank of $\mathbf{E}$.
We define a local metric $h_{\mathbf{a}}$ on $\mathbf{E}|_
{\mathbf{U}_{[(\kappa_{\mathbf{a}},b_{\mathbf{a}})]}(\varepsilon_{\mathbf{a}},\delta_{\mathbf{a}},\rho_{\mathbf{a}})}$ by
$$
h_{\mathbf a}(e^{\mathbf a}_{\alpha},e^{\mathbf a}_{\beta})=\delta_{\alpha \beta}.
$$

\v
Now we choose smooth cutoff functions $\mathbf{\Gamma}'$ as follows.
Let $(\kappa_o, b_o)$ be one of $(\kappa_1,b_1),...,(\kappa_{\mathbf{n}}, b_{\mathbf{n}})$.
We consider two cases.
\v
{\bf (1).} \;$(\kappa_o,b_o)$ lies in $\widetilde{\mathbf{U}}^T$. We define a cut-off function $\alpha_{b_o}: \widetilde{\mathbf{O}}_{b_o}(\delta_{b_{o}},\rho_{b_{o}}) \to [0, 1]$
by \eqref{cut-off-1} and let $\mathbf{\Gamma}'_{o}=\mathsf{p}^*\alpha_{b_o}(b).$ \v
{\bf (2).} \;$(\kappa_o,b_o)$ lies in a lower strata.
We define a  cut-off function $\alpha_{b_o}: \widetilde{\mathbf{O}}_{b_o}(\delta_o,\rho_o) \to [0, 1]$
by \eqref{cut-off-2}
and let $\mathbf{\Gamma}'_{o}=\mathsf{p}^*\alpha_{b_o}(b).$

\v
Thus we have $\mathbf{\Gamma}'_{\mathbf{a}}$ for every $1\leq \mathbf{a}\leq \mathbf{n}.$ Set $$\mathbf{\Gamma}_{\mathbf{a}}=\frac{\mathbf{\Gamma}'_{\mathbf{a}}}{\sum_{l=1}^{\mathbf{n} } \mathbf{\Gamma}'_{l}}.$$ Then $\sum \mathbf{\Gamma}_{\mathbf{a}}=1$ and $\mathbf{\Gamma}_{\mathbf{a}}$ is smooth on $\mathbf{U}^T_{\epsilon}$ in orbifold sense. We define a metric $\mathbf{h}$ on $\mathbf E$ over $\mathbf{U}_{\varepsilon}$ by
$$\mathbf{h}=\sum_{\mathbf{a}=1}^{\mathbf{n}}\mathbf{\Gamma}_{\mathbf{a}}  h_{\mathbf{a}}.
$$
\v
 We define a connection on $\mathbf E$ as follows.
Let $\{e^{\mathbf a}_{\alpha}\}_{1\leq\alpha\leq \mathsf r}$ be a local smooth frame field of $\mathbf E$ over ${\mathbf{U}}_{[(\kappa_{\mathbf{a}},b_{\mathbf{a}})]}(\varepsilon_{\mathbf{a}},\delta_{\mathbf{a}},\rho_{\mathbf{a}})$ as above. Consider the  Gram-Schmidt process with respect to the metric $\mathbf h$ and denote by $\hat{e}^{{\mathbf a}}_{1}, ..., \hat{e}^{{\mathbf a}}_{\mathsf r}$ the Gram-Schmidt orthonormalization of $\{e^{\mathbf a}_{\alpha}\}$.  We define a local connection $\nabla^{\mathbf a}$ by
$$
\nabla^{\mathbf a}\hat{e}^{\mathbf a}_{\alpha}=0,\;\;\;\;\;\alpha=1,\cdots,\mathsf r.
$$
For any section $e\in \mathbf E|_{{\mathbf{U}}_{\epsilon}},$ we define
\begin{equation}\label{def_conn}
\nabla e=\sum \mathbf{\Gamma}_{{\mathbf a}}\nabla^{\mathbf a}(e|_{{\mathbf{U}}_{[(\kappa_{\mathbf{a}},b_{\mathbf{a}})]}(\varepsilon_{\mathbf{a}},\delta_{\mathbf{a}},\rho_{\mathbf{a}})}).
\end{equation}
It is easy to see that $\nabla$ is a compatible connection of the metric $\mathbf{h}$. Denote $$\nabla \hat e^{\mathbf a}_{\alpha}=\sum_{\beta} \omega_{\alpha \beta}^{\mathbf {a}}\hat e^{\mathbf a}_{\beta},\;\;\;\nabla^2\hat e^{\mathbf a}_{\alpha}=\sum  \Omega_{\alpha\beta}^{\mathbf a}\hat e^{\mathbf a}_{\beta}.$$
For any ${\mathbf{U}}_{[(\kappa_{\mathbf{a}},b_{\mathbf{a}})]}(\varepsilon_{\mathbf{a}},\delta_{\mathbf{a}},\rho_{\mathbf{a}})\bigcap  {\mathbf{U}}_{[(\kappa_{\mathbf{c}},b_{\mathbf{c}})]}(\varepsilon_{\mathbf{c}},\delta_{\mathbf{c}},\rho_{\mathbf{c}})\neq \emptyset,$ let $(\hat a^{\mathbf a\mathbf c}_{\alpha\beta})_{1\leq \alpha,\beta\leq \mathsf r}$ be   functions such that $$\hat e^{\mathbf a}_{\alpha}=\sum_{\beta=1}^{\mathsf r} \hat a^{\mathbf a\mathbf c}_{\alpha \beta}\hat e^{\mathbf c}_{\beta},\alpha=1,\cdots,\mathsf r.$$   It is easy to see that
\begin{equation}\label{local_omega}
\omega^{\mathbf a}_{\alpha\beta}=\sum_{\mathbf c}\sum_{\beta=1}^{\mathsf r} \alpha_{b_{\mathbf c}}d\hat a^{\mathbf a\mathbf c}_{\alpha \gamma}\hat a^{\mathbf c \mathbf a}_{\gamma \beta}.
\end{equation}
We get a  metric $\mathbf h$ and a connection $\nabla$ in $\mathbf E$ over $\mathbf{U}_{\varepsilon}$.
\v

\section{\bf Gluing estimates}\label{estimates}

\subsection{Gluing maps}\label{gluing map}

Let $\Sigma$ be a marked nodal Riemann surfaces. Suppose that $\Sigma$
has nodes $p_{1},\cdots,p_{\mathfrak{e}}$ and  marked points $y_{1},\cdots,y_{n}$. We choose local coordinate system $\mathbf{A}$. Let $u:\Sigma\to M$ be perturbed $J$-holomorphic map. We glue $\Sigma$ and $u$ at each node with gluing parameters $(\mathbf{r})$ to get $\Sigma_{(\mathbf{r})}$ and the pregluing map $u_{(\mathbf{r})}: \Sigma_{(\mathbf{r})}\to M$. Set
 $$\ft_{i}=e^{-2r_{i}-2\pi\tau_{i}}, \;\;\;|\mathbf{r}|=min\{r_1,...,r_{\mathfrak{e}}\},\;\;\;b_{(\mathbf r)}:=(0,\mathbf{(r)}, u_{(\mathbf r)}).$$
The following lemma is proved in
\cite{LS-1}.
\begin{lemma}\label{isomor of ker}
For $|\mathbf{r}|>R_0$ there is an isomorphism
$$I_{(\mathbf r)}: \ker D \mathcal{S}_{( \kappa_{o}, b_{o})}\longrightarrow \ker D \mathcal{S}_{( \kappa_{o},b_{(\mathbf r)})}.$$
\end{lemma}
\v\n
Using Theorem 5.3 in \cite{LS-1} and the implicit function theorem with parameters we immediately obtain
\begin{lemma}\label{gluing map} There are constant $\varepsilon>0$, $R_0>0$ and a neighborhood $O_1\subset \mathbf{A}$ of $\mathbf{s}_o$ and a neighborhood $O$ of $0$ in $ker D{\mathcal S}_{(\kappa_0,b_0)}$
 such that
$$glu: O_1\times(\mathbb{D}_{\mathbf{c}}^*(0))^{\mathfrak{e}}\times O\rightarrow glu(O_1\times(\mathbb{D}_{\mathbf{c}}^*)^{\mathfrak{e}}\times O)\subset\widetilde{\mathbf{U}}^T$$
 is an orientation preserving local diffeomorphisms, where $$\mathbb{D}^{*}_{\mathbf{c}}(0):=\{\mathbf{t}\mid 0<|\mathbf{t}|<\mathbf{c}\},\;\;\;\mathbf{c}=e^{-2R_0}.$$  \end{lemma}
\n
Denote
$$ Glu_{\fs,(\mathbf{r})}=I_{(\mathbf{r})} + Q_{b_{(\mathbf{r})}}\circ f_{\fs,(\mathbf{r})}\circ I_{(\mathbf{r})}.$$

\subsection{ Equivariant gluing}

We consider the case with one node, the construction in this section can be generalized to the case with several nodes.
Let $(\kappa_{o},b_{o})\in \widetilde{\mathbf U}$ where
$$\kappa_{o}=(\kappa_{o1},\kappa_{o2}),\;\;\;  b_{oi}=(a_{oi},u_{i}),\;\;\;a_{oi}=(\Sigma_{i},j_{oi},\mathbf{y}_{i}, q),\;\;\;\;i=1,2,$$  $\Sigma_{1}$ and $\Sigma_{2}$ are smooth Riemann surfaces joining at $q$.  Assume that $(\Sigma_{i},\mathbf y_{i}, q)$ is stable. Let $G_{(\kappa_{o},b_o)}=(G_{(\kappa_{o1},b_{o1})}, G_{(\kappa_{o2},b_{o2})})$ be the isotropy group at $(\kappa_{o},b_o)$, thus,
$$G_{(\kappa_{o},b_{o})}=\{\phi=(\phi_{1},\phi_{2})|\;\phi_{i}\in Diff^+(\Sigma_{i}),\; \phi_{i}^{*}(j_{oi},\mathbf{y}_{i}, q,
\kappa_{oi},u_{i})=(j_{oi},\mathbf{y}_{i},q,
\kappa_{oi},u_{i}),i=1,2\}.$$
 Obviously,  $G_{(\kappa_{o},b_o)}$ is a subgroup of $\mathbf{G}_{a_{o}}$.  The following lemma can be easily to check.
\begin{lemma}\label{equi_diff}    For any $(\kappa,b)\in \widetilde{\mathbf U}_{(\kappa_{o},b_{o})}(\varepsilon,\delta,\rho),$ $\varphi\in Diff^+(\Sigma)$ denote
	$$ \kappa'=\varphi^{*}\kappa,\;\;\;b'=(j',{\bf y}',u') =\varphi\cdot (j,{\bf y},u)=(\varphi^*j, \varphi^{-1}{\bf y}, \varphi^{*}u).$$
	Then  $$\bar{\p}_{j',J_{u'}}=\varphi^{*}(\bar{\p}_{j,J_{u}}),\;\;\;
	\mathfrak{i}(\kappa',b')=\varphi^{*}(\mathfrak{i}(\kappa,b)).$$
\end{lemma}
\v\n
It is easy to check that  the operator  $D\mathcal{S}_{(\kappa_o,b_o)}$ is $G_{(\kappa_{o},b_{o})}$-equivariant. Then we may choose a $G_{(\kappa_{o},b_{o})}$-equivariant right inverse $Q_{(\kappa_o,b_o)}$. In fact, let $\hat Q_{(\kappa_o,b_o)}$ be a right inverse of  $D\mathcal{S}_{(\kappa_o,b_o)},$  we define
$$Q_{(\kappa_o,b_o)}(\eta)=\frac{1}{|G_{(\kappa_{o},b_{o})}|}   \sum_{\varphi\in G_{(\kappa_{o},b_{o})}}  \varphi^{-1}\cdot  \hat Q_{(\kappa_o,b_o)}( \varphi \cdot \eta).$$
Then, for any $\varphi'\in G_{(\kappa_{o},b_{o})} ,$ we have
 $$Q_{(\kappa_o,b_o)}(\varphi'\cdot\eta)=\frac{1}{|G_{(\kappa_{o},b_{o})}|}   \sum_{\varphi\in G_{(\kappa_{o},b_{o})}}  \varphi^{-1}\cdot  \hat Q_{(\kappa_o,b_o)}( \varphi \cdot \varphi'\cdot \eta)=$$$$
 \frac{1}{|G_{(\kappa_{o},b_{o})}|}   \sum_{\varphi\in G_{(\kappa_{o},b_{o})}} \varphi'\cdot(\varphi')^{-1}  \varphi^{-1}\cdot  \hat Q_{(\kappa_o,b_o)}( \varphi \cdot \varphi'\cdot \eta)=\varphi'\cdot Q_{(\kappa_o,b_o)}(\eta) $$
By the $G_{(\kappa_{o},b_{o})}$-equivariance of  $D\mathcal{S}_{(\kappa_o,b_o)}$,\;
$G_{(\kappa_{o},b_{o})}$ acts on $ker D\mathcal{S}_{(\kappa_o,b_o)}$ in a natural way.
\v
We choose cusp cylinder coordinates $(s_{i}, t_{i})$ on $\Sigma_{i}$ near $q$. Choosing the gluing parameter $({r})$ we construct $\Sigma_{({r})}$ and $u_{(r)}$ as in \S\ref{gluing}. Since the cut-off function $\beta (s)$
depends only on $s$, $G_{b_o}$ acts on $\widetilde{O}_{b_{o}}(\delta,\rho).$ Since  the $\|\cdot\|_{k,2,\alpha,r}$ is $G_{b_{o}}$ invariant, it induces a $G_{(\kappa_{o},b_{o})}$ action on $\widetilde{\mathbf U}_{(\kappa_{o},b_{o})}(\varepsilon_{o},\delta_{o},\rho_{o})$.
\v
Set $a_{(r)}=(\Sigma_{(r)},j_{o},\mathbf y)$ and $b_{(r)}=(a_{(r)},u_{(r)}).$
Denote by $G_{b_{(r)}}$ (resp. $G_{(\kappa_{o},b_{(r)})}$) the isotroy group at $b_{(r)}$ (resp. $(\kappa_{o},b_{(r)})$). It is easy to see that $G_{b_{(r)}}$ is a subgroup of $G_{b_{o}}.$ It follows that
$G_{(\kappa_{o},b_{(r)})}$ is a subgroup of $G_{(\kappa_{o},b_{o})}.$
Then $G_{(\kappa_{o},b_{(r)})}$ can be seen as rotation in the gluing part. The gluing map is the $\frac{|G_{(\kappa_{o},b_{o})}|}{|G_{(\kappa_{o},b_{(r)})}|}$-multiple covering map.
 Since $\beta_{1;r}$ is independent of $\tau$,  $Q'_{(\kappa_{o},b_{(r)})}$ is  $G_{(\kappa_{o},b_{(r)})}$-equivariant.
 By the definition of $Q_{(\kappa_{o},b_{(r)})}$ and the $G_{(\kappa_{o},b_{(r)})}$-equivarance of $D\mathcal{S}_{(\kappa_{o},b_{(r)})}$, we can conclude that  $Q_{(\kappa_{o},b_{(r)})}$ is  $G_{(\kappa_{o},b_{(r)})}$-equivariant. It follows from the definition of $I_{(r)}$ that $I_{(r)}$ is  $G_{(\kappa_{o},b_{(r)})}$-equivariant.
 By the uniquiness of the implicit function $f$, we conclude that $f$ is $G_{(\kappa_{o},b_{(r)})}$-equivariant.
Since
$$Glu_{(r)}=I_{(r)}+Q_{(\kappa_{o},b_{(r)})}\circ f \circ I_{(r)}  .$$
  $Glu_{(r)}$ is $G_{(\kappa_{o},b_{(r)})}$-equivariant.
Denote
$$
\ker D\mathcal{S}_{[\kappa_o,b_o]}=\ker D\mathcal{S}_{(\kappa_o,b_o)}/G_{(\kappa_{o},b_{o})},\;\;\;\ker D\mathcal{S}_{[\kappa_o,b_{(r)}]}=\ker D\mathcal{S}_{(\kappa_o,b_{(r)})}/G_{(\kappa_{o},b_{(r)})}.
$$
 Then we have
 \begin{lemma}\label{equi_glu}
{\bf (1)} $I_{(r)}: \ker D\mathcal{S}_{(\kappa_o,b_o)} \longrightarrow \ker D\mathcal{S}_{(\kappa_o,b_{(r)})}$ is a $\frac{|G_{(\kappa_{o},b_{o})}|}{|G_{(\kappa_{o},b_{(r)})}|}$-multiple covering map.
\v
{\bf (2)} $I_{(r)}$ induces a isomorphism $I_{(r)}:  \ker D\mathcal{S}_{[\kappa_o,b_o]} \longrightarrow \ker D\mathcal{S}_{[\kappa_o,b_{(r)}]}$.
\end{lemma}

\subsection{Exponential decay of gluing maps}

The following theorem is proved in \cite{LS-1}.
\begin{theorem}\label{coordinate_decay-2} Let $l\in \mathbb Z^+$ be a fixed integer.
There exists positive  constants  $\mathsf{C}_{l}, \mathsf{d}, R_{0}$ such that for any $(\kappa,\xi)\in \ker D \mathcal{S}_{(\kappa_{o},b_{o})}$ with $\|(\kappa,\xi)\|< \mathsf{d}$ and for any $X_{i}\in \{\frac{\p}{\p r_{i}},\frac{\p}{\p \tau_{i}}\},i=1,\cdots,\mathfrak{e}$, restricting to the compact set $\Sigma(R_0)$, the following estimate hold
$$ \left\|X_{i} \left(Glu_{\fs,(\mathbf{r})}(\kappa,\xi) \right)  \right\|_{C^{l}(\Sigma(R_{0}))}
	\leq  \mathsf{C}_{l}e^{-(\fc-5\alpha)\tfrac{r_{i}}{4} },$$
$$ \left\|X_{i}X_{j}\left(Glu_{\fs,(\mathbf{r})}(\kappa,\xi) \right)  \right\|_{C^{l}(\Sigma(R_{0}))}
	\leq  \mathsf{C}_{l}e^{-(\fc-5\alpha)\tfrac{r_{i}+r_{j}}{4} },$$
$1\leq i\neq j\leq\mathfrak{e},$
	for any $\mathbf{s}\in \bigotimes_{l=1}^{\iota} O_l$ when $|r|$ big enough.
\end{theorem}

\subsection{Estimates of exponential decay of the line bundle}

The following theorem is proved in \cite{LS-2}

\begin{theorem}\label{thm_est_mix_deri}
Let $l\in \mathbb Z^+$ be a fixed integer.
 Let $u:\Sigma\to M$ be a $(j,J)$-holomorphic map.	Let $\fc\in (0,1)$ be a fixed constant.    For any $0<\alpha<\frac{1}{100\fc}$, there exists positive  constants  $\mathsf C_{l}, \mathsf{d},R$ such that for any  $\zeta\in \ker D^{\widetilde{\mathbf L}}|_{b_{o}}$,   $(\kappa,\xi)\in \ker D \mathcal{S}_{(\kappa_{o},b_{o})}$ with    $$\|\zeta\|_{\mathcal W,k,2,\alpha}\leq \mathsf{d},\;\;\;\;\;\|(\kappa,\xi)\|< \mathsf{d},\;\;\;\;\;|\mathbf r|\geq R, $$ restricting to the compact set $\Sigma(R_0)$, the following estimate hold.
	\begin{equation}\label{eqn_est_L_1st}
	\left\|X_{i}  \left( Glu_{\mathbf{s},h_{(\mathbf r)},(\mathbf{r})}^{\widetilde{\mathbf L}}(\zeta) \right) 	\right\|_{C^{l}(\Sigma(R_{0}))}
	\leq \mathsf C_{l} e^{-(\fc-5\alpha)\tfrac{r_{i}}{4} } ,
	\end{equation}
	\begin{align} \label{eqn_est_L_2nd}
	&\left\|X_{i}X_{j} \left( Glu^{\widetilde{\mathbf L}}_{\mathbf{s},h_{(\mathbf r)},(\mathbf{r})}(\zeta) \right)
	\right\|_{C^{l}(\Sigma(R_{0}))} \leq   \mathsf C_{l} e^{-(\fc-5\alpha)\tfrac{r_{i}+r_{j}}{4} }
	\end{align}
	for any $X_{i}\in \{\frac{\p}{\p r_{i}},\frac{\p}{\p \tau_{i}}\},i=1,\cdots,\mathfrak{e}$,    $\mathbf{s}\in \bigotimes_{l=1}^{\iota} O_l$ and any $1\leq i\neq j\leq \mathfrak{e}$, where $h_{(\mathbf r)}=\Pi_{2}(Glu_{\fs,(\mathbf r)}(\kappa,\xi))$ and $\Pi_{2}:  \widetilde{\mathbf F}_{b_{(\mathbf r)}}\times T_{u_{(\mathbf r)}} \widetilde{\mathcal B} \to T_{u_{(\mathbf r)}}\widetilde{ \mathcal B}$ denotes the projection.
\end{theorem}

\v
\subsection{Estimates of Thom forms}\label{est_Thom_E}

\v
We estimate the derivatives  of the metric $\mathbf h$ near the boundary of $\mathbf F|_{\mathbf U^{T}}$. Let $(\kappa_{o},b_{o})$ be one of $\{(\kappa_{\mathbf{a}},b_{\mathbf{a}}), \mathbf{a}=\mathbf{n}_{t}+1,\cdots,\mathbf{n}\}$ and  $b_{o}=(a_{o},u).$
Fix  a basis $\{\mathbf{e}_{1},\cdots,\mathbf{e}_{d}\}$ of $Ker\;D{\mathcal  S}_{(\kappa_o,b_o)}$  and let $\mathfrak{z}=(\mathfrak{z}_{1},\cdots,\mathfrak{z}_{d})$ be the corresponding coordinates. Set $\ft_{i}=e^{-2r_{i}-2\pi\tau_{i}}$, $1\leq i\leq \mathbf{e}$. Denote
$$
\mathcal L(\fs,(\mathbf r),\mathfrak{z}):= I_{(\mathbf r)}\left(\sum_{i=1}^{d} \mathfrak{z}_{i}\mathbf{e}_{i}\right)+Q_{( \kappa_{o},b_{(\mathbf r)})}\circ f_{\fs,(\mathbf r)} \circ I_{(\mathbf r)} \left(\sum_{i=1}^{d} \mathfrak{z}_{i}\mathbf{e}_{i}\right),$$
where    $b_{(\mathbf r)}=(0,(\mathbf r),u_{(\mathbf r)}).$ Then $(\fs,(\mathbf r),\fkz)$ is  a local coordinates of ${\mathbf{U}}_{(\kappa_{\mathbf{a}},b_{\mathbf{a}})}(\varepsilon_{\mathbf{a}},\delta_{\mathbf{a}},\rho_{\mathbf{a}})$.
We say that $f(\mathbf{s},(\mathbf r),\fkz)$ satisfies $(\mathbf r)$-exponential decay if
\begin{equation}
 \left(\left| \frac{\p f}{\p r_{i}}\right|+ \left| \frac{\p f}{\p \tau_{i}}\right|\right) \leq  C e^{-\delta r_{i}},\;\;\forall \; 1\leq i\leq \mathbf{e}
\end{equation}
\begin{equation}
\left| \frac{\p f}{\p s_{j}}\right|+ \left| \frac{\p f}{\p \fkz_{\alpha}}\right|\leq C,\;\;\forall \;1\leq j\leq \iota,\; 1\leq \alpha \leq d.
\end{equation}
Let
$$
\Pi_{1}:\widetilde{\mathbf F}_{b_{(\mathbf r)}}\times T_{u_{(\mathbf r)}} \widetilde{\mathcal B} \to \widetilde{\mathbf F}_{b_{(\mathbf r)}},\;\;\;\;\Pi_{2}:  \widetilde{\mathbf F}_{b_{(\mathbf r)}}\times T_{u_{(\mathbf r)}} \widetilde{\mathcal B} \to T_{u_{(\mathbf r)}}\widetilde{ \mathcal B}
$$
be the projection.
By Theorem \ref{coordinate_decay-2},  the implicit function Theorem  and  \eqref{cut-off-2}, we conclude that $\mathbf{\Gamma}_{\mathbf a}$  satisfies $(\mathbf r)$-exponential decay,
  where $\mathbf{\Gamma}_{\mathbf a}$ is the cutoff function defined in section   \S\ref{a metric}.

\v

For any $ {\mathbf{U}}_{(\kappa_{\mathbf{a}},b_{\mathbf{a}})}(\varepsilon_{\mathbf{a}},\delta_{\mathbf{a}},\rho_{\mathbf{a}})\bigcap  {\mathbf{U}}_{(\kappa_{\mathbf{c}},b_{\mathbf{c}})}(\varepsilon_{\mathbf{c}},\delta_{\mathbf{c}},\rho_{\mathbf{c}})\neq \emptyset,$
 let $a^{\mathbf a\mathbf c}_{\alpha\beta},\alpha,\beta=1,\cdots,\mathsf r$ be   functions such that $e^{\mathbf a}_{\alpha}=\sum_{\beta=1}^{\mathsf r_{i}} a^{\mathbf a\mathbf c}_{\alpha \beta}e^{\mathbf c}_{\beta},\alpha=1,\cdots,\mathsf r.$
By  the implicit function theorem, Theorem \ref{thm_est_mix_deri} we have, for any $p\in \Sigma(R_{0}),$
 $e^{\mathbf a}_{\alpha}(p),e^{\mathbf c}_{\beta}(p)$ satisfies $(\mathbf r)$-exponential decay.
Since   $a^{\mathbf a\mathbf c}_{\alpha \beta}$ is a  function of $(\mathbf{s},(\mathbf r),\fkz)$, we have
\begin{equation}\label{equ_exp_a}
d (e^{\mathbf a}_{\alpha}(p))=\sum_{\beta=1}^{\mathsf r} e^{\mathbf c}_{\beta} (p) \cdot da^{\mathbf a\mathbf c}_{\alpha \beta}+\sum_{\beta=1}^{\mathsf r} a^{\mathbf a\mathbf c}_{\alpha \beta}\cdot d(e^{\mathbf c}_{\beta}(p)),\;\;\;\;\forall\; p\in \Sigma(R_{0}),
\end{equation}
Recall that $e^{\mathbf a}_{\alpha}=\left(I_{(\mathbf r)}^{\widetilde{\mathbf L}}+Q_{(\mathbf r)}^{\widetilde{\mathbf L}}f^{\mathbf L}_{\fs,h_{(\mathbf r)},(\mathbf r)}I_{(\mathbf r)}^{\widetilde{\mathbf L}}\right)(e^{\mathbf a}_{\alpha}|_{(\kappa_{\mathbf a},b_{\mathbf a})}).$ Using the implicit function theorem we get $$\|Q_{(\mathbf r)}^{\widetilde{\mathbf L}}f^{\mathbf L}_{\fs,h_{(\mathbf r)},(\mathbf r)}I_{(\mathbf r)}^{\widetilde{\mathbf L}}(e^{\mathbf a}_{\alpha}|_{(\kappa_{\mathbf a},b_{\mathbf a})})\|_{k,2,\alpha,\mathbf r}\leq 2C\left\|D^{\widetilde{\mathbf L}}_{b}\circ(P^{\mathbf{L}}_{b,b_{(\mathbf{r})}})^{-1}(I_{(\mathbf r)}^{\widetilde{\mathbf L}}(e^{\mathbf a}_{\alpha}|_{(\kappa_{\mathbf a},b_{\mathbf a})}))\right\|.$$ Choosing $\delta_{\mathbf a}$ and $\rho_{\mathbf a}$ small enough, by the exponential estimates of $e^{\mathbf a}_{\alpha}|_{b_{\mathbf a}}$ we have
$$
\|(e^{\mathbf a}_{\alpha}|_{(\kappa,b)})|_{\Sigma(R_{0})}\|_{k,2,\alpha}\geq \frac{1}{4}\|e^{\mathbf a}_{\alpha}\|_{k,2,\alpha}.
$$
 So $\max\limits_{\Sigma(R_{0})}|e_{\alpha}^{\mathbf a}|$ has uniform lower bound. Then we obtain   the  $(\mathbf r)$-exponential decay of $a^{\mathbf a\mathbf c}_{\alpha \beta}$.
Denote $h^{\mathbf a}_{\alpha\beta}=\langle e_{\alpha}^{\mathbf a},e_{\beta}^{\mathbf a}  \rangle_{\mathbf h}.$ By the definition of $\mathbf h$ and the $(\mathbf r)$-exponential decay of $\mathbf{\Gamma}_{\mathbf a}$,  $a^{\mathbf a\mathbf c}_{\alpha \beta}$ we conclude that $h^{\mathbf a}_{\alpha\beta}$ satisfies   the  $(\mathbf r)$-exponential decay.
By   the GramSchmidt orthonormalization and the similar argument above
we obtain the $(\mathbf r)$-exponential decay of $\hat a^{\mathbf a\mathbf c}_{\alpha\beta}.$

Let $\Delta_r$ be the open disk in $\mathbb C$
with radius $r$, let  $\Delta_r^*=\Delta_r\setminus\{0\}$ and $\Delta ^*=\Delta
\setminus\{0\}$. Set $N=3g-3 + n$. For each point $p\in \p \overline{\mathcal{M}}_{g,n}$ we can find a coordinate
chart $(U,\fs_{1},\cdots, \fs_{N-\mathfrak{e}},\ft_1,\cdots,\ft_{\mathfrak{e}})$ around $p$ in $\overline{\mathcal{M}}_{g,n}$ such that
$U\cong\Delta^N$ and $V=U\cap \mathcal{M}_{g,n}\cong \Delta^{N-\mathfrak{e}}\times(\Delta^*)^{\mathfrak{e}}$. We assume that $U\cap \overline{\mathcal{M}}_{g,n}$ is
defined by the equation $\ft_1\cdots \ft_{\mathfrak{e}}=0$.
Let $\{U_{\alpha}\}$ be a local chart of $\overline{\mathcal{M}}_{g,n}$.
On each chart $V_{\alpha}$ of $\overline{\mathcal{M}}_{g,n}$ we can define a local Poincare metric:
\begin{eqnarray}\label{localpo}
g^{\alpha}_{loc}= \sum_{i=1}^{\mathsf m} \frac{|d\ft_i|^2}{|\ft_i|^2   (\log|\ft_i|)^2}+
 \sum_{\alpha=1}^{N-\mathsf m} |d\fs_\alpha|^2.
\end{eqnarray}
We let $U_{\alpha}(r)\cong
\Delta_r^N$ for $0<r<1$ and let $V_{\alpha}(r)=U_{\alpha}(r)\cap \mathcal{M}_{g,n}$.

Let $\fs, (\mathbf r),\fkz$ be the local coordinates of ${\mathbf{U}}^T_{(\kappa_{\mathbf{a}},b_{\mathbf{a}})}(\varepsilon_{\mathbf{a}},\delta_{\mathbf{a}},\rho_{\mathbf{a}}).$ In the coordinates $(\fs,(\mathbf r),\fkz)$
the local Poincare metric $g_{loc}$ can be written as
\begin{equation}\label{omega_loc}
g_{loc}= \sum_{i=1}^{\mathfrak{e}} \frac{4(d^2r_i+d^2\tau_i)}{   r_i^2}+
\sum_{l=1}^{3g-3+n-\mathfrak{e}} |d\fs_l|^2+\sum_{i=1}^{d}d\fkz_{i}^2.
\end{equation}

\v

\begin{lemma}\label{lem_omega}
	There exists a constant $C>0$ such that
$$|\omega^{\mathbf a}_{\alpha\beta}(X_{1})|^2\leq g_{loc}(X_{1},X_{1}),\;\;\;\;
|\Omega^{\mathbf a}_{\alpha\beta}(X_{1},X_{2})|^2\leq \Pi_{i=1}^{2}g_{loc}(X_{i},X_{i})$$
$$|d\Omega^{\mathbf a}_{AB}(X_{1},X_{2},X_{3})|^2\leq \Pi_{i=1}^{3}g_{loc}(X_{i},X_{i})$$
for any $X_{i}\in T\mathbf U^{T},i=1,2,3.$
\end{lemma}
\begin{proof}
	The first inequality follows from \eqref{local_omega} and $(\mathbf r)$-exponential decay of $\hat a^{\mathbf a\mathbf c}_{\alpha\beta}.$ By $\Omega_{\alpha\beta}=d\omega_{\alpha \beta}+\sum_{\gamma}\omega_{\alpha\gamma}\wedge \omega_{\gamma\beta}$ and $(\mathbf r)$-exponential decay of $\hat a^{\mathbf a\mathbf c}_{\alpha\beta}$ and $\mathbf{\Gamma}_{\mathbf a}$,  we can get the second inequality. The last inequality follows from the   Bianchi identity.
\end{proof}

Let $\mathsf{p}^*\mathbf{E}$ be the pull-back of the bundle $\mathbf E$ to a bundle over $\mathbf E$, where $\mathsf{p}:\mathbf E\to \mathbf U$ is the projection.  Then the bundle $\mathsf{p}^* \mathbf E$ has a metric $\mathsf{p}^*\mathbf{h}$ with compatible connection $\mathsf{p}^*\nabla$. To simply notation we write these as $\mathbf{h}$ and $\nabla$.
Let $\hat\sigma$ be  the tautological section of  $\mathsf{p}^*\mathbf E.$ Then the elements $|\hat \sigma|_{\mathbf{h}}^2\in  \mathcal{A}^{0}(\mathbf E,  \wedge^{0}(\mathsf{p}^*\mathbf E))$, and the covariant derivative $\nabla \sigma\in  \mathcal{A}^{1}(\mathbf E,  \wedge^{1}(\mathsf{p}^*\mathbf E)).$
The curvature $\mathsf{p}^*\Omega$ of the connection $\nabla$ on $\mathbf E$ can also seen as an element of $\mathcal{A}^2(\mathbf E, \wedge^2(\mathsf{p}^*\mathbf E))$.
By \cite{BGV,Z}, the Mathai-Quillen type Thom form can be written as
\begin{equation}\label{MQ_thom_gauss}
\Theta_{MQ}=c(\mathsf{r})\int^{B}e^{-\frac{|  \hat \sigma|^2_{\mathbf{h}}}{2}-\nabla\hat \sigma-\mathsf{p}^* \Omega} \in \mathcal A^{\mathsf{r}}(\mathbf E)
\end{equation}
where $c(\mathsf{r})$ is a constant depending only $\mathsf{r},$ $\int^{B}$ denotes the Berezin integral on $\wedge^* (\mathsf{p}^*\mathbf E)$.
Here $\Theta_{\mathbf E}$ is Gaussian shaped Thom form. Let $B_{\epsilon}(0)$ denote the open $\epsilon$-ball in $R^{2\mathbf r}$ and consider the map $\rho_{\epsilon}:B_{\epsilon}(0)\to \mathbb R^{\mathbf r} $ defined by $\rho_{\epsilon}(v)=\frac{v}{\epsilon^2-|v|^2}.$
If we extend $\rho_{\epsilon}^*\Theta_{MQ}$ by setting it
equal to zero outside $B_{\epsilon}(0)$, still denoted by $\Theta_{\mathbf E}:=\rho_{\epsilon}^*\Theta_{MQ},$ we obtain a form $\Theta_{\mathbf E}$ of compact support.

\v
Finally, we have the following estimate for $\sigma^{*}\Theta_{\mathbf E}$.
\begin{lemma}\label{lem_Thom}
	There exists a constant $C>0$ such that
$$	|\sigma^{*}\Theta_{\mathbf E}(X_{1},\cdots,X_{\mathsf r})|^2\leq C \Pi_{i=1}^{\mathsf r}g_{loc}(X_{i},X_{i})$$for any $X_{i}\in T\mathbf U^{T},i=1,2,3.$
\end{lemma}
\begin{proof}
	One can easily check that
	$$\sigma^{*}\Theta_{\mathbf E} =\sigma^{*} \rho^{*}\Theta_{MQ}=c(\mathsf r)\int^{B}e^{-\frac{|\sigma|^2_{\mathbf{h}}}{(\epsilon^2-|\sigma|^2_{\mathbf{h}})^2}-\nabla (\frac{\sigma}{\epsilon^2-|\sigma|^2})- \Omega} \in \mathcal A^{\mathsf{r}}(\mathbf M).$$
Denote $\sigma=\sum_{\alpha} \sigma_{\alpha} \hat e_{\alpha}^{\mathbf a}.$  For any $p\in \Sigma(R_{0})$, by $d\sigma(p)=\sum_{\alpha} d \hat e_{\alpha}^{\mathbf a}(p)\sigma_{\alpha}^{\mathbf a}+\sum_{\alpha}  \hat e_{\alpha}^{\mathbf a}(p) d\sigma^{\mathbf a}_{\alpha}$, as above we obtain the $(\mathbf r)$ exponential decay $\sigma^{\mathbf a}_{\alpha}.$ 	Since $\nabla \sigma= \sum _{\alpha}d\sigma_{\alpha}\hat e_{\alpha}^{\mathbf a}+ \sum_{\alpha,\beta}\sigma_{\beta}\omega_{\alpha \beta}\hat e_{\alpha}^{\mathbf a} $, $\Omega=\sum \Omega_{\alpha\beta}\hat e_{\alpha}^{\mathbf a}\wedge\hat  e_{\beta}^{\mathbf a}$ and
$$
\int^{B}\hat e_{1}^{\mathbf a}\wedge \cdots\wedge\hat  e_{\mathsf r}^{\mathbf a}=1,
$$
the lemma follows from Lemma \ref{lem_omega} and a direct calculation.
\end{proof}

\v

\v

\section{\bf Gromov-Witten invariants}\label{gromov-witten}

\subsection{\bf The convergence of the integrals }\label{conver_GW}

Denote
$$\mathbf {U}_{\epsilon}=\{[(\kappa,b)]\in \mathbf U\;| \;|\kappa|_{\mathbf h}\leq \epsilon\},\;\;\;\;\mathbf {U}_{\epsilon}^{T}=\{[(\kappa,b)]\in \mathbf U^{T}\;| \;|\kappa|_{\mathbf h}\leq \epsilon\}. $$
We choose open covering
$$\{\mathbf{U}_{[(\kappa_{\mathbf{a}},b_{\mathbf{a}})]}(\varepsilon_{\mathbf{a}},\delta_{\mathbf{a}},\rho_{\mathbf{a}})
, \;1\leq \mathbf{a}\leq \mathbf{n}\}$$
of $\mathbf{U}_{2\varepsilon}$
and a family cutoff  functions $\{\mathbf{\Gamma}_{\mathbf{a}},\;1\leq \mathbf{a}\leq \mathbf{n} \}$ as in \S\ref{a metric}.
Let $\Theta_{\mathbf E}$ be the Thom form of $\mathbf E$  supported in a small $\varepsilon$-ball of the $0$-section of $\mathbf{E}$. To simply notation we denote $\Theta_{\mathbf E}$ by $\Theta.$

\v

\v
\begin{remark}\label{partition of unity}
$\{\mathbf{\Gamma}_{\mathbf{a}}\}$ is not a partition of unity in the classical sense, since it is not smooth on lower stratum, and it is not compactly supported. But it is smooth on $\mathbf{U}^T$ and $\mathbf{\Gamma}_{\mathbf{a}}\sigma^{*}\Theta$ is compactly supported. This is enough to define Gromov-Witten invariants.
\end{remark}

Denote
$$\mathbf{V}_{([\kappa_{\mathbf{a}},b_{\mathbf{a}})]}(\varepsilon_{\mathbf{a}},\delta_{\mathbf{a}},\rho_{\mathbf{a}})
:=\mathbf{U}_{[(\kappa_{\mathbf{a}},b_{\mathbf{a}})]}(\varepsilon_{\mathbf{a}},\delta_{\mathbf{a}},\rho_{\mathbf{a}})\cap
\mathbf{U}^T,$$
$$\widetilde{\mathbf{V}}_{\kappa_{\mathbf{a}},b_{\mathbf{a}}}(\varepsilon_{\mathbf{a}},\delta_{\mathbf{a}},\rho_{\mathbf{a}})
:=\widetilde{\mathbf{U}}_{\kappa_{\mathbf{a}},b_{\mathbf{a}}}(\varepsilon_{\mathbf{a}},\delta_{\mathbf{a}},\rho_{\mathbf{a}})\cap
\widetilde{\mathbf{U}}^T.$$
Sometimes we write the above two sets by $\mathbf{V}_{\mathbf{a}}$ and $\widetilde{\mathbf{V}}_{\mathbf{a}}$ to simplify notations. Let $p:
\widetilde{\mathbf{V}}_{\mathbf{a}}\to \mathbf{V}_{\mathbf{a}}$,
let $\widetilde{\mathbf{\Gamma}}_{\mathbf{a}}$, $\widetilde{K}$ and $\widetilde{\Theta}$ be the lift of $\mathbf{\Gamma}_{\mathbf{a}}$, $K$ and $\Theta$ to $\widetilde{\mathbf{V}}_{\mathbf{a}}$.
We write the Gromov-Witten invariants as
\begin{align}\label{def_GRW}
\Psi_{A,g,n}(K;\alpha_1,...,
\alpha_{n}) =\sum_{\mathbf{a}=1}^{\mathbf{n}_{c}} (\mathbf{I})_{\mathbf{a}}
\end{align}
 where
\begin{equation}\label{integrand}
(\mathbf{I})_{\mathbf{a}}:=\int_{\mathbf{V}_\mathbf{a}}
\mathbf{\Gamma}_{\mathbf{a}}
\cdot \mathscr{P}^*(K)\wedge\prod^n_j ev^*_j\alpha_j\wedge \sigma^*\Theta.
\end{equation}

\v\n
{\bf Proof of Theorem \ref{Conver}.}
Note that the integration region $\overline{U}_\mathbf{a}$ for $1\leq \mathbf{a}\leq \mathbf{n}_{t}$ are compact set in $\mathbf{U}^T$ and the integrand in \eqref{integrand} are smooth we conclude that $\sum_{\mathbf{a}=1}^{\mathbf{n}_{t}}(\mathbf{I})_{\mathbf{a}}$ is bounded.
So we only need to prove the convergence of
$(\mathbf{I})_{\mathbf{a}}$ for $\mathbf{a}=\mathbf{n}_{t}+1,\cdots,\mathbf{n}.$
Denote
$$(\mathbf{J})_{\mathbf{a}}=\int_{\widetilde{\mathbf{V}}_\mathbf{a}}
\widetilde{\mathbf{\Gamma}}_{\mathbf{a}}
\cdot \mathscr{P}^*(\widetilde{K})\wedge\prod^n_j \widetilde{ev}^*_j\alpha_j\wedge \widetilde{\sigma}^*\widetilde{\Theta}.$$
It suffices to prove the convergence of $(\mathbf{J})_{\mathbf{a}}$.
\v
Let $(\kappa_{o},b_{o})$ be one of $\{(\kappa_{\mathbf{a}},b_{\mathbf{a}}), \mathbf{a}=\mathbf{n}_{t}+1,\cdots,\mathbf{n}\}$  and  $b_{o}=(a_{o},u).$ We choose  coordinates $(\fs,\ft,\fkz)$. To simplify notation we   denote
$$dV=\bigwedge_{i} \left(dr_{i}\wedge  d\tau_{i}\right)\wedge \left(\bigwedge_{j}(\frac{\sqrt{-1}}{2}d\fs_{j}\wedge d\bar \fs_{j})\right)\wedge d\fkz_1\wedge \cdots \wedge d\fkz_{d} $$ and
 $$\delta_{i}=glu(\ft)_*\left(\tfrac{\p}{\p r_{i}}\right),\;\; \eta_{i}=glu(\ft)_*\left(\tfrac{\p}{\p \tau_{i}}\right),\;\;1\leq i \leq \mathfrak{e}$$

$$\delta_{\alpha}=glu(\ft)_*\left(\tfrac{\p}{\p \fs_{\alpha-\mathfrak{e}}}\right),\;\;\;\eta_{\alpha}=glu(\ft)_*\left(\tfrac{\p}{\p \bar\fs_{\alpha-\mathfrak{e}}}\right), \;\;\mathfrak{e}+1\leq \alpha\leq 3g-3+n$$ $$\varrho_{ \mathsf{i}}=glu(\ft)_*\left(\tfrac{\p}{\p \fkz_\mathsf{i}}\right),\;\;\;1\leq \mathsf{i} \leq d.$$
We will denote by $(E_1, E_2,...,E_{6g-6+2n+d})$
the frame
$$\left(
\delta_{1},...,\delta_{\mathfrak{e}},\eta_1,...,\eta_{\mathfrak{e}},
\delta_{\mathfrak{e}+1},...,\delta_{3g-3+n},\eta_{\mathfrak{e}+1},...,\eta_{3g-3+n},
\varrho_1,...,\varrho_d \right).$$
Then, for $\mathbf{a}=\mathbf{n}_{t}+1,\cdots,\mathbf{n}$,
\begin{align*}
 (\mathbf{J})_{\mathbf{a}}=&\int_{\widetilde{V}_\mathbf{a}}\widetilde{\mathbf{\Gamma}}_{\mathbf{a}}\cdot
 \left( \mathscr{P}^*\widetilde{K}\wedge\prod^n_i \widetilde{ev}^*_i\alpha_i\wedge \widetilde{\sigma}^{*}\widetilde{\Theta} (E_1,E_2,...,E_{6g-6+2n+d})\right)dV.
\end{align*}
{\bf 1. Estimates for $\mathscr{P}^*\widetilde{K}$ }
\v

We can choose $\fs,\ft,\mathfrak{z}_{1},\cdots,\mathfrak{z}_{d}$ as the local coordinates of $\mathbf{U}^T$.
In this coordinates  $\mathscr{P}:\mathbf U^{T}\to {\mathcal M}_{g,n}$ can be written as
$$\mathscr{P}(\fs,\ft,\mathfrak{z}_{1},\cdots,\mathfrak{z}_{d})=(\fs,\ft).$$   Noth that  $\mathscr{P}_{*}E_{i}=E_{i}$ for $i\leq 6g-6+2n$, $\mathscr{P}_{*}E_{i}=0$ for $i\geq 6g-6+2n+1.$ We assume that for any $1\leq j\leq \deg(K),$  $E_{i_j}\in \{E_{1},\cdots, E_{6g-6+2n}\}$. Since $K$ has Poincare growth we have
\begin{equation}\label{eqn_est_Lclass}
 |\mathscr{P}^* \widetilde{K}  (E_{i_{1}},\cdots,E_{i_{\deg(\widetilde{K} )}} )|= |\widetilde{K}   (E_{i_{1}},\cdots,E_{i_{\deg(\widetilde{K} )}} )|
\leq   C\left[\Pi_{j=1}^{\deg(\widetilde{K} )}  g_{loc} (E_{i_{j}},E_{i_{j}})\right]^{\frac{1}{2}}.\end{equation}
\v\n
{\bf 2. Estimates for $\prod^n_i \widetilde{ev}^*_i\alpha_i$ }
\v\n
  For any $p\in M$ and $\xi\in T_pM$ we denote
$D\exp_p(\xi): T_pM\rightarrow T_{\exp_{p}\xi}M,$ then
\begin{equation}\label{dExp}
D\exp_p(\xi)\xi':=\frac{d}{dt}\exp_{p}(\xi + t\xi')\mid_{t=0}.
\end{equation}
Obviously, $D\exp_p(\xi)$ is an isomorphism when $|\xi|$ small enough.  By a  direct calculation we have,  for any $X\in \{\frac{\p}{\p  s_{i}},\frac{\p}{\p \bar s_{i}},\frac{\p}{\p r_{l}}, \frac{\p}{\p \tau_{l}},\frac{\p}{\p \fkz_{j}},1\leq i\leq 3g-3+n-\mathfrak{e},1\leq l\leq \mathfrak{e},1\leq j\leq d\},$
\begin{align}\label{eqn_ev_glu}
 | (\widetilde{ev}_{i})_{*}(glu(\ft))_*X |=|\Pi_{2,u}(X(glu(\fs,\ft,\fkz))(y_{i}))|=\left| D\exp_{u}(\Pi_{2,u}\mathcal{L})(\Pi_{2,u}X (\mathcal{L}))(y_{i})) \right|.
\end{align}
By Theorem \ref{coordinate_decay-2} and \eqref{eqn_ev_glu}  we have
 $$
 \|\widetilde{ev}_{*}E_{i}\|_{G_{J}}+\|\widetilde{ev}_{*}E_{\mathfrak e+i}\|_{G_{J}}\leq Ce^{-\delta r_{i}} ,\;\;\;\;\; \|\widetilde{ev}_{*}E_{j}\|_{G_{J}}\leq C,
 $$
 $$\left[g_{loc} (E_{i},E_{i})\right]^{\frac{1}{2}}=\left[g_{loc} (E_{\mathfrak e+i},E_{\mathfrak e+i})\right]^{\frac{1}{2}}=\frac{2}{r_{i}},\;\;\;\;  \left[g_{loc} (E_{j},E_{j})\right]^{\frac{1}{2}}=1$$
for $1\leq i\leq \mathfrak{e}, \mathfrak{2e}+1\leq j\leq 6g-6+2n+d.$
It follows that
\begin{equation}\label{eqn_est_ev}
|\Pi \widetilde{ev}^*_i\alpha_i(E_{i_1},\cdots,E_{i_c})| \leq C \Pi_{i_{j}=i_{1}}^{i_{c}}\left[g_{loc} (E_{i_{j}},E_{i_{j}})\right]^{\frac{1}{2}},
\end{equation}
where $\{E_{i_{1}},\cdots,E_{i_{c}}\}\subset \{E_1, E_2,...,E_{6g-6+2n+d}\}.$
\v\n
{\bf 3. Estimates for the Thom form} By Lemma \ref{lem_Thom}  we have
\begin{equation}\label{eqn_est_sthom}
|\widetilde{\sigma}^*\widetilde{\Theta}(E_{i_{1}},\cdots,E_{i_{\mathsf r}})|\leq C \Pi_{i_1}^{i_{\mathsf r}}\left[g_{loc} (E_{i_{j}},E_{i_{j}})\right]^{\frac{1}{2}},
\end{equation}
where $\{E_{i_{1}},\cdots,E_{i_{\mathsf r}}\}\subset \{E_1, E_2,...,E_{6g-6+2n+d}\}.$
\v\v
It follows from \eqref{eqn_est_Lclass}, \eqref{eqn_est_ev} and \eqref{eqn_est_sthom} that
$$
\left|\mathscr{P}^*\widetilde{K} \wedge\prod_i \widetilde{ev}^*_i\alpha_i\wedge \widetilde{\sigma}^{*}\widetilde{\Theta} (E_{1},\cdots,E_{6g-6+2n+d})\right|\leq C \Pi_{i= 1}^{ 6g-6+2n+d}\left[g_{loc} (E_{i},E_{i })\right]^{\frac{1}{2}}\leq \frac{C}{\Pi_{i=1}^{\mathfrak{e}} r_{i}^2}.
$$
 Hence the integral $(\mathbf{I})_{\mathbf{a}}$ is convergence.

\section{\bf Properties of Gromov-Witten invariants}\label{properties of gromov-witten}

\subsection{\bf Some common properties}

Denote by $\mathbf{U}^T_n$ (resp.  $\mathbf{U}^T_{n-1}$) the top stratum of virtual neighborhood of $\overline{\mathcal M}_{g,n}(A)$ (resp. $\overline{\mathcal M}_{g,n-1}(A)$).
We begin with the smoothness of the forgetful maps and the evaluation maps.
\v
\begin{theorem} \label{thm_3.5}
Restricting to $\mathbf{U}^T$, the following hold:
\v
(a) the forgetful map $\chi$ is smooth,
\v
(b) suppose that $(g,n)\neq (0,3),(1,1)$, the map  $\pi: \mathbf{U}^T_n\rightarrow \mathbf{U}^T_{n-1}$ is smooth,
\v
(c) the evaluation map $ev_i$ is smooth.
\end{theorem}
\v\n
{\bf Proof.} (a) and (b).  Restricting to $\mathbf{U}^T$  and in the fixed coordinate system $(\psi, \Psi)$ for $\mathcal{Q}$,
we may choose $(a, y_1,...,y_n,\mathfrak{z})$ as local coordinates of $\mathbf{U}$ around $b_0$ as in subsection \S\ref{global_regu}. In this coordinates the maps $\mathscr{P}$ and $\pi$ are given by $\mathscr{P}(a, y_1,...,y_n,\mathfrak{z})=(a, y_1,...,y_n)$, $\pi(a, y_1,...,y_n, \mathfrak{z})=(a, y_1,...,y_{n-1},\mathfrak{z})$ respectively. It is obvious that $\mathscr{P}$ and $\pi$ are smooth. Note that for (a), since $n + 2g \geq 3$, after forgetting the map the domain is stable, for (b), since  $(g,n)\neq (0,3),(1,1)$ and we restrict to the top strata, forgetting the last marked point the domain is still stable.
(c). Since $\mathbf{U}^T$ is smooth, the evaluation map $ev_i$ is smooth.
$\Box$
\v\n
{\bf Proof of Theorem \ref{Common Properties}}

\v\n
(1) follows from the definition, we omit it.
\v\n
{ Proof of (2)}
\v
We only prove for $\alpha_i$, the proofs for $K$ and $\Theta$ are the same.
If $\alpha_{i}'\in [\alpha_{i}]$ is another closed form,
there is a  form $\vartheta_{i}$  such that
$
\alpha_{i}'-\alpha_{i}=d\vartheta_{i}.
$
\v
For any $\mathbf{n}_{t}+1\leq \mathbf{a}\leq \mathbf{n},$  let $\mathbf{U}_{\mathbf a,R}^{T} \subset \mathbf{V}_\mathbf{a}$ be an open set, defined  in terms of the coordinates system $(\fs,(\mathbf{r}),\mathfrak{z})$ by
$$\mathbf{U}_{\mathbf a,R}^{T}:=\{(\fs,(\mathbf{r}),\mathfrak{z})\mid r_i\leq R,i=1,\cdots,\mathfrak{e}_{\mathbf{a}}\}.$$
Put
$$\mathbf{U}_{\epsilon,R}'^{T}=\left(\cup_{\mathbf a=1}^{\mathbf n_{t}} \mathbf{V}_\mathbf{a}\right)\bigcup \left(\cup_{\mathbf a=\mathbf n_{t}+1}^{\mathbf n} \mathbf{U}_{\mathbf a,R}^{T}\right).$$
Then we smoothen it at corners, and denote the resulted neighborhood  by $\mathbf{U}_{\epsilon,R}^{T}$.
\v
By the Stokes theorem and Theorem \ref{coordinate_decay-2}, Theorem \ref{thm_est_mix_deri}
we have
\begin{align}\label{c_invariant-1}
&\int_{\mathbf{U}_{\epsilon,R}^{T}}\mathscr{P}^*K\wedge\prod_i ev^*_i\alpha_i\wedge \sigma^*\Theta -\int_{\mathbf{U}_{\epsilon,R}^{T}}\mathscr{P}^*K\wedge\prod_{j\neq i} ev^*_j\alpha_j\wedge ev^*_i\alpha_i'\wedge\sigma^*\Theta \\ =&\int_{\p\mathbf{U}_{\epsilon,R}^{T}}i^*\left(\mathscr{P}^*K\wedge\prod_{j\ne i}ev_{j}^*\alpha_j \wedge ev^*_i\vartheta_{i}\wedge \sigma^*\Theta\right)\rightarrow 0,\nonumber \;\;\;\;\;\;\;\mbox{ as } d\to 0,
\end{align}
where $i:\p\mathbf{U}_{\epsilon,R}^{T}\rightarrow \mathbf{U}^T$ is the inclusion map.
We explain $\int_{\p\mathbf{U}_{\epsilon,R}^{T}}i^*(\cdot)\to 0$:
 The area $Area$ of the hypersurface $\p\mathbf{U}_{\epsilon,R}^{T}$ satisfies
$|Area|\leq C'R$ for some constant $C'>0$. By the estimates in \S\ref{conver_GW} we have
$$\left|\int_{\p\mathbf{U}_{\epsilon,R}^{T}\cap V_{\mathbf a}}i^*(\cdot)\right|\leq \frac{C''R}{\Pi_{i=1}^{\mathfrak{e}_{\mathbf a}} R^2}\to 0,\;\;\;\;\;as\;R\to \infty.$$
Then (2) follows.
\v\n
{Proof of (3)}
\v
Suppose that $(\mathcal{H}', \mathfrak{i}'([\kappa',b]))$ is another choice and $(\mathbf{U}', \mathbf{E}', \sigma')$ is the virtual neighborhood constructed by $(\mathcal{H}', \mathfrak{i}'([\kappa',b]))$. Let $\Theta'$ be the Thom form of $\mathcal{H}'$ supported in a neighborhood of zero section. Let
$$\mathcal{S}_{(t)}([\kappa,b])
=\bar{\partial}_{j,J}v + (1-t)\mathfrak{i}([\kappa,b]) + t\mathfrak{i}'([\kappa',b]): \mathcal{H}\oplus \mathcal{H}'\times [0,1]\to \mathcal{E}.
$$
Let $(\mathbf{U}_{(t)},\mathbf{E}\oplus \mathbf{E}', \sigma_{(t)})$ be the virtual neighborhood cobordism constructed by $\mathcal{S}_{(t)}$. Using the same method as in (2), by Stokes theorem we have
$$\int_{\mathbf{U}^T_{0}}\mathscr{P}^*K\wedge\prod^n_i ev^*(0)_i\alpha_i\wedge \sigma_0^*(\Theta\wedge \Theta') -\int_{\mathbf{U}^T_{1}}\mathscr{P}^*K\wedge\prod^n_i ev^*(1)_i\alpha_i\wedge \sigma_1^*(\Theta\wedge \Theta')$$
$$=\int_{\mathbf{U}_{t}^T}d\left(\mathscr{P}^*K\wedge\prod^n_i ev^*(t)_i\alpha_i\wedge \sigma_{(t)}^*(\Theta\wedge \Theta')\right)=0,$$
where $ev^*(t):\mathbf{U}_t \longrightarrow
M$ is the evaluation map.
On the other hand, $\sigma_0=\sigma \times I_d$, $ \pi:\mathbf{U}^T_{0}\to \mathbf{U}^T$ is a bundle with fibre $\mathbf{E}'$. It follows that
$$\int_{\mathbf{U}^T_{0}}\mathscr{P}^*K\wedge\prod^n_i ev^*(0)_i\alpha_i\wedge \sigma_0^*(\Theta\wedge \Theta')=
\int_{\mathbf{U}^T}\mathscr{P}^*K\wedge\prod^n_i ev^*_i\alpha_i\wedge \sigma^*(\Theta).$$
By the same way we have
$$\int_{\mathbf{U}^T_{1}}\mathscr{P}^*K\wedge\prod^n_i ev^*(1)_i\alpha_i\wedge \sigma_1^*(\Theta\wedge \Theta')=\int_{\mathbf{U'}^T}\mathscr{P}^*K\wedge\prod^n_i ev^{'*}_i\alpha_i\wedge \sigma'^*(\Theta'),$$
where $ev^{'*}:\mathbf{U'}^T \longrightarrow
M$. Then (3) follows.
\v\n
{ Proof of (4)}
\v
Let $J'$(resp. $\omega'$) be another smooth almost complex structure (resp. symplectic form). Suppose that $ \mathbf{F}' $ is another choice of finite rank bundle. Let $\omega_{t}$ be a family of symplectic structures and $J_t$ be a family
of almost complex structures such that $J_t$ is tamed with $\omega_{t}$ and
$$J_0=J',\;\;\;J_1=J,\;\;\;\omega_0=\omega',\;\;\;\omega_1=\omega.$$ Let $D_{t}$ be the linearized operator  $\bar{\partial}_{j,J_{t}}.$ We cut  the interval $[0,1]$ into $[t_{i},t_{i+1}],0\leq i\leq l $ with $t_{0}=0,t_{l+1}=1,$ and construct a finit rank bundle $\mathbf K_{t} =\oplus_{i} \mathbf F_{t_{i}}$ with $\mathbf F_{0}=\mathbf F',\;\mathbf F_{1}=\mathbf F$. We choose  a smooth family  bundle map $\mathfrak{i}_{t},t\in [0,1] $ such that for any $t\in [0,1],$
  $  D_{t}+d\mathfrak{i}_{t}   $ is surjective and
  $$\mathfrak{i}_{0}=\mathfrak{i}',\;\;\;\mathfrak{i}_{1}=\mathfrak{i}.$$ Let $\mathcal{S}_{(t)}([\kappa,b])
=\bar{\partial}_{j,J_{t}}v + \mathfrak{i}_{t}([\kappa,b])$ and $(\mathbf{U}_{(t)}, \oplus \mathbf{E}_{t_{i}}, \sigma_{(t)})$ be the virtual neighborhood cobordism constructed by $\mathcal{S}_{(t)}$. Then by the same argument of (3) we can prove (4).
\v\n
\v\n
{Proof of (5)}.
\v
When M is semi-positive, we can use the same method of Ruan (\cite{R2}) to complete the proof. \;\;\;$\Box$
\v\n

\v\n\v\n
{\bf Proof of Theorem \ref{Common Properties-1}}
\v\n
We have a commutative diagram
$$\begin{array}{ccc}
forg_{g,n}: \overline{\mathcal{M}}_{g,n}(A)&\rightarrow &\overline{\M}_{g,n}\\
           \downarrow \pi          &             &\downarrow \pi\\
forg_{g,n-1}: \overline{\mathcal{M}}_{g,n-1}(A)&\rightarrow &\overline{\M}_{g,n-1}
\end{array}$$
\vskip 0.1in
\noindent
We construct virtual manifold $\mathbf{U}_{n-1}$ for $\overline{\mathcal{M}}_{g,n-1}(A)$. By pulling back of $\pi$, this is also as virtual manifold for $\overline{\mathcal{M}}_{g,n}(A)$. That is,
$$\mathbf{E}_{n}=\pi^*\mathbf{E}_{n-1},\;\;\mathbf{U}_{n}=\pi^*\mathbf{U}_{n-1},\;\;\mathcal{S}_{n}=
\mathcal{S}_{n-1}\circ \pi.$$
Furthermore, $\pi_*\mathscr{P}_{g,n}^*(K)=\mathscr{P}^*_{g,n-1}(\pi_*(K)).$
So
$$\begin{array}{lll}
\Psi_{(A,g,n)}(K; \alpha_1, \cdots,\alpha_{n-1}, 1)&=&\int_{\mathbf{U}_{n,\varepsilon}} \mathscr{P}_{g,
 n}^*(K)\wedge \prod^{n-1}_1 ev_i^*\alpha_i \wedge 1\wedge\Theta\\
&=&\Psi_{(A,g,n-1)}(\pi_*(K); \alpha_1, \cdots, \alpha_{n-1})
\end{array}.$$
On the other hand, for $\alpha_n\in H^2(M, \mathbb{R})$, one can check that
$\pi_*(ev^*_n(\alpha_n))=\alpha_n(A).$
Therefore,
$$\begin{array}{lll}
\Psi_{(A,g,n)}(\pi^*(K); \alpha_1, \cdots, \alpha_{n-1}, \alpha_n)&=&
\int_{U_{n}} \mathscr{P}_{g,
 n}^*(\pi^*(K))\wedge \prod^{n-1}_1 ev_i^*\alpha_i \wedge ev_n^*\alpha_n\wedge \Theta\\
&=&\alpha_n(A)\Psi_{(A,g,n-1)}(K; \alpha_1, \cdots, \alpha_{n-1}).
\end{array}.$$
$\Box$

\vskip 0.1in
\noindent
\subsection{\bf Axioms for Gromov-Witten invariants}

In \cite{KM} Kontsevich and Manin listed the following axioms for Gromov-Witten invariants.
\v\n
{\bf Effectivity Axiom.} If $\omega(A)<0$ then $\Psi_{(A,g,n)}=0$.
\v\n
{\bf Symmetry Axiom.} The symmetric group $S_n$ acts naturally on marked points. This axiom asserts that $\Psi_{(A,g,n)}$ is $S_n$-equivariant. This means that
$$\Psi_{(A,g,n)}(K; \alpha_{1},...,\alpha_{i},\alpha_{i+1},...,\alpha_{n})
=(-1)^{deg\alpha_ideg\alpha_{i+1}}
\Psi_{(A,g,n)}(K; \alpha_{1},...,\alpha_{i+1},\alpha_{i},...,\alpha_{n}).$$
\v\n
{\bf Grading Axiom.} If $\Psi_{(A,g,n)}(K; \alpha_{1},...,\alpha_{n})\ne 0$ then
$$\sum_{i=1}^n \deg \alpha_i + \deg K=2(1-g)(m-3) + 2c_1(A)+2n.$$
\v\n
{\bf Fundamental Class Axiom.} For any $\alpha _1, \cdots , \alpha _{n-1}$ in $H^*(M, \R)$,
$$\Psi_{A,g,n}(K;\alpha_1,...,
\alpha_{n-1},1)=\Psi_{A,g,n-1}(\pi_*(K); \alpha _1, \cdots,\alpha _{n-1}),$$
\v\n
{\bf Divisor Axiom.} If $(A,n)\ne (0,3)$ and $\deg\alpha_n=2$ then
$$\Psi_{A,g,n}(\pi^*(K); \alpha _1,
\cdots,\alpha _{n-1}, \alpha _n)=\alpha_n (A)
\Psi_{A,g,n-1}(K; \alpha _1, \cdots,\alpha _{n-1}).$$
\v\n
{\bf Zero Axiom.} If $A=0$ then $\Psi_{(A,g,n)}(K; \alpha_{1},...,\alpha_{n})=0$ whenever $deg K>0$, and
$$\Psi_{(A,g,n)}(PD([pt]); \alpha_{1},...,\alpha_{n})=\int_{M}\wedge_{i=1}^n \alpha_i.$$
\v\n
{\bf Deformation Axiom.} $\Psi_{A,g,n}(K;\alpha_1,...,
\alpha_{n})$ is independent of $J$ and is a symplectic deformation invariant.
\v\n
{\bf Splitting Axiom.} cf. Theorem \ref{split-3}.
\v
The Effectivity Axiom, the Symmetry Axiom, the Grading Axiom and the Zero Axiom are easy to prove. The Deformation Axiom is proved in Theorem \ref{Common Properties}. The Fundamental Class Axiom and the Divisor Axiom are proved in Theorem \ref{Common Properties-1}. In \S\ref{Splitting axiom} we state and prove the Splitting axiom.

\section{\bf Splitting axiom}\label{Splitting axiom}

\v

Assume $g = g_1+g_2$ and $n = n_1+n_2$ with $2g_i + n_i+1\geq 3$, $(g_i,n_i)\ne (1,0),n_{i}>0,i=1,2$.
Fix a partition of the index set $\{1,\cdots, n \}=S_{1}\cup S_{2}$, such that $n_i=|S_i|$ for $i=1,2$. We denote $\overline{\mathcal{M}}_{g_1,g_2,n_1,n_2}$ the moduli space   which identifies the last marked point of a stable curve
in $\overline{\mathcal{M}}_{g_{1},n_1+1}$ with the first marked point of a stable curve in $\overline{\mathcal{M}}_{g_{2},n_2+1}.$ Denote by $q$ the last marked point of of a stable curve
in $\overline{\mathcal{M}}_{g_{1},n_1+1}$.
 The
remaining indices have the unique ordering such that the relative order
is preserved, the first $n_1$ points in $\overline{\mathcal{M}}_{g_{1},n_1+1}$ are mapped to the points
indexed by $S_1$, and the last $n_2$ points in $\overline{\mathcal{M}}_{g_{2},n_2+1}$ are mapped to the points
indexed by $S_2$.
Let
$$
\theta:\overline{\mathcal{M}}_{g_1,g_2,n_1,n_2}\rightarrow \overline{\mathcal{M}}_{g,n}
$$
be the  map.
Clearly, $im(\theta)$ is a  submanifold of $\overline{\M}_{g,n}$.
\v
Let $\mathcal{C}$ be the set of all decomposition of $A=A_1+A_2$, $A_i\in H_2(M,\mathbb{Z})$. Given $C=(A_1,A_2)\in \mathcal{C}$, let $\overline{\mathcal{M}}_{C}(g_1,g_2,n_1,n_2)$ be the moduli space of all stable configuration $(\Sigma,{\bf y},\nu, j, u)$ with $(\Sigma,{\bf y},\nu, j)\in \overline{\mathcal{M}}_{g_{1},g_2,n_1,n_2},$  $[u_i(\Sigma_i)]=A_i$. Set
$$\overline{\mathcal{M}}_{A }(g_1,g_2,n_1,n_2)=\bigcup_{C\in \mathcal{C}}\overline{\mathcal{M}}_{C}(g_1,g_2,n_1,n_2).$$
Denote by $\overline{\mathcal{M}}_{g,n}(A,\theta)$  the moduli space of all stable configuration $(\Sigma,{\bf y},\nu, j, u)$ with homology class $A$ and $(\Sigma,{\bf y},\nu, j)\in \theta(\overline{\mathcal{M}}_{g_{1},g_2,n_1,n_2})$. Obviously, $\overline{\mathcal{M}}_{g,n}(A,\theta)\subset \overline{\mathcal{M}}_{g,n}(A)$. \;The map $\theta$ induces an isomorphism  between $\overline{\mathcal{M}}_{A }(g_1,g_2,n_1,n_2)$  and $\overline{\mathcal{M}}_{g,n}(A,\theta)$. We identify $\overline{\mathcal{M}}_{A }(g_1,g_2,n_1,n_2)$ with $\overline{\mathcal{M}}_{g,n}(A,\theta)$ if no danger of confusion. Denote by $\mathcal{M}_{A }(g_1,g_2,n_1,n_2)$ the top strata, the element of which have one node $q$.
\v

\subsection{\bf Constructing virtual neighborhoods}\label{const_virt_neig}

Let $C=(A_1,A_2)$. Consider $\overline{\mathcal{M}}_{C}(g_1,g_2,n_1,n_2)$.
Let
$$[b_o]=[\left(b_{o1}, b_{o2}\right)]\in \overline{\mathcal{M}}_{C}(g_1,g_2,n_1,n_2)$$
be a point. We view $[b_o]$ as a point in $\overline{\mathcal{M}}_{g,n}(A)$ and choose a local coordinate system and a local orbifold model
$\widetilde{\mathbf{O}}_{b_o}(\delta_o,\rho_o)/G_{b_o}$.
Denote by $\widetilde{\mathbf{O}}^{c}_{b_o}(\delta_o,\rho_o)/G_{b_o}$
 its restriction to $\overline{\mathcal{M}}_{C }(g_1,g_2,n_1,n_2)$.
 \v
Let $b_o=(\Sigma, j, {\bf y}, u)$. We write
 $(\Sigma, j, {\bf y}, u)=(\Sigma_1, j_1, {\bf y}_1, u_1)\cup (\Sigma_2, j_2, {\bf y}_2, u_2)$, $u=(u_1,u_2)$, where $u_1:\Sigma_1\to M$, $u_2:\Sigma_2\to M$ with $u_1(q)=u_2(q)$.
 The following lemma is obtained by the same method as before.
\begin{lemma}\label{finite cov_O}
There exist finite points $[b_i]\in \overline{\mathcal{M}}_{C }(g_1,g_2,n_1,n_2)$, $1\leq i \leq \mathfrak{m}_c$, such that
\begin{itemize}
\item[(1)] The collection $\{\mathbf{O}^{c}_{[b_i]}(\delta_i/3,\rho_i/3)
\mid 1\leq i \leq \mathfrak{m}_c\}$ is an open cover of $\overline{\mathcal{M}}_{C }(g_1,g_2,n_1,n_2)$.
\item[(2)] Suppose that $\widetilde{\mathbf{O}}^{c}_{b_i}(\delta_i,\rho_i)
\cap \widetilde{\mathbf{O}}^{c}_{b_j}(\delta_j,\rho_j)
\neq\phi$. For any $b\in \widetilde{\mathbf{O}}^{c}_{b_i}(\delta_i,\rho_i)
\cap \widetilde{\mathbf{O}}^{c}_{b_j}(\delta_j,\rho_j)$, $G_b$ can be imbedded into both $G_{b_i}$ and $G_{b_j}$ as subgroups.
\end{itemize}
\end{lemma}
\v\n
By Lemma \ref{finite cov} we have a continuous orbi-bundle $\mathbf{F}(\mathsf{k}_i)\rightarrow \mathcal{U}$ such that
$\mathbf{F}(\mathsf{k}_i)\mid_{b_i}$ contains a copy of group ring
$\mathbb{R}[G_{b_i}].$ Set
$$\mathcal{U}^{c}=\bigcup_{i=1}^{\mathfrak{m}_c}
\mathbf{O}^{c}_{[b_{i}]}(\delta_i,\rho_i).$$
For each $C\in \mathcal{C}$ we do this and put
$$\mathcal{U}_\theta=\bigcup_{C\in \mathcal{C}}\mathcal{U}^{c},\;\;\;\;\mathbf{F}'=\bigoplus_{C\in \mathcal{C}}\bigoplus_{i=1}^{\mathfrak{m}_c}\mathbf{F}(\mathsf{k}_i).
$$
Now we choose finite many points and local obifold models  $$[b_{\alpha}]\in \mathcal{M}_{g,n}(A),\;\; \widetilde{\mathbf{O}}_{b_{\alpha}}(\delta_{\alpha},\rho_{\alpha})/G_{b_{\alpha}},\;\;
1\leq \alpha \leq \mathfrak{m}_o$$
such that the collection $$\bigcup_{c\in \mathcal{C}}\{\mathbf{O}_{[b_i]}(\delta_i,\rho_i),\;\;1\leq i\leq \mathfrak{m}_c\}\bigcup \{\mathbf{O}_{[b_{\alpha}]}(\delta_{\alpha},\rho_{\alpha}),\;\;
1\leq \alpha \leq \mathfrak{m}_o\}$$
is an open cover of $\overline{\mathcal{M}}_{g,n}(A)$.
We have a continuous ``orbi-bundle" $\mathbf{F}(\mathsf{k}_{\alpha})\rightarrow \mathcal{U}$ such that
$\mathbf{F}(\mathsf{k}_{\alpha})\mid_{b_{\alpha}}$ contains a copy of group ring
$\mathbb{R}[G_{b_{\alpha}}]$ for any $1\leq \alpha\leq \mathfrak{m}_o$. Put
$$\mathbf{F}=\bigoplus_{\alpha=1}^{\mathfrak{m}_o}\mathbf{F}(\mathsf{k}_{\alpha})
\bigoplus \mathbf{F}'.$$
Define a bundle map
$\mathfrak{i}:  {\mathbf{F}}\rightarrow \mathcal{E}$ as in \S\ref{global_regu}.
We define a global regularization for $\overline{\mathcal{M}}_{g,n}(A)$ to be the bundle map $\mathcal{S}:\mathbf{F}\to \E$
by
$$\mathcal{S}([\kappa,b])
=[\bar{\partial}_{j,J}v] + [\mathfrak{i}(\kappa,b)].
$$
Denote
$$\mathbf{U}=\mathcal{S}^{-1}(0)|_{\mathcal{U}}.$$
There is a bundle of finite rank $\mathbf{E}$ over $\mathbf{U}$ with a canonical section $\sigma$. We have a virtual neighborhood for $\overline{\mathcal{M}}_{g,n}(A)$:
$$(\mathbf{U},\mathbf{E},\sigma).$$
The map $\theta$ induce a bundle $\pi: \theta^{*}\mathbf F\to \mathcal U^{c}$ and a bundle map $\theta^*\mathcal S :\theta^{*}\mathbf F\to \mathcal E.$ Then restricting on $\mathcal{U}^c$ we have
a virtual neighborhood for $\overline{\mathcal{M}}_{C }(g_1,g_2,n_1,n_2)$:
$$(\mathbf{U}_{c},\mathbf{E}_{c},\sigma_{c}).$$
One can check that
$$
 \mathbf{U}_{c}=\left.\theta^{*}\mathbf{U}\right|_{\mathcal U^{c}},\;\;\;\;\mathbf{E}_{c}=\left.\theta^*\mathbf{E}\right|_{\mathbf{U}_{c}}.$$
Restricting on $\mathcal{U}_\theta$ we have
a virtual neighborhood for $\overline{\mathcal{M}}_{A }(g_1,g_2,n_1,n_2)$:
$$(\mathbf{U}_{\theta},\mathbf{E}_{\theta},\sigma_{\theta}).$$
Then
$$
 \mathbf{U}_{\theta}=\bigsqcup_{C\in \mathcal C} \mathbf{U}_{c},\;\;\;\; \mathbf{E}_{\theta}|_{\mathbf{U}_{c}}=  \mathbf{E}_{c},\;\;\;\;\;\sigma_{\theta}|_{\mathbf{U}_{c}}=  \sigma_{c}.
$$
Denote by $\mathbf{U}_{c}^T$ ( resp. $\mathbf{U}_{\theta}^T,\mathbf{U}^T$ ) the top strata of $\mathbf{U}_{c}$ (resp. $\mathbf{U}_{\theta},\mathbf{U}$ ). The element of $\mathbf{U}_{c}^T$ has the form
$$\left((\Sigma_1, \kappa_1, j_1, {\bf y}_1, u_1),\;(\Sigma_2, \kappa_2, j_2, {\bf y}_2, u_2)\right),$$
where $u_1:\Sigma_1\to M$, $u_2:\Sigma_2\to M$ with $u_1(q)=u_2(q)$.
Set
$$\mathbf{U}_{c,\varepsilon}=\{(\kappa,b)\in \mathbf{U}_{c} | |\kappa|_{\mathbf h}\leq \varepsilon\},\;\;\;\;\;\;\;\;\mathbf{U}_{\theta,\varepsilon}=\{(\kappa,b)\in \mathbf{U}_{\theta} | |\kappa|_{\mathbf h}\leq \varepsilon\}.$$
Let $(\mathbf U_{ic},\mathbf E_{ic},\sigma_{ic}),i=1,2$ be the virtual neighborhood of $\overline{\mathcal{M}}_{A_{i}}(g_i,n_{i}+1),i=1,2,$ where $A=A_{1}+A_{2}$, $C=(A_1,A_2)$. Then
$$
\mathbf U_{c}=\{(b_{1},b_{2})\in \mathbf U_{1c}\times \mathbf U_{2c} | ev^{1c}_{n_{1}+1}(b_{1})=ev^{2c}_{n_{2}+1}(b_{2})\},\;\;\;\mathbf E_{c}=\mathbf E_{1c}\times \mathbf E_{2c}|_{\mathbf U_{c}},\;\;\;\sigma_{c}=(\sigma_{1c},\sigma_{2c})|_{\mathbf U_{c}}.
$$

By the same method as in Theorem \ref{Smooth} we have
\begin{theorem} \label{Smooth theorem-1}
$\mathbf{U}_{c}^T$, $\mathbf{U}^T_{1c}$, $\mathbf{U}^T_{2c}$  and $\mathbf{U}_{\theta}^T$ are smooth oriented effective orbifolds.
\end{theorem}
\v\n
Note that $\{\mathbf{\Gamma}_{\mathbf{a}}\}$ are smooth on $\mathbf{U}_{c}^T$, $\mathbf{U}^T_{1c}$ and $\mathbf{U}^T_{2c}$, and $\mathbf{\Gamma}_{\mathbf{a}}\Theta$ is compactly supported.
 For any $[K]=[K_1\times K_2]\in H^*( \overline{\mathcal{M}}_{g_1,g_2,n_1,n_2} , \mathbb{R})$, let $[K]\in H^*( \overline{\mathcal{M}}_{g_1,g_2,n_1,n_2} , \mathbb{R})$. We take a Thom form $\Theta$ supported in a small $\varepsilon$-ball of the $0$-section of $\mathbf{E}$. Let $\Theta_{c}=\theta^*\Theta \mid_{\mathbf{U}_c}.$
 Then we define the GW-invarians $\Psi_{(A, g_1,g_2,n_1,n_2)}(K_1\times K_2; \{\alpha_i\})$  as
\begin{equation}\label{def_GW_sp}
\Psi_{(A, g_1,g_2,n_1,n_2)}(K_{1}\times K_{2}; \{\alpha_i\})=\sum_{C\in \mathcal C}
\int_{ \mathbf U^T_{c,\varepsilon} }\mathscr{P}^*(  K)\wedge \prod_j ev'^*_j\alpha_j  \wedge \sigma_{c}^*\Theta_{c},
\end{equation}
where   $\Theta_{c}$  is the Thom form of   $\mathsf{p}:\mathbf E_{c}\to \mathbf U_{c},$  $ev'_j$ denote the evaluation map
  $ev'_j: {\bf U'}_{c,\varepsilon} \longrightarrow
M $  at $j$-th marked point.
     Using the same method as in Theorem \ref{Conver} we can prove that the integrals are convergence.
   We can also define
$\Psi_{(A_i,g_i,n_i+1)}(K_i; \{\alpha_i\}_
{i\in S_{1}})$ for $\overline{\mathcal{M}}_{A_i} (g_i,n_i+1)$, $i=1,2$.

\v
\vskip 0.1in
\noindent

Next we define the invariant  $\Psi_{(A,g,n)}(\theta_{!}(K_{1}\times K_{2}); \{\alpha_i\}). $ First we define the following transfer map

\begin{definition} Suppose that $X, Y$ are two topological space such that
Poincare duality holds over $\mathbb{R}$. Let $f: X\rightarrow Y$. Then, the transfer
map
$$f_{!}: H^*(X, \mathbb{R})\rightarrow H^*(Y, \mathbb{R})$$
is defined by $f_{!}(K)=PD(f_*(PD(K)))$.
\end{definition}

  We can identify a tubular neighborhood $\mathbb{O}$ of $im(\theta)$ with a neighborhood of
zero section of the normal bundle $\mathcal{N}$ of $im(\theta)$ in $\overline{\mathcal{M}}_{g,n} $.
Let $im(\theta)^*$ be the Thom form of the bundle
$$\pi:\mathcal N\to im(\theta),$$  which can be chosen to be supported
in the tubular neighborhood $\mathbb{O}$ of $im(\theta)$ and  $im(\theta)^*$ can be seen as  the Poincare dual of $im(\theta)$.
For any $[K]=[K_1\times K_2]\in H^*( \overline{\mathcal{M}}_{g_1,g_2,n_1,n_2} , \mathbb{R})$,
choose $  K_{i}\in H^{*}(\overline{\mathcal M}_{g_{i},n_{i}+1},\mathbb R)$.
Then $  K=( {K}_{1}, {K}_{2})\in\overline{\mathcal M}_{g_{1},g_{2},n_{1},n_{2}}$.  Let $K_{\overline{\M} } $ be the Poincare dual of $ \theta_*(PD( K))$ in $\theta(\overline{\mathcal{M}}_{g_1,g_2,n_1,n_2})$.
 Through $\pi$ we can pull $K_{\overline{\M} }$ back to the total space of the normal bundle $\mathcal{N}$ , denoted  by $\pi^{*}K_{\overline{\M}}$. Then, $\pi^{*}K_{\overline{\M}}$ is defined over a tubular
neighborhood of $im(\theta)$. Since $im(\theta)^*$ is supported in the tubular neighborhood,
$im(\theta)^*\wedge\pi^{*} K_{\overline{\M} }$
is a closed differential form defined over $\overline{\M}_{g,n}$. One can check that
\begin{equation}\label{vir_DM_sp}
\theta_{!}[ K]=[im(\theta)^*\wedge \pi^{*}K_{\overline{\M} }],\;\;\;\; \pi_{*}[\theta_{!}( K)]=\theta_{!}[K].
\end{equation}
Then we have
\begin{equation}\label{def_GW_tr}
\Psi_{(A,g,n)}(\theta_{!}(K_{1}\times K_{2}); \{\alpha_i\})=
\int_{{\mathbf U}_{\varepsilon}^{T}}\mathscr{P}^* \left(im(\theta)^*\wedge \pi^{*}K_{\overline{\M} }\right)\wedge\prod_j ev^*_j\alpha_j\wedge \sigma^*\Theta.
\end{equation}
where $\Theta$  is the Thom form of $\mathsf{p}:\mathbf E\to \mathbf U$,  $ev_j$ denote the evaluation map
$ev_j: \mathbf{U}^T_{\varepsilon} \longrightarrow
M $   at $j$-th marked point. As in the proof of Theorem \ref{Conver} we can prove the convergence of $\Psi_{(A,g,n)}(\theta_{!}(K_{1}\times K_{2}); \{\alpha_i\}).$
By the same argument of Theorem \ref{Common Properties} we can prove that
\begin{lemma}\label{lem_sp_9.3}
 $\Psi_{(A,g,n)}(\theta_{!}(K);\{\alpha _i\})$ is independent of the choice of $im(\theta)^{*}.$
\end{lemma}
\v

\subsection{A gluing formula}
In this subsection our main purpose is to prove the following theorem:
\begin{theorem}\label{split-2} For any $K_1\times K_2 \in H^*(\overline{\mathcal{M}}_{g_1,g_2,n_1,n_2}, \mathbb{R})$,
$\alpha _1,\cdots,\alpha _n \in H^*(M,\mathbb{R})$, represented by smooth forms, we have
$$\Psi_{(A,g,n)}( \theta_{!}(K_{1}\times K_{2}); \{\alpha_i\})=\Psi_{(A, g_1,g_2,n_1,n_2)}(K_{1}\times K_{2}; \{\alpha_i\}).$$
\end{theorem}
\begin{remark}\label{split-3}
If $\mathbf {U}_{\varepsilon}$ is smooth, by using \eqref{vir_DM_sp} we can prove Theorem \ref{split-2} directly. But what we know only the smoothness of the top strata $\mathbf{U}^T$ and $\mathbf{U}_{\theta}^T$. So we use our estimates in Section \S\ref{estimates} to prove this theorem.

\end{remark}
\v
As in \S\ref{conver_GW} we choose finite many points  $P=\{[(\kappa_{\mathbf{a}},b_{\mathbf{a}})]\in \mathbf{U}_{\theta,2\epsilon}| \; \mathbf{a}=1,\cdots, \mathbf{n}_{\theta} \}$
with
$$[(\kappa_{\mathbf{a}},b_{\mathbf{a}})]\in \mathbf  U_{\theta}^{T},\;\mathbf{a}=1,\cdots,\mathbf{n}_{t},\;\;\;\;\;\;\;\;[(\kappa_{\mathbf{a}},b_{\mathbf{a}})]\in \mathbf {U}_{\theta}\setminus\mathbf {U}_{\theta}^{T},\;\mathbf{a}= \mathbf{n}_{t}+1,...,\mathbf{n}_{\theta}$$
such that $\{\mathbf{U}_{\theta,[(\kappa_{\mathbf{a}},b_{\mathbf{a}})]}(\varepsilon_\mathbf{a},\delta_\mathbf{a},\rho_\mathbf{a})\;\;\mathbf{a}=1,...,\mathbf{n}_\theta\}$
is an open cover of $\mathbf {U}_{\theta,2\epsilon}$ and $\overline{\mathbf{U}}_{\theta,[(\kappa_\mathbf{a},b_\mathbf{a})]}\subset \mathbf U^{T}_{\theta}$ for all $1\leq \mathbf{a}\leq \mathbf{n}_{t}$.
Choose $\varepsilon_{\mathbf{a}},\delta_{\mathbf{a}},\rho_{\mathbf{a}},\iota$ small  such that
$$glu: (\mathbf{U}_{\theta,[(\kappa_{\mathbf{a}},b_{\mathbf{a}})]}(\varepsilon_\mathbf{a},\delta_\mathbf{a},\rho_\mathbf{a})\cap \mathbf{U}^T_{\theta})\times \mathbb D^{*}_{\iota}(0)\rightarrow glu\left((\mathbf{U}_{\theta,[(\kappa_{\mathbf{a}},b_{\mathbf{a}})]}(\varepsilon_\mathbf{a},\delta_\mathbf{a},\rho_\mathbf{a})\cap \mathbf{U}^T_{\theta})
\times \mathbb D^{*}_{\iota}(0)\right)$$
is an orientation preserving diffeomorphism in orbifold sense.
To simplify notations we denote
$${V}_\mathbf{\theta,\mathbf a}=\mathbf{U}_{\theta,[(\kappa_{\mathbf{a}},b_{\mathbf{a}})]}(\varepsilon_\mathbf{a},\delta_\mathbf{a},\rho_\mathbf{a})
  \cap \mathbf{U}^T_{\theta},\;\;\;\;  W_{\mathbf{a}}=glu( {V}_{\theta,\mathbf{a}}\times \mathbb D_{\iota}(0)),$$
$$U_{\theta,\mathbf{a}}=\mathscr{P}(\mathbf{U}_{\theta,[(\kappa_{\mathbf{a}},b_{\mathbf{a}})]}(\varepsilon_\mathbf{a},\delta_\mathbf{a},\rho_\mathbf{a})).$$

For any $d$ denote   $$\mathbf{U}^{T,\theta}_{d}=\bigcup_{\mathbf a}glu({V}_\mathbf{\theta,\mathbf a}
\times \mathbb D^{*}_{d}(0)).$$
We fix a number $\iota,$ let $d$ be a small constant with  $0<d<\iota/3$.

 The map $\theta$ induces a embedding $\mathbf{U}_{\theta}\to \mathbf{U}$.
 We   choose  $im(\theta)^*$  such that
\begin{equation}\label{eqn_im_cond}
supp\left(\mathscr{P}^*\left(im(\theta)^*\wedge \pi^{*}K_{\overline{\M} }\right)\cap \mathbf {U}^{T}\right)\subset \mathbf{U}^{T,\theta}_{d}.
\end{equation}

\v
 We can choose a partion of unit  $\{\mathbf{\Gamma}_{\mathbf{a}},\;1\leq \mathbf{a}\leq \mathbf{n}_\theta \}$ of $\mathbf U_{\theta}$ as in \S\ref{a metric}.
Let $\beta_{\iota}$ be tha cut-off function satisfying
$\beta_{\iota}|_{D_{\iota/3}(0)}=1,\;\;supp\;\beta_{\iota}\subset D_{\iota}(0).$
For any $1\leq \mathbf{a}\leq \mathbf n_{\theta},$ let
$\pi_{\mathbf a}: V_{\theta,\mathbf a}\times D_{\iota}(0)\to  V_{\theta,\mathbf a}$ be the projection.  $\beta_{\iota}$ can be naturally seen as a function on $V_{\theta,\mathbf a}\times D_{\iota}(0).$ Set
$$
\hat{\mathbf \Gamma}_{\mathbf a}=(glu^{-1})^{*}(\beta_{\iota}\pi_{\mathbf a}^{*}\mathbf{\Gamma}_{\mathbf a}).
$$
Using Theorem \ref{coordinate_decay-2} and by a direct calculation we have
\begin{equation}\label{eqn_gam'_2}
\left|\frac{\p \hat{\mathbf {\Gamma}}_{\mathbf{a}}}{\p r}\right|\leq Ce^{-\mathfrak{c}_{1}r}.
 \end{equation}
where we used the smoothness of cut-off function.
\v
We can choose finite many points $[(\kappa_{\mathbf{a}'},b_{\mathbf{a}'})]\in \mathbf U_{\epsilon}\setminus \mathbf{U}^{T,\theta}_{\iota/3},$ $\mathbf n_{\theta}+1\leq \mathbf a' \leq \mathbf n$, and choose $\varepsilon_{\mathbf a'},\delta_\mathbf{a'},\rho_\mathbf{a'}$ small such that
\begin{itemize}
\item[(1)]  $ \{glu( {V}_\mathbf{a}\times \mathbb D_{\iota/3}(0)),\mathbf{U}^{T}\cap\mathbf{U}_{[(\kappa_{\mathbf{a'}},b_{\mathbf{a'}})]}(\varepsilon_\mathbf{a'},\delta_\mathbf{a'},\rho_\mathbf{a'}) , \mathbf a\leq \mathbf n_{\theta}, \mathbf n_{\theta}+1\leq \mathbf a' \leq \mathbf n\}$ is an open covering of $\mathbf{U}^T_{\epsilon}$
\item[(2)]  $\mathbf{U}_{[(\kappa_{\mathbf{a'}},b_{\mathbf{a'}})]}(\varepsilon_\mathbf{a'},\delta_\mathbf{a'},\rho_\mathbf{a'}) \cap \mathbf{U}^{T,\theta}_{d}=\emptyset$ as $d$ small enough for any $\mathbf n_{\theta}+1\leq \mathbf a' \leq \mathbf n$.
    \end{itemize}
As in section \S\ref{a metric} we can construct  finite many cut-off functions $\hat{\mathbf{\Gamma}}_{\mathbf a'}$ supported in $\mathbf{U}_{[(\kappa_{\mathbf{a'}},b_{\mathbf{a'}})]}(\varepsilon_\mathbf{a'},\delta_\mathbf{a'},\rho_\mathbf{a'})$, $ \mathbf n_{\theta}+1\leq \mathbf a' \leq \mathbf n$, satisfying
$$\sum_{\mathbf a'=n_{\theta}+1}^{\mathbf n} \hat{\mathbf{\Gamma}}_{\mathbf a'}|_{\mathbf U^{T}_{\epsilon}\setminus \mathbf{U}^{T,\theta}_{\iota/3}}>0.$$
By \eqref{eqn_gam'_2} we have $\sum_{\mathbf a'=1}^{\mathbf n} \hat{\mathbf{\Gamma}}_{\mathbf a'}|_{\mathbf U^{T}_{\epsilon}}>0$ as $\iota$ small enough. Then $\{\hat{\mathbf{\Gamma}}_{\mathbf{a}},1\leq \mathbf{a}\leq \mathbf n\}$ induces a partition of unity $\{ \mathbf{\Gamma}'_{\mathbf{a}},1\leq \mathbf{a}\leq \mathbf n\}$ of $\mathbf U_{\epsilon}$ defined by
 $$
 \mathbf{\Gamma}'_{\mathbf{a}}=\frac{ \hat{\mathbf{\Gamma}}_{\mathbf{a}}}{\sum_{1\leq \mathbf a\leq \mathbf n} \hat{\mathbf{\Gamma}}_{\mathbf{a}}}.
 $$
  By (2) it is easy to see that
 $\{ \mathbf{\Gamma}'_{\mathbf{a}},1\leq \mathbf{a}\leq \mathbf n_{\theta}\}$ is a partition of unity of
 $\mathbf{U}^{T,\theta}_{d}$ and in $\mathbf{U}^{T,\theta}_{d}$
\begin{equation}\label{eqn_gam'}
\mathbf{\Gamma}'_{\mathbf{a}}=\frac{ \hat{\mathbf{\Gamma}}_{\mathbf{a}}}{\sum_{1\leq \mathbf a\leq \mathbf n_{\theta}} \hat{\mathbf{\Gamma}}_{\mathbf{a}}}.
 \end{equation}

\v

We use $\mathscr{P}$ to denote both $\mathscr{P}: \mathbf {U}_{c,\varepsilon}\to \overline{\M}'_{g_1,g_2,n_{1},n_{2}}$ and
$\mathscr{P}: \mathbf U_{\varepsilon}\to \overline{\M}_{g,n}.$
We have
\begin{equation}\label{eqn_Im_Kcon}
supp \left( \mathbf{\Gamma}'_{\mathbf{a}}\cdot \mathscr{P}^{*}im(\theta)^{*}\wedge \sigma^* \Theta \right)\subset   W_{\mathbf{a}},\;\;\;\forall \mathbf a\leq \mathbf n_{\theta}.
\end{equation}

\v

\n
{\bf Proof of Theorem \ref{split-2}.} Denote
\begin{align*}
(A)&=\sum_{C\in \mathcal C}\int_{ \mathbf U^T_{c,\varepsilon} }\mathscr{P}^*( K)\wedge \prod_j ev'^*_j\alpha_j  \wedge \sigma^*_{c}\Theta_{c}-\sum_{C\in \mathcal C}\int_{ \mathbf U^T_{c,\varepsilon} }\mathscr{P}^*(\theta^* K_{\overline{\mathcal M}})\wedge \prod_j  {ev'}^*_j\alpha_j  \wedge \sigma^*_{c}\theta^*\Theta ,\\
(B)&=  \int_{ \mathbf U^T_{ \varepsilon} }{\mathbb{F}}_{r}-\sum_{C\in \mathcal C}\int_{  \theta(\mathbf U^T_{c,\varepsilon})} \mathscr{P}^*(K_{\overline{\mathcal M}})\wedge \prod_j {ev}^*_j\alpha_j  \wedge \sigma^* \Theta .
\end{align*}
where
$$ {\mathbb{F}}_{r}=\mathscr{P}^* \left(im(\theta)^*\wedge \pi^{*} {K}_{\overline{\M} }\right)\wedge\prod_j ev^*_j\alpha_j\wedge \sigma^* {\Theta}.$$
 Note that $  {ev}_j \cdot\theta={ev'}_j, \; \theta \cdot \sigma_{c}=\sigma \cdot \theta,\;\mathscr{P}\cdot\theta=\theta\cdot\mathscr{P}$.
We have
$$
\int_{  \theta(\mathbf U^T_{c,\varepsilon})} \mathscr{P}^*(K_{\overline{\mathcal M}})\wedge \prod_j      {ev}^*_j\alpha_j  \wedge  \sigma^* \Theta=\int_{ \mathbf U^T_{c,\varepsilon} }\mathscr{P}^*(\theta^* K_{\overline{\mathcal M}})\wedge \prod_j  {ev'}^*_j\alpha_j  \wedge \sigma^*_{c}\theta^*\Theta.
$$
By \eqref{def_GW_sp} and \eqref{def_GW_tr}, we only need to prove that $(A)-(B)=0.$
Since $ K$ and $\theta^* {K}_{\mathcal M}$ are   in the same cohomology, we have $(A)=0$, so it suffices to prove $(B)=0$.
\v

\v
For $1\leq \mathbf{a}\leq \mathbf{n}_{t},$  we choose   $\fs$ as a local coordinates of $U_{\theta,\mathbf{a}}.$
Let $\ft_{o}=e^{-2r-2\pi\sqrt{-1}\tau}$ be the gluing parameter at node $q.$ Then  $(\ft_{o},\fs)$ is a local coordinates of  $ \pi^{*} U_{\theta,\mathbf{a}}$.  On the other hand, since the bundle $\mathcal{N}$ has a Riemannian structure, we can
choose a smooth orthonormal frame field. This defines a coordinate $\mathfrak{y}$ over fiber.
Denote   $\mathfrak{y}=e^{-2\hat r-2\pi \sqrt{-1}\hat \tau}$ and  $\hat{\fs}=\pi^*\fs.$   Then $(\mathfrak{y},\hat{\fs})$ is also a local coordinates of  $\mathbb O_{d}\cap \pi^{*} U_{\theta,\mathbf{a}}$. Denote the Jacobi matrix  by $(a_{ij})=\frac{\p(t_{0},\fs)}{\p (\mathfrak{y},\hat{\fs})}$.
Since $\overline{\mathcal M}^{red}_{g,n}$ is a smooth orbifold, $(a_{ij})$ and the inverse matrix $(a^{-1}_{ij})$ are uniform bounded in the coordinates.
Then
 $(\ft_{o},\mathbf{s}, \fkz )$ and $(\mathfrak{y},\hat{ \mathbf{s}},\fkz)$  are the local coordinates of $ W_{\mathbf{a}},$ where $\fkz=(\fkz_{1},\cdots,\fkz_{d}).$
We have  the  coordinates tranformation
$$
\fs=\fs(\mathfrak{y},\hat{\fs}),\;\;\ft_{o}=\ft_{o}(\mathfrak{y},\hat{\fs}),\;\;\;\;\fkz_{j}=\fkz_{j}.
$$
 Denote by $(b_{ij})=\frac{\p (\mathfrak{y},\hat{\fs}, \fkz)}{\p(\ft_{o} ,\fs,\fkz)}.$
It follows from the bound  of $(a_{ij})$ and $(a^{-1}_{ij})$ that   $(b_{ij})$ and the inverse matrix $(b^{-1}_{ij})$ are uniform bounded in the coordinates.
\v
In each $W_{\mathbf{a}}$,  the map $\pi: \overline{\mathcal M}_{g,n}\to\overline{\mathcal M}'_{g_{1},g_{2},n_{1},n_{2}} $ induce a map $\pi:W_{\mathbf{a}}\to  V_{\theta,\mathbf{a}}$ defined by
$$(\mathfrak{y},\hat{\fs},\fkz )\to (\fs,\fkz),$$
Let $$(B_{\mathbf{a}})=\int_{ W_{\mathbf{a}} }\pi^{*} \mathbf{\Gamma}_{\mathbf{a}}\mathbb{H}_{r}-\int_{  \theta( V_{\theta,\mathbf{a}})}
\mathbf{\Gamma}_{\mathbf{a}}
\mathscr{P}^*(K_{\overline{\mathcal M}})\wedge \prod_j  ev^*_j\alpha_j  \wedge  \sigma^* \Theta.$$
Then $(B)=\sum (B_{\mathbf{a}})+\sum \int_{ W_{\mathbf{a}} } (\mathbf{\Gamma}'_{\mathbf{a}}-\pi^{*}  \mathbf{\Gamma}_{\mathbf{a}} )\mathbb{H}_{r}.$ By \eqref{eqn_im_cond}, \eqref{eqn_gam'}, \eqref{eqn_gam'_2}  and the bound of matrix $(b_{ij}), (b^{-1}_{ij})$ we obtain
$$
\left|\int_{ W_{\mathbf{a}} } (\mathbf{\Gamma}'_{\mathbf{a}}-\pi^{*}  \mathbf{\Gamma}_{\mathbf{a}} )\mathbb{H}_{r}\right|\leq Cd^{\mathfrak{c}_{1}}.
$$
We only need estimate $(B_{\mathbf a})$.

\v
We recall the expression of $im(\theta)^*$, the detail can be found in \cite{bott}. Let $\Gamma$ be an increasing  function of the radius $|\mathfrak{y}|$ such that
 \begin{equation}\label{def_Gamma}
  \int_{\mathbb{R}^+}d\Gamma    =1,\;\;\; \Gamma(0)=-1,\;\;\;|\Gamma|\leq 1,\;\;\;d\Gamma \mbox{ is a compact support form}.
 \end{equation}
 Let $\phi_{\mathbf{a}\mathbf{c}}:U_{\mathbf{a}}\cap U_{\mathbf{c}}\to S^{1}$ the transformation function of $\mathcal N.$  Then
\begin{equation}\label{Thom_form_N}
im(\theta)^*=  d\Gamma \wedge \psi_{\mathcal N}-\Gamma \pi^{*}(e(\mathcal N)),\;\;\;\; e(\mathcal N)=\frac{\sqrt{-1}}{2\pi}\sum_{\mathbf{c}} d(\Gamma_{\mathbf{a}}d\log \phi_{\mathbf{a}\mathbf{c}}).
\end{equation}

\v
Denote  $i:\theta(\mathbf U_{c})\to \mathbf U $ is the inclusion map.
Then $(B_{\mathbf{a}})$ can be re-written as
$$
\int_{W_{\mathbf{a}} }\pi^*
\mathbf{\Gamma}_{\mathbf{a}}
\mathscr{P}^*  im(\theta)^*\wedge\left(\pi^{*}K_{\overline{\M} }\wedge \prod_j ev^*_j\alpha_j\wedge \sigma^*\Theta-\pi^{*}i^{*} \left( \pi^{*}K_{\overline{\M} }\wedge\prod_j   ev^*_j\alpha_j  \wedge \sigma^* \Theta\right)\right).
$$
As in the proof of Theorem \ref{Conver} denote by $(E_1,  ...,E_{6g-6+2n+\mathsf d})$ (resp.  $(\hat E_1,  ...,\hat E_{6g-6+2n+d})$) the induced vector field by $(\ft_{o},\mathbf{s},\fkz)$ (resp. $(\mathfrak{y},\hat{\fs},\fkz)$).  Set
  $$d V_{\theta}= \bigwedge_{j}\left(\frac{\sqrt{-1}}{2}d\hat{\fs}_{j} \wedge d\bar {\hat{\fs}}_{j}\right)  \wedge d\fkz_1\wedge \cdots \wedge d\fkz_{d}$$
 Then  $\pi^{*}K_{\overline{\M}}\wedge\prod_j  ev^*_j\alpha_j  \wedge \sigma^* \Theta$ can be written as $$f_{1}(\mathfrak{y},\hat{\fs},\fkz)d\hat V+f_{2}(\mathfrak{y},\hat{\fs},\fkz)\wedge d\hat r \wedge \hat \tau,$$
where
 $
f_{1} =\pi^{*}K_{\overline{\M}}\wedge\prod_j ev^*_j\alpha_j  \wedge \sigma^* \Theta(\hat E_{3},\cdots,\hat E_{6g-6+2n+\mathsf d}).$
Since
$$\pi^{*}i^{*}f_{1}(\mathfrak{y},\hat{\fs},\fkz)=f_{1}(0,\hat{\fs},\fkz),\;\;\;\;\pi^{*}i^{*}(f_{2}(\mathfrak{y},\hat{\fs},\fkz)\wedge d\hat r \wedge \hat \tau)=0,$$
by \eqref{Thom_form_N} we have
\begin{align}
(I)_{\mathbf a}=\int_{ W_{\mathbf{a}} }\pi^*
\mathbf{\Gamma}_{\mathbf{a}}\left[(f_{1}(\mathfrak{y},\hat{\fs},\fkz)-f_{1}(0,\hat{\fs},\fkz))d\hat V\wedge d\Gamma \wedge \psi_{\mathcal N}  \right] \\
(II)_{\mathbf a}=\int_{ W_{\mathbf{a}} }\pi^*
\mathbf{\Gamma}_{\mathbf{a}} \Gamma f_{2}(\mathfrak{y},\hat{\fs},\fkz) \mathscr{P}^{*} \pi^{*}(e(\mathcal N))\wedge d\hat{r} \wedge \hat{\tau}
\end{align}
As in the proof of Theorem \ref{Conver}, using \eqref{def_Gamma}, Theorem \ref{coordinate_decay-2} and Lemma \ref{lem_Thom} we can prove that
 $$
 |(II)_{\mathbf a}|\to 0 ,\;\;\;\;\mbox{ as } d\to 0.
 $$
Similarly, by Theorem \ref{coordinate_decay-2}, Lemma \ref{lem_omega} and smoothness of $b_{ij}$ we have
 \begin{align*}
\left|(f_{1}(\mathfrak{y},\hat{\fs},\fkz)-f_{1}(0,\hat{\fs},\fkz)\right|\leq C (e^{-\fc r}+\hat \delta)\to 0,\;\;\;\;\mbox{ as } d\to 0.
\end{align*}
Using \eqref{def_Gamma} and \eqref{Thom_form_N} we get
 $$
 |(I)_{\mathbf a}|\to 0 ,\;\;\;\;\mbox{ as } d\to 0.
 $$
 Then for $1\leq \mathbf{a}\leq \mathbf{n}_{t}$
\begin{equation}\label{eqn_est_Bi}
|(B_{\mathbf{a}})|\to 0 ,\;\;\;\;\mbox{ as }d\to 0.
\end{equation}
For $ \mathbf{a}> \mathbf{n}_{t}$,   we     choose  $( \ft_{o}, \fs,  \ft,\fkz )$ and  $(\mathfrak{y},\hat\fs,\hat \ft,\fkz )$ as local coordinates of $ W_{\mathbf{a}}$, where $(\hat \fs,\hat \ft )= (\pi^*\fs,\pi^*\ft ).$ By the similar argument above we have
\eqref{eqn_est_Bi} also holds.
 \v
 By Lemma \ref{lem_sp_9.3} $\Psi_{(A,g,n)}(\theta_{!}(K);\{\alpha _i\})$ is independent of the choice of $im(\theta)^{*}.$ Hence
$$
\left|\Psi_{(A,g,n)}(\theta_{!}(K);\{\alpha _i\})
-\Psi_{(A, g_1,g_2,n_1,n_2)}(K; \{\alpha _i\})\right|=|(A)-(B)|\to 0,\;\;\;\;\mbox{ as }d\to 0.
$$
Then the lemma is proved.
$\Box$

\v

\subsection{\bf Splitting axiom}\label{splitting axiom}

Denote by $\Delta\subset M\times M$ the diagonal.
Let $\pi:N\to \Delta$ be the normal bundle in $M\times M$, and let $\Phi$ be a Thom form on $N$.  There is a natural map
$$
ev^c_{n_{1}+1,n_{2}+1}: \mathbf{U}^T_{1c}\times \mathbf{U}^T_{2c}\to M\times M
$$
defined by
$$
ev^c_{n_{1}+1,n_{2}+1}(b_{1},b_{2})=(ev^{1c}_{n_{1}+1}(b_{1}),ev^{2c}_{n_{2}+1}(b_{2})).
$$
Then $(ev^c_{n_{1}+1,n_{2}+1})^*N$ is a vector bundle on
$ \mathbf{U}^T_{c}$. Set $\Phi_c=(ev^c_{n_{1}+1,n_{2}+1})^*\Phi.$ Then for any differential form $\alpha\in \mathbf{U}^T$ with exponential decay on $\p \mathbf{U}^T$ we have
$$\int_{\mathbf{U}^T_c}\alpha =\int_{\mathbf{U}^T_{1c}\times \mathbf{U}^T_{2c}} \Phi_c\wedge \pi^*\alpha.$$
Choose a homogeneous basis $\{\beta _b\}_{1\le b\le L}$ of $H^*(M,\mathbb{R})$. Let $(\eta _{ab})$ be its
intersection matrix. Note that
$\eta _{ab} = \beta _a \cdot \beta _b =0$ if the dimensions of
$\beta _a$ and $\beta _b$ are not complementary to each other.
Put $(\eta ^{ab})$ to be the inverse of $(\eta _{ab})$.  Then, the Poincare dual of $\Delta$ is
$$\Delta^*=\sum_{a,b} \eta^{ab} \beta_a\otimes \beta_b.$$
There is a smooth form $\sigma\in C^{\infty}(M\times M)$ such that
$$\Phi-\Delta^*=d\sigma.$$
Then
\begin{equation}\label{Poicare}
\int_{\mathbf{U}^T_c} \alpha =\int_{\mathbf{U}^T_{1c}\times \mathbf{U}^T_{2c}}(ev^c_{n_{1}+1,n_{2}+1})^* (\Delta^* + d\sigma)\wedge \alpha
=\int_{\mathbf{U}^T_{1c}\times \mathbf{U}^T_{2c}}(ev^c_{n_{1}+1,n_{2}+1})^* \Delta^*\wedge \alpha.\end{equation}
In the last equality we used the Stokes theorem and the same argument as in the proof of (2) in Theorem \ref{Common Properties}.

 \begin{lemma}\label{split-1} Let  $K_{1}\times K_{2}\in H^{*}(\overline{\mathcal M}_{g_{1},g_{2},n_{1},n_{2}},\mathbb R)$, $\alpha _1,\cdots,\alpha _n\in H^*(M,\mathbb{R})$ be represented by smooth forms. Then
\begin{align*}
&\Psi_{(A_1,A_2,g_1,g_2,n_1,n_2)}(K_1\times K_2; \{\alpha_i\})\\
&=\epsilon(K,\alpha)
 \sum_{a,b}\eta^{ab}\Psi_{(A_1,g_1,n_1+1)}(K_1; \{\alpha_i\}_
{i\leq n_1}, \beta_a)\Psi_{(A_2, g_2, n_2+1)}(K_2; \{\alpha_j\}_{j>n_1},
\beta_b),
\end{align*}
where $\epsilon(K,\alpha)=(-1)^{deg(K_2)\sum^{n_1}_{i=1} (deg (\alpha_i))}$.
\end{lemma}
\begin{proof}  Let $ {K}=( {K}_{1}, {K}_{2})$ be the smooth form as in subsection \S\ref{const_virt_neig}.
Let $\Theta_{ic}$ be the Thom form of $\mathbf E_{ic}$ supported in a neighborhood of the zero section. Then $\Theta_{c}=\Theta_{1c}\wedge \Theta_{2c}$ be the Thom form of $\mathbf E_{c}$. Using \eqref{Poicare}, a direct calculation gives us
$$\begin{array}{lll}
&&\Psi_{(A_1,A_2,g_1,g_2,n_1,n_2)}(K_1\times K_2; \{\alpha_i\})\\
&=&\int_{\mathbf U^T_{c,\epsilon}} \mathscr{P}^*(K)\wedge (\prod_i ev_{i}^*\alpha_i)
\wedge \sigma_{c}^{*}\Theta_c\\
&=&\sum_{C\in \mathcal C}\int_{\mathbf U^T_{1c,\epsilon}\times \mathbf U^T_{2c,\epsilon}}
(ev^c_{n_{1}+1,n_{2}+1})^*(\Delta^*)\wedge \mathscr{P}^*({K}_1\times {K}_2)
 \wedge (\prod_i ev_{i}^*\alpha_i)
\wedge \sigma^{*}_{1c} \Theta_{ 1c} \wedge  \sigma_{2c}^*\Theta_{ 2c} \\
&=&\sum_{a,b}\eta^{ab}\int_{U^T_{1c,\epsilon}\times \mathbf U^T_{2c,\epsilon}}
 (ev^c_{n_{1}+1,n_{2}+1})^*(\beta_a \wedge  \beta_b) \wedge \mathscr{P}^* ({K}_1\times {K}_2)\wedge \prod_i ev_i^*\alpha_i
\wedge \sigma^{*}_{1c} \Theta_{ 1c} \wedge  \sigma_{2c}^*\Theta_{ 2c}\\
&=&\epsilon(K,\alpha)\sum_{a,b}\eta^{ab} \left[\int_{U^T_{1c,\epsilon}}
(ev^c_{n_{1}+1})^*\beta_a  \wedge \mathscr{P}^* {K}_1\wedge \prod_{i\leq n_{1}} ev_i^*\alpha_i\wedge \sigma^{*}_{1c} \Theta_{ 1c} \right] \\ &&\quad \cdot\left[\int_{U^T_{2c,\epsilon}} (ev^c_{n_{2}+1})^* \beta_b \wedge \mathscr{P}^*  {K}_2\wedge \prod_{i>n_{1}} ev_i^*\alpha_i\wedge \sigma^{*}_{2c} \Theta_{ 2c} \right]\\
&=&\epsilon(K,\alpha)\sum_{a,b}\eta^{ab}\Psi_{(A_1,g_1,n_1+1)}(K_1; \{\alpha_i\}_
{i\leq n_1}, \beta_a)\Psi_{(A_2, g_2, n_2+1)}(K_2; \{\alpha_j\}_{j>n_1},
\beta_b),
\end{array}.$$
The lemma is proved.
\end{proof}
\v\n

\v
Combination of Lemmas \ref{split-1} and \ref{split-2} give us
\begin{theorem}\label{split-3}
 For any $K_1\times K_2 \in H^*(\overline{\mathcal{M}}_{g_1,g_2,n_1,n_2}, \mathbb{R})$,
$\alpha _1,\cdots,\alpha _n \in H^*(M,\mathbb{R})$, represented by smooth forms, we have
$$\Psi_{(A,g,n)}((\theta)_{!}(K_1\times K_2));\{\alpha _i\})$$$$
=\epsilon(K,\alpha) \sum \limits _{A=A_1+A_2} \sum \limits_{a,b}
\Psi_{(A_1,g_1,n_1+1)}(K_1;\{\alpha _{i}\}_{i\le n_1}, \beta _a)
\eta ^{ab}
\Psi_{(A_2,g_2,n_2+1)}(K_2;\beta _b,
\{\alpha _{j}\}_{j>n_1}),
$$
\end{theorem}
where $\epsilon(K,\alpha)=(-1)^{deg(K_2)\sum^{n_1}_{i=1} (deg (\alpha_i))}$.

\end{document}